\documentclass[a4paper,11pt,reqno]{amsart}

\usepackage{setspace}
\usepackage[margin=4.1cm]{geometry}

\usepackage{amsmath, amsthm, amssymb}
\usepackage{mathtools}      
\usepackage{mathabx}        
\usepackage[bb=fourier,cal=euler,scr=rsfs]{mathalfa}	
\usepackage[shortlabels]{enumitem}       

\usepackage[utf8]{inputenc}

\usepackage{tikz-cd}

\usepackage{comment}

\usepackage{transparent}

\newtheorem*{thm*}{Theorem}

\newtheorem{thm}{Theorem}[section]

\newtheorem{quest}{Question}
\newtheorem{cor}[thm]{Corollary}
\newtheorem{af}[thm]{Claim}
\newtheorem{lemma}[thm]{Lemma}
\newtheorem{prop}[thm]{Proposition}

\newtheorem{teoA}{Theorem}

\newtheorem{propB}[teoA]{Proposition}

\newtheorem{teoAprima}{Theorem}

\newtheorem{teoBprima}{Theorem}

\newtheorem{propCprima}{Proposition}

\theoremstyle{definition}
\newtheorem{definition}[thm]{Definition}

\newtheorem{remark}[thm]{Remark}

\newtheorem{example}[thm]{Example}

\newcommand{\R}{\mathbb{R}}
\newcommand{\W}{\mathcal{W}}
\newcommand{\WW}{\tilde{\mathcal{W}}}
\newcommand{\F}{\mathcal{F}}

\newcommand{\U}{\mathcal{U}}
\newcommand{\C}{\mathcal{C}}

\DeclareMathOperator{\Id}{Id}
\DeclareMathOperator{\D}{D}
\DeclareMathOperator{\id}{id}
\DeclareMathOperator{\intt}{int}

\DeclareMathOperator{\length}{length}
\DeclareMathOperator{\graph}{graph}

\DeclareMathOperator{\PH}{PH}
\DeclareMathOperator{\Diff}{Diff}
\DeclareMathOperator{\Fix}{Fix}

\usepackage{accents}

\usepackage{xcolor}
\usepackage{hyperref}
\hypersetup{
    colorlinks,
    linkcolor={red!50!black},
    citecolor={blue!50!black},
    urlcolor={blue!80!black}
}

\usepackage{quoting,xparse}

\newenvironment{itquote}
  {\begin{quote}\itshape}
  {\end{quote}\ignorespacesafterend}

\numberwithin{equation}{section}

\begin{document}

\author[Santiago Martinchich]{Santiago Martinchich} 
\address{CMAT, Universidad de la Rep\'ublica, 11400, Montevideo, Uruguay \vspace{-0.20cm}}

\address{LMO, Université Paris-Saclay, 91405, Orsay, France \vspace{0.10cm}}

\email{smartinchich@cmat.edu.uy}

\title[Global stability of discretized Anosov flows]{Global stability of discretized Anosov flows}

\thanks{S. M. was partially supported by `ADI 2019' project funded by IDEX Paris-Saclay ANR-11-IDEX-0003-02, ERC project 692925
NUHGD and CSIC No 618 `Sistemas Dinámicos'.}


\begin{abstract}
The goal of this article is to establish several general properties of a somewhat large class of partially hyperbolic diffeomorphisms called \emph{discretized Anosov flows}. A general definition for these systems is presented and is proven to be equivalent with the definition introduced in \cite{BFFP}, as well as with the notion of flow type partially hyperbolic diffeomorphisms introduced in \cite{BFT}.

The set of discretized Anosov flows is shown to be $C^1$-open and closed inside the set of partially hyperbolic diffeomorphisms. Every discretized Anosov flow is proven to be dynamically coherent and plaque expansive.
Unique integrability of the center bundle is shown to happen for whole connected components, notably the ones containing the time 1 map of an Anosov flow. For general connected components, a result on uniqueness of invariant foliation is obtained.

Similar results are seen to happen for partially hyperbolic systems admitting a uniformly compact center foliation extending the studies initiated in \cite{BoBo}.
\end{abstract}

\maketitle

\section{Introduction}


A diffeomorphism $f:M\to M$ on a closed manifold $M$ is called \emph{partially hyperbolic} if there exists a $Df$-invariant continuous decomposition $$TM=E^s \oplus E^c \oplus E^u$$ such that vectors in $E^s$ and $E^u$ are uniformly contracted by forward and backward iterates of $f$, respectively, and vectors in $E^c$ experience an intermediate contraction. See Section \ref{sectionprelims} for a precise definition.

The classical examples of partially hyperbolic diffeomorphisms (at least in dimension 3) are:
\begin{itemize}
\item \emph{Deformations of Anosov diffeomorphisms.} 
\item \emph{Partially hyperbolic skew-products.} 
\item \emph{Perturbations of time one maps of Anosov flows.} 
\end{itemize}

The class of \emph{discretized Anosov flows} is conceived as a natural generalization of the lattermost type of examples. The notion of a discretized Anosov flow and related ones have appeared recently and independently in different works. The aim of this text is to establish some general properties for this class of systems in any dimension. In particular, we show the equivalence among many of these notions that are a priori defined in different ways.

\subsection{Discretized Anosov flows}

\begin{definition}\label{defDAFintro} 
We say that $f\in \PH_{c=1}(M)$ is a \emph{discretized Anosov flow} if there exist a continuous flow $\varphi_t:M\to M$, with $\frac{\partial \varphi_t}{\partial t}|_{t=0}$ a continuous vector field without singularities, and a continuous function $\tau:M\to \mathbb{R}$ such that
$$f(x)=\varphi_{\tau(x)}(x)$$ for every $x\in M$.
\end{definition}

The prototypical example of a discretized Anosov flow is the time 1 map of an Anosov flow and all its sufficiently small $C^1$-perturbations. The latter is a consequence of \cite{HPS} and will be revisited in this text (see Theorem \ref{thmA} and Section \ref{sectionproofthmA}). 

The term \emph{discretized Anosov flow} derives from the fact that the flow $\varphi_t$ needs to be a \emph{topological Anosov flow} (see Definition \ref{deftopAnosov} and Proposition \ref{propsTopAF}). Hence $f$ can be regarded as a discretization of the topological Anosov flow $\varphi_t$.

The orbits of $\varphi_t$  form a \emph{center foliation} that is tangent to $E^c$ (see Proposition \ref{prop1DAFs}). An alternative definition for a discretized Anosov flow is the following: $f$ is a partially hyperbolic diffeomorphism admitting a one-dimensional center foliation $\W^c$ such that $f(x)\in \W^c_L(x)$ for every $x\in M$, where $L>0$ is some global constant and $\W^c_L(x)$ denotes the ball of center $x$ and radius $L$ inside the leaf $\W^c(x)$. The equivalence with Definition \ref{defDAFintro} is shown in Propostion \ref{propcenterfixingL}.

The notion of discretized Anosov flow studied in \cite{BFFP}, \cite{BFP}, \cite{BaGo} and \cite{GM}, the class of systems considered in  \cite[Theorem 2]{BW} and \cite{BG1}, and the notion of \emph{flow-type partially hyperbolic diffeomorphism} studied in \cite{BFT} are shown to be all equivalent to each other and equivalent with the definition of discretized Anosov given in Definition \ref{defDAFintro}. This is done in Section \ref{sectionequivalences}.

\subsection{\texorpdfstring{$C^1$}{C1}-openness and closedness} 

We denote by $\PH(M)$  the set of partially hyperbolic diffeomorphisms in $M$ and by $\PH_{c=1}(M)$ the ones such that $\dim(E^c)=1$. In the $C^1$ topology, the sets $\PH(M)$ and $\PH_{c=1}(M)$ are open subsets of the space of $C^1$ diffeomorphisms as a consequence of the cone criterion (see for example \cite{CP}).

We show that discretized Anosov flows constitute a somewhat large class of partially hyperbolic diffeomorphisms with one-dimensional center:

\begin{teoA}\label{thmA}
The set of discretized Anosov flows is a $C^1$-open and closed subset of $\PH_{c=1}(M)$.
\end{teoA}

In other words, Theorem \ref{thmA} shows that the class of discretized Anosov flows comprises whole connected components of $\PH_{c=1}(M)$. Many natural questions arise. One may ask, for example, which connected components contain the time 1 map of an Anosov flow, which are the properties that are preserved in whole connected components and which ones are not.

In Example \ref{exnonuniqint} we give an example of a discretized Anosov flow that does not belong to the same connected component as the time 1 map of an Anosov flow (see Corollary \ref{corotherccs} below). This is done via a simple modification of the dynamically coherent examples given in \cite{HHU}. However, Example \ref{exnonuniqint} is still very particular as it is not transitive and the center flow $\varphi_t$ is orbit equivalent to a suspension flow.

Recall that a pair of partially hyperbolic systems and invariant center foliations $(f,\W^c_f)$ and $(g,\W^c_g)$ are said to be \emph{leaf-conjugate} if there exists a homeomorphism $h$ taking leaves of $\W^c_f$ to leaves of $\W^c_g$ such that $g\circ h(L)=h\circ f(L)$ for every leaf $L\in \W^c_f$. 

From the proof of Theorem \ref{thmA} we also obtain:

\begin{cor}\label{corleafconjintro}
Two discretized Anosov flows in the same $C^1$ connected component of $\PH_{c=1}(M)$ are leaf-conjugate.
\end{cor}

A key tool for the proof of Theorem \ref{thmA} is a strengthening of the stability of center foliations from \cite{HPS}:
\begin{thm}\label{thmunifHPSintro}

Suppose $f_0\in \PH_{c=1}(M)$. There exists $\delta>0$ and a $C^1$ neighborhood $\U$ of $f_0$ such that, 
if some $f\in \U$ admits a center foliation $\W^c$ so that $(f,\W^c)$ is $\delta$-plaque expansive, then every $g\in\U$ admits a center foliation $\W^c_g$ such that $(f,\W^c)$ and $(g,\W^c_g)$ are leaf-conjugate. 

\end{thm}

The definition of \emph{$\delta$-plaque expansitivy} is given in Definition \ref{defdeltape}. It is merely a quantitative version of the usual \emph{plaque expansivity} property, though sensible to the metric one fixes in $M$.

The classical result from Hirsch-Pugh-Shub  \cite{HPS} constructs a stability neighborhood around every plaque expansive system but gives no a priori control over the size of these neighborhoods. For $f_n$ a sequence of plaque expansive systems converging to some $f_0$ there is no a priori reason for $f_0$ to lie in the stability neighborhood of some $f_n$. From Theorem \ref{thmunifHPSintro} one can ensure this and induce a center foliation on $f_0$ provided the sequence $(f_n,\W^c_{f_n})$ is $\delta$-plaque expansive.

A more complete version of Theorem \ref{thmunifHPSintro} is given in Theorem \ref{thmgraphtransf}. Its proof is the main goal of Section \ref{appendixA} and comprises a considerable portion of this text. We point out that Section \ref{appendixA} may be of independent interest.

It is worth noting that, in a different but
 related context, a statement similar to Theorem \ref{thmunifHPSintro}
has been observed in \cite{BFP} (see \cite[Theorem 4.1 and Theorem 4.2]{BFP}). Moreover, in dimension 3 the statement of Theorem \ref{thmA} essentially follows from \cite{BFP} if one combines \cite[Proposition 5.25, Proposition 5.26 and Theorem C]{BFP} once the definition of discretized Anosov flow given in \cite{BFP} is shown to be equivalent with Definition \ref{defDAFintro} (see Corollary \ref{corequivDAFS}).

\subsection{Dynamical coherence and uniqueness of invariant foliations}

Recall that a partially hyperbolic diffeomorphism $f:M\to M$ is said to be \emph{dynamically coherent} if the bundles $E^s\oplus E^c$ and $E^c\oplus E^u$ integrate to $f$-invariant foliations $\W^{cs}$ and $\W^{cu}$, respectively. If this is the case, then $\W^c=\W^{cs}\cap \W^{cu}$ is an $f$-invariant foliation tangent to $E^c$.

Let us call a leaf of a foliation of dimension $d>0$  a \emph{plane} if it is homeomorphic to $\R^d$, and a \emph{cylinder} if it is homeomorphic to a fiber bundle over the circle whose fibers are homeomorphic to $\R^{d-1}$.

\begin{teoA}\label{thmB} 
Let $f$ be a discretized Anosov flow. Let $\varphi_t$ be the flow appearing in the definition of $f$ and $\W^c$ be the center foliation whose leaves are the orbits of $\varphi_t$. Then:

\begin{enumerate}
\item\label{thm2i}(Topological Anosov flow). The flow $\varphi_t$ is a topological Anosov flow (see Definition \ref{deftopAnosov} below).
\item\label{thm2ii}(Dynamical coherence). The map $f$ is dynamically coherent. Moreover, it admits a center-stable foliation $\W^{cs}$ and a center-unstable foliation $\W^{cu}$ such that $\W^c=\W^{cs}\cap \W^{cu}$.
\item\label{thm2iii}(Uniqueness of foliations).  The foliations $\W^{cs}$ and $\W^{cu}$ are the only $f$-invariant foliations tangent to $E^s\oplus E^c$ and $E^c\oplus E^u$, respectively.
\item\label{thm2iv}(Completeness of leaves). The leaves of $\W^{cs}$ and $\W^{cu}$ satisfy that $\W^{cs}(x)=\bigcup_{y\in \W^c(x)}\W^s(y)$ and $\W^{cu}(x)=\bigcup_{y\in \W^c(x)}\W^u(y)$ for every $x\in M$.
\item\label{thm2v}(Topology of leaves)  The leaves of $\W^{cs}$ and $\W^{cu}$ are homeomorphic to either planes or cylinders. The former contain no compact center leaves while the latter contain exactly one.
\end{enumerate}
\end{teoA}

We point out that in dimension 3 the statement of Theorem \ref{thmB} was mostly known. Indeed, once (\ref{thm2ii}) is proven then (\ref{thm2i}), (\ref{thm2iv}) and (\ref{thm2v}) follow from \cite[Theorem 2]{BW}. In addition, once (\ref{thm2i}) is proven then (\ref{thm2ii}) has already appeared in \cite[Proposition G.2]{BFP} and (\ref{thm2iii}) follows from \cite{BFFP} (see \cite[Lemma 1.1]{BaGo}). Note also that (\ref{thm2iv}) has already appeared in \cite{GM}. It is included here for the sake of completeness.

\subsection{Unique integrability of the center bundle}

Given $f$ in $\PH_{c=1}(M)$ by Peano's existence theorem there exists at least one local $C^1$ curve tangent to $E^c$ through any point $x\in M$. We say that $E^c$ is \emph{uniquely integrable} if such a curve is unique (modulo reparametrizations) for every $x\in M$.

We show that unique integrablity is a property which only depends on the $C^1$ connected component of the space of discretized Anosov flows: 

\begin{propB}\label{propC}
Suppose $f$ is a discretized Anosov flow with $E^c$  uniquely integrable. Then every system in the same $C^1$ connected component of $f$ in $\PH_{c=1}(M)$ has a uniquely integrable center bundle.
\end{propB}

In particular, every discretized Anosov flow in the connected component of the time 1 map of an Anosov flow has a uniquely integrable center bundle (see Corollary \ref{cortime1uniqint}). This extends \cite[Remark 10.9]{BFP} to any dimension.

In the already mentioned Example \ref{exnonuniqint} we give an example of a discretized Anosov flow $f(x)=\varphi_{\tau(x)}(x)$ such that $E^c$ is not uniquely integrable. The center flow $\varphi_t$ is orbit equivalent to the suspension of a linear Anosov diffeomorphism $A:\mathbb{T}^2\to \mathbb{T}^2$ on the 2-torus, yet by Proposition \ref{propC} the map $f$ is not in the same connected component as the time 1 map of the suspension of $A$. We conclude the following.

\begin{cor}\label{corotherccs} There exists connected components of discretized Anosov flows that do not contain the time 1 map of an Anosov flow.
\end{cor}

\subsection{Uniformly compact center and quasi-isometric center behavior}

Recall that the center foliation $\W^c$ of a partially hyperbolic system is said to be \emph{uniformly compact} if the leaves of $\W^c$ are compact and their volume is uniformly bounded on $M$. In particular, this includes \emph{partially hyperbolic skew-products} where the center foliation induces a fiber bundle structure on $M$.

An analogous statement to Theorem \ref{thmA} is satisfied for this class of systems provided the center dimension is one. As far as we are aware this was not stated elsewhere:

\begin{teoAprima}\label{thmAprima}
The systems in $\PH_{c=1}(M)$ admitting an invariant uniformly compact center foliation form a $C^1$ open and closed subset of $\PH_{c=1}(M)$.
\end{teoAprima}

Again, Theorem \ref{thmAprima} shows that the systems in $\PH_{c=1}(M)$ admitting an invariant uniformly compact center foliation comprise whole connected components of $\PH_{c=1}(M)$. Moreover, two maps in the same connected component need also be leaf-conjugate (see Corollary \ref{corleafconjunifcompact}).

We say that a partially hyperbolic diffeomorphism \emph{acts quasi-isometrically} on $\W^c$ if there exist constants $l,L>0$ such that $f^n(\W^c_l(x))\subset \W^c_L(f^n(x))$
for every $x\in M$ and $n\in \mathbb{Z}$. This property is satisfied by discretized Anosov flows as well as by systems admitting a uniformly compact center foliation (see Remark \ref{rmkDAFQI} and Remark \ref{rmkunifcompactQI}). Some parts of Theorem \ref{thmB} extend to systems acting quasi-isometrically on a center foliation.

\begin{teoBprima}\label{thmB'} Suppose $f\in \PH_{c=1}(M)$ acts quasi-isometrically on an $f$-invariant center foliation $\W^c$. Then:
\begin{enumerate}
\item\label{thm2'i}(Dynamical coherence). The map $f$ is dynamically coherent. Moreover, it admits a center-stable foliation $\W^{cs}$ and a center-unstable foliation $\W^{cu}$ such that $\W^c=\W^{cs}\cap \W^{cu}$.
\item\label{thm2'ii}(Uniqueness of foliations).  The foliations $\W^{cs}$ and $\W^{cu}$ are the only $f$-invariant foliations tangent to $E^s\oplus E^c$ and $E^c\oplus E^u$, respectively.
\item\label{thm2'iii}(Completeness of leaves). The leaves of $\W^{cs}$ and $\W^{cu}$ satisfy that $\W^{cs}(x)=\bigcup_{y\in \W^c(x)}\W^s(y)$ and $\W^{cu}(x)=\bigcup_{y\in \W^c(x)}\W^u(y)$ for every $x\in M$.
\end{enumerate}
\end{teoBprima}

Note that item (\ref{thm2'i}) for systems admitting a uniformly compact center (and for any center dimension) has been proven in \cite[Theorem 1]{BoBo} .

Item (\ref{thm2'ii}) in Theorem \ref{thmB'} shows that $\W^c$ is the only $f$-invariant center foliation where $f$ acts quasi-isometrically. For uniformly compact center foliations this gives a partial answer to \cite[Question 8.4.]{BoBo} (the general question is for any center dimension).

An analogous result to Proposition \ref{propC} is also verified in this context:

\begin{propCprima}\label{propC'}
Suppose $f\in \PH_{c=1}(M)$ admits a uniformly compact center foliation such that $E^c$ is uniquely integrable. Then every systems in the same $C^1$ connected component of $f$ in $\PH_{c=1}(M)$ has a uniquely integrable center bundle.
\end{propCprima}

In particular, Proposition \ref{propC'} shows that
if $f=A\times \Id$ is the product of an Anosov diffeomorphism $A:N\to N$ and the identity map $\Id:S^1\to S^1$, then every system in the same $C^1$ connected component of $\PH_{c=1}(N\times S^1)$ than $f$ has a uniquely integrable center bundle.

\subsection{Diffeomorphisms that are center fixing or admit a compact center foliations in dimension 3} 

This text focuses on partially hyperbolic diffeomorphisms on any ambient dimension. However, particularly on dimension 3 we are able to conclude a classification result for transitive systems modulo conditions on a center foliation:

\begin{teoA}\label{thmD}
Suppose $f\in \PH_{c=1}(M^3)$ is transitive and admits an $f$-invariant center foliation $\W^c$.
\begin{enumerate}
\item\label{item1propcenterfixingdim3} If $f(W)=W$ for every $W\in \W^c$ then $f$ is a discretized Anosov flow.
\item\label{item2propcenterfixingdim3} If $W$ is compact for every $W\in \W^c$ then, modulo double cover, $f$ is a partially hyperbolic skew-product.
\end{enumerate}
\end{teoA}

The main new step in order to prove Theorem \ref{thmD} is to show that $f$ is dynamically coherent (see Proposition \ref{propdim3dyncoh}). Then the rest follows from previous results on the matter. Namely, \cite[Theorem 2]{BW} for showing item (\ref{item1propcenterfixingdim3}), and \cite{DMM} and \cite{Boh} for showing item (\ref{item2propcenterfixingdim3}).

\subsection{Some words on `global stability'}

We may see Theorem \ref{thmA} and Theorem \ref{thmAprima} as a `global stability' result where a plaque expansive center system induces leaf-conjugacy among its whole $C^1$ partially hyperbolic connected component. 

This has also been shown to be true in \cite{FPS} whenever $f$ is a hyperbolic linear automorphism of the torus $\mathbb{T}^n$ (seen as a partially hyperbolic diffeomorphism), and generalized in \cite{Pi} for linear Anosov automorphisms on nilmanifolds.

We may ask if this is true in general:

\begin{quest}\label{questintro} Suppose $f\in \PH(M)$ admits an $f$-invariant center foliation $\W^c$ such that $(f,\W^c)$ is plaque expansive. Does every $g$ in the $C^1$ partially hyperbolic connected component of $f$ admits a $g$-invariant center foliation $\W^c_g$ such that $(g,\W^c_g)$ is plaque expansive and leaf-conjugate to $(f,\W^c)$?
\end{quest}

Particular examples where an answer to Question \ref{questintro} is unknown (as far as we are aware)  are $f=\varphi_1\times \Id$ and  $f=\varphi_1\times \varphi_1$ for $\varphi_t$ an Anosov flow. To the best of our knowledge, Question \ref{questintro} has a positive answer for every known system $(f,\W^c)$  such that $\dim(E^c)=1$.

Recall that it is not known if every system admitting a center foliation needs to be plaque expansive. This is commonly referred to as the \emph{plaque expansivity conjecture}. Question \ref{questintro} is of course related to this conjecture but may a priori be a different problem. On the one hand, a positive answer to Question \ref{questintro} may not rule out the existence of center foliations that are not plaque expansive. On the other hand, if the plaque expansivity conjecture happens to be true then the systems that are leaf-conjugate to a given plaque expansive system $f$ may a priori form a $C^1$-open set that fails to be the whole $C^1$ partially hyperbolic connected component containing $f$.

\section{Preliminaries}\label{sectionprelims}

\vspace{0.15cm}
\noindent 
\emph{ \bf Partially hyperbolic diffeomorphisms.} A $C^1$-diffeomorphism $f:M\to M$ in a closed Riemannian manifold $M$  is called \emph{partially hyperbolic} if it preserves a continuous splitting $TM=E^s\oplus E^c \oplus E^u$, with non-trivial \emph{stable} $E^s$ and \emph{unstable} $E^u$ bundles, such that for some positive integer $\ell>0$ it satisfies
\begin{center}$\|Df^\ell_xv^s\|< \frac{1}{2}\|v^s\|$, \hspace{0.2cm} $\|Df^{-\ell}_xv^u\|< \frac{1}{2} \|v^u\|$ \hspace{0.2cm} and

$\|Df^{\ell}_xv^s\|<\|Df^{\ell}_xv^c\|< \|Df^{\ell}_xv^u\|$
\end{center}
for every $x\in M$ and unit vectors $v^\sigma\in E^\sigma(x)$ for $\sigma\in\{s,c,u\}$. Modulo changing the constant $\ell>0$, the property of being partially hyperbolic is independent of the Riemannian metric in $M$.

\vspace{0.15cm}
\noindent \emph{\bf Invariant manifolds.} If $f$ is a partially hyperbolic diffeomorphism it is known since \cite{HPS} that the bundles $E^s$ and $E^u$ uniquely integrate to $f$-invariant foliations. We denote them as $\W^s$ and $\W^u$, respectively. The bundles $E^s\oplus E^c$ and $E^c\oplus E^u$ may or may not be integrable. Whenever they integrate to  $f$-invariant foliations ($\W^{cs}$ and $\W^{cu}$, respectively) we say that $f$ is \emph{dynamically coherent}. If this is the case then $\W^c=\W^{cs}\cap \W^{cu}$ is an $f$-invariant foliation whose leaves are tangent to $E^c$. 

Notations: Whenever a foliation $\W^\sigma$ tangent to $E^\sigma$ is well defined for $\sigma\in\{s,c,u,cs,cu\}$ we will denote by $\W^\sigma_\delta(x)$ the ball of radius $\delta>0$ and center $x$ inside the leaf $\W^\sigma(x)$ with respect to the intrinsic metric induced by the Riemannian metric in $M$. In this context, if $A$ is any subset of $M$ we will denote by $\W^\sigma(A)$ the saturation of $A$ by $\W^\sigma$-leaves, that is, the set $\bigcup_{y\in A}\W^\sigma(y)$. We will also denote by $\W^\sigma_\delta(A)$ the set $\bigcup_{y\in A}\W^\sigma_\delta(y)$.

Every invariant foliation $\W^\sigma$ for $\sigma\in\{s,c,u,cs,cu\}$ has $C^1$ leaves that are tangent to the continuous bundle $E^\sigma$. All the foliations considered on this article will have, by definition, $C^1$ leaves that are tangent to a continuous subbundle of $TM$. Foliations with this type of regularity have the following easy to check property that we will implicitly use several times along the text: For every $R>0$ and $\epsilon>0$ there exists $\delta>0$ such that if $d(x,y)<\delta$ then $d_H(\W^\sigma_R(x),\W^\sigma_R(y))<\epsilon$, where $d_H$ denotes the Hausdorff distance among subsets of $M$.

\vspace{0.15cm}
\noindent 
\emph{\bf Invariant cone fields.}
We say that $\C$ is a continuous cone field in the Riemannian manifold $M$ if there exists a continuous splitting $TM=E \oplus F$ such that for every $x\in M$ the cone $\C(x)\subset T_xM$ is given by $\C(x)=\{v=v_E+v_F\in T_xM:\|v_E\|_E\geq \|v_F\|_F\}$ 
for some continuous norms $\|\cdot\|_E$ and $\|\cdot\|_F$ in $E$ and $F$, respectively (not necessarily the ones induced by the underlying Riemannian metric). In this context we say that $\C$ has dimension $\dim(E)$. We define the interior of the cone by $\intt \C(x)=\{v=v_E+v_F\in T_xM:\|v_E\|_E> \|v_F\|_F\}\cup \{0\}$ for every $x\in M$. 

We say that $\C$ is $f$-invariant if for some $N>0$ one has $Df^N\C(x)\subset \intt \C(f^N(x))$ for every $x\in M$. If this is the case, we say that $\C$ is uniformly expanded by $f$ if $\|f^N(v)\|>\|v\|$ for every $v\in \C\setminus \{0\}$.

If $f:M\to M$ is a partially hyperbolic diffeomorphism one can check  that there exists $\C^u$ and $\C^{cu}$ continuous cone fields of dimension $\dim(E^u)$ and $\dim(E^{cu})$, respectively, that are $f$-invariant and such that $E^u$ is uniformly expanded by $f$ and $E^{u}(x)=\bigcap_{n\geq 0} Df^n(\C^u(f^{-n}(x))$ and $E^{cu}(x)=\bigcap_{n\geq 0} Df^n(\C^{cu}(f^{-n}(x))$ for every $x\in M$. Analogously for $f^{-1}$-invariant cone fields $\C^s$ and $\C^{cs}$.

In fact, the \emph{cone criterion} gives us a kind of reciprocal of the above: A $C^1$ diffeomorphism $f:M\to M$ is partially hyperbolic whenever there exists an $f$-invariant cone field $\C^u$ uniformly expanded by $f$ and a $f^{-1}$-invariant cone field $\C^s$ uniformly expanded by $f^{-1}$. As a consequence, it is immediate to check that $\PH(M)$ is $C^1$ open in $\Diff^1(M)$.  See for example \cite{CP}.

\vspace{0.15cm}
\noindent 
\emph{\bf Anosov flows.} A $C^1$ flow $\varphi_t:M\to M$ (that is, such that $(x,t)\mapsto \varphi_t(x)$ is a $C^1$ map) is called an \emph{Anosov flow} if there exists a continuous $D\varphi_t$-invariant splitting $TM=E^s\oplus E^c \oplus E^u$ such that $E^c$ is the bundle generated by $\frac{\partial \varphi_t}{\partial t}|_{t=0}$ and such that for some $t_0\neq 0$ the map $f=\varphi_{t_0}$ is a partially hyperbolic diffeomorphism with respect to the decomposition $TM=E^s\oplus E^c \oplus E^u$. If $\varphi_t$ is an Anosov flow it is immediate to check that $g=\varphi_{t_1}$ is a partially hyperbolic diffeomorphism for every $t_1\neq 0$.
\vspace{0.15cm}

\noindent 
\emph{\bf Notation.} Throughout the text we will write \emph{disc} to mean `ball of dimension $d$', where $d$ may not necessarily be 2.
\vspace{0.15cm}

The following will be used several times along the text.
\begin{lemma}\label{lemmaWcWsC1}
Suppose $f\in \PH_{c=1}(M)$. There exists $\delta>0$ such that for every $C^1$ arc $\eta$ tangent to $E^c$ with $\length(\eta)<\delta$ the set $\W^s_\delta(\eta)$ is a $C^1$ submanifold tangent to $E^s\oplus E^c$.
\end{lemma}
A proof of the lemma above can be found in \cite[Proposition 3.4.]{BBI} (it is stated for \emph{absolute} partially hyperbolic diffeomorphism but the proof does not use this fact). See also \cite[Theorem 6.1]{HPS} and \cite[Remark 4.7.]{BoBo}.

\section{Discretized Anosov flows}

\subsection{Definition and first properties}

Recall from Definition \ref{defDAFintro} that we call $f\in \PH_{c=1}(M)$ a \emph{discretized Anosov flow} if there exist a continuous flow $\varphi_t:M\to M$, with $\frac{\partial \varphi_t}{\partial t}|_{t=0}$ a continuous vector field without singularities, and a continuous function $\tau:M\to \mathbb{R}$ satisfying
$$f(x)=\varphi_{\tau(x)}(x)$$ for every $x\in M$.

In contrast with the definition given in \cite{BFFP}, we do not ask for $\varphi_t$ to be a topological Anosov flow. This is derived as a consequence in Proposition \ref{propDAFtopAnosov}. Nor do we ask a priori that the orbits of $\varphi_t$ form a center foliation. This follows from the next proposition.

\begin{prop}\label{prop1DAFs} If $f$ is a discretized Anosov flow then:
\begin{enumerate}[label=(\roman*)]
\item\label{item1prop1} The vector field $\frac{\partial \varphi_t}{\partial t}|_{t=0}$ generates the bundle $E^c$ and the flow lines of $\varphi_t$ form a center foliation $\W^c$ whose leaves are fixed by $f$. 
\item\label{item2prop1} The function $\tau$ has no zero and is $C^1$ restricted to each leaf of $\W^c$.
\end{enumerate}
\end{prop}

\begin{proof} 
Let $F$ be the one-dimensional bundle generated by $\frac{\partial \varphi_t}{\partial t}|_{t=0}$. In order to show \ref{item1prop1} let us see that $F=E^c$. This has essentially been done for $\dim(M)=3$ in \cite[Proposition G.2.]{BFFP} and the arguments are equally valid in any dimension. We will briefly reproduce them for the sake of completeness.

We claim first that it is enough to show that $F$ is never contained in $E^s$ nor $E^u$. Indeed, if $F(x)$ is not contained in $E^s(x)$ nor $E^u(x)$ for every $x\in M$ then the angle formed by $F$ and $E^s$ is bounded away from zero by a positive constant independent of the point in $M$. As a consequence, for every $x\in M$ the subspace $Df^n(F(f^{-n}(x)))$ gets arbitrarily close to $E^{cu}(x)$ as $n$ tends to $+\infty$. As $F$ is $Df$-invariant (see justification below) we deduce that $F(x)$ needs to be contained in $E^{cu}(x)$. Arguing analogously for backwards iterates using the never-zero angle between $F$ and $E^u$ one obtains that $F(x)$ has to be contained in $E^{cs}(x)$ for every $x$ in $M$. We conclude that $F$ coincides everywhere with $E^c=E^{cs}\cap E^{cu}$.

The bundle $F$ needs to be $Df$-invariant as every small piece of $\varphi_t$-orbit through a point $x\in M$ is sent by $f$ to a $C^1$ curve that is a reparametrization of a small piece of $\varphi_t$-orbit through $f(x)$. Thus $F(f(x))$ that is generated by $\frac{\partial \varphi_t}{\partial t}|_{t=0}(f(x))$ coincides with $DfF(f(x))$ that is generated by $\frac{\partial f\varphi_t}{\partial t}|_{t=0}(x)$.

It remains to see now that $F$ is never contained in $E^s$ nor $E^u$. Without loss of generality suppose by contradiction that $F(x)$ is contained in $E^u(x)$ for some $x$. Note that $F(f^{-n}(x))$ is then contained in $E^u(f^{-n}(x))$ for every $n\geq 0$.

Let $\mathcal{C}^u$ be a continuous $f$-invariant unstable cone field  such that $Df^N\mathcal{C}^u\subset \intt \mathcal{C}^u$ for some $N>0$ and $\bigcap_{n\geq 0} Df^n(\mathcal{C}^u(f^{-n}y))=E^u(y)$ for every $y\in M$ (see preliminaries for more details). Since for every $n\geq 0$ a piece of $\varphi_t$-orbit containing $f^{-n}(x)$ is tangent to $\mathcal{C}^u$ we obtain in the limit with $n$ that at least a piece $\eta$ of $\varphi_t$-orbit containing $x$ is contained in $\W^u(x)$. 

As $\tau:M\to \mathbb{R}_{>0}$ is continuous it has some positive upper bound so there exists $L>0$ such that every forward iterate of $\eta$ has length less than $L$. This contradicts the fact that $f$ expands uniformly the length of any $C^1$ arc tangent to $E^u$. This end the proof of $F(x)=E^c(x)$ for every $x\in M$.

It follows that the flow lines of $\varphi_t$ are tangent to $E^c$ and consequently they form a center foliation $\W^c$ whose leaves are fixed by $f$. Property \ref{item1prop1} is settled.

Since $f$ is $C^1$ and preserves the bundle $E^c$ it is immediate to check that the function $\tau$ needs to be $C^1$ restricted to each leaf of $\W^c$. In order to end \ref{item2prop1} it remains to show that $\tau$ has no zeros. For this we will use a similar argument as in  \cite[Lema 1.2.]{BG1} or \cite[Proposition 5.14.]{BFP}.

Let us suppose by contradiction that $\tau(x)=0$ for some $x\in M$ and consider $U$ a small $\varphi_t$ flow box neighborhood of $x$. By the continuity of $f$ there exists $\epsilon>0$ such that $B_{2\epsilon}(x)\subset U$ and $f(B_{2\epsilon}(x))\subset U$. 

We claim that $\epsilon$ can be considered small enough so that $y$ and $f(y)$ need to lie in the same segment of $\varphi_t$-orbit of $U$ for every $y\in B_\epsilon(x)$. Indeed, let $l>0$ be a constant smaller than the distance between $B_\epsilon(x)$ and $M\setminus U$ and let $C>0$ be a constant larger than $\|\frac{\partial \varphi_t}{\partial t}|_{t=0}(y)\|$ for every $y\in M$. By the continuity of $\tau$  we can consider $\epsilon$ small enough so that $\tau(y)C<l$ for every $y\in B_\epsilon(x)$. It follows that the center arc $[y,f(y)]_c$ from $y$ to $f(y)$ along $\varphi_t$ needs to have length less that $l$ for every $y\in B_\epsilon(x)$. Hence, $[y,f(y)]_c$ needs to be contained in $U$ for every $y\in B_{\epsilon}(x)$ and this proves the claim.

As $f$ contracts distances inside $\W^s$-leaves for large enough forward iterates there exists $\delta>0$ such that $\W^s_\delta(x)$ and $f^n(\W^s_\delta(x))$ for every $n\geq 0$ are contained in $B_\epsilon(x)$. Moreover, for every $y\in \W^s_\delta(x)\setminus \{x\}$ the sequence $f^n(y)$ tends to $x$. This contradicts the fact that, by the previous claim, every point in $\{f^n(y)\}_{n\geq 0}$ must lie in the same segment of $\varphi_t$-orbit of $U$ than $y$ (which is at positive distance from $x$).
\end{proof}

Note that because of \ref{item2prop1} in the previous proposition one can always assume that $\tau$ is positive (modulo inverting the time of $\varphi_t$ if needed). 

The next remark shows that Definition \ref{defDAFintro} is independent of reparametrizations of the flow $\varphi_t$. In particular, one can always assume that $\varphi_t$ has been parametrized by arc-length. 

\begin{remark}\label{rmkDAFreparam}
Suppose $f$ is a discretized Anosov flow such that $f(x)=\varphi_{\tau(x)}(x)$ for every $x\in M$ as in Definition \ref{defDAFintro}. Let $\alpha:M\to \mathbb{R}_{>0}$ be a continuous function. If $\tilde{\varphi}_t$ is the reparametrization of $\varphi_t$ generated by the continuous vector field $\alpha \frac{\partial \varphi_t}{\partial t}|_{t=0}$ then there exists $\tilde{\tau}:M\to M$ continuous such that $f(x)=\tilde{\varphi}_{\tilde{\tau}(x)}(x)$ for every $x\in M$.
\end{remark}

\begin{proof}
Let $\W^c$ be the foliation by flow lines of $\varphi_t$. As 
$\frac{\partial \varphi_t}{\partial t}|_{t=0}(x)\neq 0$ for every $x\in M$, the vector field $\alpha \frac{\partial \varphi_t}{\partial t}|_{t=0}$ is a continuous vector field without singularities restricted to each leaf of $\W^c$. It follows that it uniquely integrates inside each leaf of the one-dimensional foliation $\W^c$. The flow $\tilde{\varphi}_t:M\to M$ obtained in this way has the same flow lines as $\varphi_t$.

Moreover, there exists $r:M\times \mathbb{R} \to \mathbb{R}$ continuous such that $\varphi_t(x)=\tilde{\varphi}_{r(x,t)}(x)$ for every $x\in M$ and $t\in \mathbb{R}$. Then $\tilde{\tau}(x)=r(x,\tau(x))$ satisfies that $f(x)=\tilde{\varphi}_{\tilde{\tau}(x)}(x)$.
\end{proof}

\subsection{Fixed center foliation and bounded displacement along center}\label{sectionfixedcenterfoliation}

As pointed out in Proposition \ref{prop1DAFs}, an immediate consequence of Definition \ref{defDAFintro} is that discretized Anosov flows fix the leaves of a one dimensional center foliation $\W^c$. We do not know if, conversely, this property is enough for characterizing discretized Anosov flows:

\begin{quest}\label{questWcfixed}
Suppose $f\in \PH_{c=1}(M)$ admits a center foliation $\W^c$ such that $f(W)=W$ for every leaf $W\in \W^c$. Is $f$ a discretized Anosov flow?
\end{quest}

In \cite[Question 1.3.]{Gog} a similar question has been posed. For transitive systems in dimension 3 a positive answer to Question \ref{questWcfixed} is given in Section \ref{sectionfixedcenterindim3}.

For any ambient dimension the following characterization result gives a partial answer to Question \ref{questWcfixed} provided an extra condition is satisfied.

\begin{prop}\label{propcenterfixingL} Suppose $f\in\PH_{c=1}(M)$. The following are equivalent:
\begin{enumerate}[label=(\roman*)]
\item\label{item1prop2} The map $f$ is a discretized Anosov flow.
\item \label{item2prop2}  There exists a center foliation $\W^c$ and a constant $L>0$  such that $f(x)\in \W^c_L(x)$ for every $x\in M$.
\end{enumerate}
\end{prop}

\begin{proof}
Suppose $f$ is a discretized Anosov flow. Let $\varphi_t$ be the flow appearing in the definition of $f$ such that $f(x)=\varphi_{\tau(x)}(x)$ for every $x\in M$. Proposition \ref{prop1DAFs} shows  that $f$ fixes the leaves of the center foliation $\W^c$ given by the flow lines of $\varphi_t$. If $T>0$ denotes an upper bound for $\tau$ and $C>0$ an upper bound for $y \mapsto \|\frac{\partial \varphi_t}{\partial t}|_{t=0}(y)\|$  it follows that $f(x)\in \W^c_{TC}(x)$ for every $x$ in $M$. Thus \ref{item1prop2} implies \ref{item2prop2}.

Let us see that \ref{item2prop2} implies \ref{item1prop2}. Suppose that there exists $L>0$ such that $f(x)\in \W^c_L(x)$ for every $x\in M$. In particular, $f(W)=W$ for every leaf $W\in \W^c$. 

Note first that, by transverse hyperbolicity, every compact leaf of $\W^c$ of length less than $2L$ can not be accumulated by compact leaves of $\W^c$ of length less that $2L$. Then the number of compact leaves of length less than $2L$ needs to be finite. 

Let $U\subset M$ denote the union of leaves of $\W^c$ with length larger or equal to $2L$. For every $x\in U$ let $[x,f(x)]_c$ denote the center segment in $\W^c_L(x)$ joining $x$ with $f(x)$. It is immediate to check that $[x,f(x)]_c$ varies continuously in the Hausdorff topology for every $x$ in $U$.

Essentially the same argument used to show \ref{item2prop1} in Proposition \ref{prop1DAFs} shows that $f$ has no fixed points in $U$: If $x$ is a fixed point of $f$ consider $U_\epsilon(x)\subset U$ a small foliation box neighborhood of $\W^c$ containing $x$ such that $\W^c_{L}(y)\cap U_\epsilon(x)$ has only one connected component for every $y\in U_\epsilon(x)$. For $\delta>0$ small enough, if $y\in \W^s_\delta(x)\setminus\{x\}$ then $f^n(y)\in U_\epsilon(x)$ for every $n\geq 0$ and $\lim_n f^n(y)=x$. However, $f^n(y)\in \W^c_L(f^{n-1}(y))$ and $f^n(y)\in U_\epsilon$ implies that $f^n(y)$ must lie in the center segment $\W^c_L(y)\cap U_\epsilon(x)$ for every $n\geq 0$. This gives us a contradiction with $\lim_n f^n(y)=x$ and ends the proof that $f$ has no fixed points in $U$.

As $f$ has no fixed points in $U$, for every $x\in U$ we can define $X^c(x)$ to be the unit vector in $E^c(x)$ pointing inwards to the segment $[x,f(x)]_c$. As $[x,f(x)]_c$ varies continuously with $x$ in $U$ it follows that $X^c$ is a continuous vector field in $U$.

Let $\varphi_t:U\to U$ be the flow whose orbits are the leaves of $\W^c$ in $U$ and such that $\frac{\partial \varphi_t}{\partial t}|_{t=0}$ is equal to $X^c$. Let us define $\tau(x)$ to be the length of $[x,f(x)]_c$ for every $x$ in $U$. Clearly $f(x)=\varphi_{\tau(x)}(x)$ for every $x\in U$. It remains to see that $X^c$, $\varphi_t$ and $\tau$, which are a priori defined only in $U$, extend well to $M$. That is, that they extend well to the union of compact center leaves of length less than $2L$.

Let $\eta$ be a compact center leaf of length less than $2L$. For every $x\in \eta$ consider $V_x$ a small $\W^c$-box neighborhood containing $x$ so that if $V_x\cap V_y\neq \emptyset$ then $\W^c|_{V_x\cup V_y}$ is orientable. We can suppose that for every $x$ the neighborhood $V_x$ is small enough so that it is disjoint from every other compact center leaf of length less than $2L$. 

Consider $V$ be the neighborhood of $\eta$ that is the union of the elements of $\{V_x\}_{x\in \eta}$.
It follows that $\W^c|_{V}$ is orientable since any orientation given to $\eta$ can be extended to an orientation on each $V_x$ and this orientations coincide in $V_x\cap V_y$ whenever $V_x\cap V_y\neq \emptyset$. Then, as the set $U\cap V$ is connected, it follows that the orientation induced by $X^c$ in $\W^c|_U$ can be extended to $\W^c|_{U\cap V}$. Now that the a priori orientation issue has been ruled out, it follows immediately that $X^c$ and $\varphi_t$ extend continuously to $\eta$.

It remains to extend  $\tau$ continuously to $\eta$ so that $f(x)=\varphi_{\tau(x)}(x)$ for every $x\in \eta$. To this end, for every $x$ in $\eta$ let us denote by $[x,f(x)]_c$ the center segment from $x$ to $f(x)$ such that $X^c(x)$ points inwards in $[x,f(x)]_c$. Note that it may be the case that if $x_n\to x$ with $(x_n)_n\subset U$ then $[x_n,f(x_n)]_c$ `turns around' $\eta$ many times so that $[x_n,f(x_n)]_c$ accumulates in the Hausdorff topology to $\eta$ instead of $[x,f(x)]_c$

However, since $\W^c$ is a continuous foliation tangent to a continuous subbundle there exists $\epsilon>0$ such that if $d(y,x)<\epsilon$ then $\varphi_t(y)$ is in $V_{\varphi_t(x)}$ for every $t\in [0,L]$ and $x\in\eta$. It follows that the `number of turns' (measured, for example, as the number of connected component of $[x_n,f(x_n)]_c\cap V_x$ minus 1) needs to be constant for $x_n$ close enough to $x$. As this integer number varies continuously with $x$ in $\eta$ it has to be a constant $N$ independent of the point $x$.  Hence by defining $\tau$ in $\eta$ as $$\tau(x)=\length [x,f(x)]_c + N\length \eta$$ it follows that $\tau$ extends continuously to $\eta$. 

By doing the above for every center leaf $\eta$ of length less than $2L$ it follows that $\tau$ is well defined and continuous in $M$, and that $$f(x)=\varphi_{\tau(x)}(x)$$ is satisfied for every $x\in M$. This settles $(ii)$ implies $(i)$.
\end{proof}

\subsection{Quasi-isometrical center action and dynamical coherence}\label{sectionqi}

A key property for discretized Anosov flows turns out to be that segments inside $\W^c$ do not get arbitrarily long for past and future iterates of $f$. We will use this fact to show that every discretized Anosov flow is dynamically coherent. 

It is worth noting that this property sets an essential bridge between the class of discretized Anosov flows and that of partially hyperbolic systems admitting a uniformly compact center foliation.

The following definition is valid for any center dimension.

\begin{definition} A partially hyperbolic diffeomorphism $f$ admitting an $f$-invariant center foliation $\W^c$ is said to \emph{act quasi-isometrically on $\W^c$} if there exist constants $l,L>0$ such that \begin{equation*} f^n(\W^c_l(x))\subset \W^c_L(f^n(x))
\end{equation*} 
for every $x$ in $M$ and $n\in \mathbb{Z}$.
\end{definition}

The following is immediate to check.

\begin{remark}\label{rmkDAFQI} Every discretized Anosov flow acts quasi-isometrically on the center foliation $\W^c$ given by the flow lines of the flow $\varphi_t$ as in Definition \ref{defDAFintro}. Indeed, since $f(x)=\varphi_{\tau(x)}(x)$ for every $x\in M$ then $f$ acts quasi-isometrically on $\W^c$ with constants $l=\min \|\frac{\partial \varphi_t}{\partial t}|_{t=0}\|. \min\tau$ and $L=\max \|\frac{\partial \varphi_t}{\partial t}|_{t=0}\|.\max \tau$.
\end{remark}

\begin{remark}\label{rmkunifcompactQI}
Every partially hyperbolic diffeomorphism admitting an invariant uniformly compact center foliation $\W^c$ acts quasi-isometrically on $\W^c$. Indeed, it is enough to show that under these circumstances the diameter of the center leaves is uniformly bounded and then set $L>0$ larger than this bound.

To show that the diameter of $\W^c$-leaves is uniformly bounded one can argue as follows. Let $\delta,\epsilon>0$ be such that for every $x\in M$ the set $\W^c_\delta(x)$ has volume less than $\epsilon$. Suppose by contradiction that there exist center leaves with arbitrarily large diameter. It follows that for every $N>0$ one can find $N$ points in the same center leaf such that any two points are separated more than $2\delta$. Then the volume of the center leaf containing these points is larger that $N\epsilon$. This contradicts the fact that center leaves have a uniformly bounded volume.
\end{remark}

By Remark \ref{rmkDAFQI} the following proposition shows that discretized Anosov flows are dynamically coherent and have \emph{complete} $\W^{cs}$ and $\W^{cu}$ leaves (items (\ref{thm2ii}) and (\ref{thm2iv}) in Theorem  \ref{thmB}, respectively). Moreover, it also shows Theorem  \ref{thmB'} items (\ref{thm2i}) and (\ref{thm2iii}).

\begin{prop}[Dynamical coherence]\label{propdyncoh}
Suppose $f\in \PH_{c=1}(M)$ acts quasi-isometrically on a center foliation $\W^c$. Then $f$ is dynamically coherent, admitting center-stable foliation $\W^{cs}$ and center-unstable foliation $\W^{cu}$ such that $\W^c=\W^{cs}\cap \W^{cu}$. Moreover, $\W^{cs}(x)=\W^s(\W^c(x))$ and $\W^{cu}(x)=\W^u(\W^c(x))$ for every $x\in M$.
\end{prop}

Recall that for every $x\in M$ the set $\W^s(\W^c(x))$ is by definition equal to $\bigcup_{y\in\W^c(x)}\W^s(y)$ and the set $\W^u(\W^c(x))$ is equal to $\bigcup_{y\in\W^c(x)}\W^u(y)$.

Proposition \ref{propdyncoh} will be derived from the following lemma that may be of independent interest.

\begin{lemma}\label{lemmaetainWsWc}
Suppose $f\in\PH_{c=D}(M)$ for some $D>0$ admits an invariant center foliation $\W^c$. Let $y$ be a point in $\W^s(x)$ for some $x\in M$ and suppose $\eta\subset \W^c(y)$ is a $C^1$ curve through $y$ such that $\{\length(f^n\eta)\}_{n\geq 0}$ is bounded. Then $\eta$ is contained in $\W^s(\W^c(x))$.
\end{lemma}

\begin{proof}
Let $x\in M$, $y\in \W^s(x)$ and $\eta:[0,1]\to \W^c(y)$ be a $C^1$ curve with $\eta(0)=y$. Suppose that $\{\length(f^n\eta)\}_{n\geq 0}\}$ is bounded by some constant $L>0$.

Let $\delta>0$ be as in Lemma \ref{lemmaWcWsC1} so that $\W^s_\delta(\W^c_\delta(z))$ is a $C^1$ submanifold tangent to $E^s\oplus E^c$ for every $z\in M$. Recall that the bundles $E^s$, $E^c$ and $E^u$ vary continuously in $M$. By taking $\delta$ small enough we can ensure that for every $z$ and $z'$ in $M$ such that $d(z,z')<\frac{\delta}{2}$ the sets $\W^u_\delta(z')$ and $\W^s_\delta(\W^c_\delta(z))$ intersect, and that this intersection takes place in a unique point.

We claim that there exists a constant $\delta'>0$ such that if $d(z,z')<\delta'$ and $\gamma:[0,1]\to\W^c(z')$ is a curve of length at most $L$ with $\gamma(0)=z'$ then there exists a continuous curve $H^{su}\gamma:[0,1]\to \W^c(z)$ such that $\W^u_\delta(\gamma(t))\cap \W^s_\delta(H^{su}\gamma(t))\neq \emptyset$ for every $t\in [0,1]$ and $H^{su}\gamma(0)\in \W^c_\delta(z)$. Note that if this claim is true then $H^{su}\gamma$ is a particular choice of continuation by center holonomy of $\gamma$ along $\W^c(z)$ that is uniquely determined by the properties $\W^u_\delta(\gamma(t))\cap \W^s_\delta(H^{su}\gamma(t))\neq \emptyset$ for every $t\in [0,1]$ and $H^{su}\gamma(0)\in\W^c_\delta(z)$.

Let us prove the claim. As $\W^c$ is a foliation tangent to a continuous bundle we can consider $\delta'>0$ so that whenever $z$ and $z'$ are points in $M$ satisfying $d(z,z')<\delta'$ and $\gamma:[0,1]\to\W^c(z')$ is a curve of length at most $L$ with $\gamma(0)=z'$, then there exists a continuous curve $H\gamma:[0,1]\to \W^c(z)$ with $H\gamma(0)=z$ and $d(\gamma(t),H\gamma(t))\leq \frac{\delta}{2}$ for every $t\in [0,1]$. 

The curve $H\gamma$ is an auxiliary curve used to define $H^{su}\gamma$. Indeed, we can consider $P^u(t)$ as the intersection point of $\W^u_\delta(\gamma(t))$ and $\W^s_\delta(\W^c_\delta(H\gamma(t)))$ for every $t\in [0,1]$. Then $H^{su}\gamma(t)$ can be defined as the unique point in $\W^c_\delta(H\gamma(t))$ such that $P^u(t)$ is contained in $\W^s_\delta(H^{su}\gamma(t))$. This proves the claim.

Let $N>0$ be such that $d(f^n(x),f^n(y))<\delta'$ for every $n\geq N$. For simplicity, let $\gamma$ denote the curve $f^N\circ \eta$. Then $H^{su}(f^n\circ\gamma)$ is well defined for every $n\geq 0$. Moreover, as $f$ preserves  $\W^s$, $\W^c$ and $\W^u$-leaves, the special choice of $H^{su}$ gives us the following invariance: the curve $H^{su}(f^n\circ\gamma)$ coincides with the curve $f^n \circ H^{su} \gamma$ for every $n\geq 0$.

In particular, $f^n(P^u(t))$ lies in $\W^u_\delta(f^n\circ \gamma (t))$ for every $t\in [0,1]$ and $n\geq 0$. Iterating $n$ times backwards yields that $P^u(t)$ lies in $\W^u_{C(1/2)^{n/\ell} \delta}(\gamma(t))$ for some constants $\ell\in \mathbb{Z}^+$ and $C>0$ given by the partial hyperbolicity of $f$.

It follows that $P^u(t)=\gamma(t)$ for every $t\in [0,1]$. That is, $f^N\circ \eta$ is contained in $\W^s(\W^c(f^N(x))$. Then $\eta$ is contained in $\W^s(\W^c(x))$.
\end{proof}

\begin{proof}[Proof of Proposition \ref{propdyncoh}]
Suppose $f\in \PH_{c=1}(M)$ acts quasi-isometrically on a center foliation $\W^c$.

Given $x\in M$ and $y \in \W^s(\W^c(x))$ let us see first that $\W^c(y)$ is contained in $\W^s(\W^c(x))$. Indeed, as $f$ acts quasi-isometrically on $\W^c$ for every $l>0$ there exists $L>0$ such every $f$-iterate of $\W^c_l(y)$ is bounded in length by $L$. By Lemma \ref{lemmaetainWsWc} it follows that $\W^c_l(y)\subset \W^s(\W^c(x))$. Since this happens for every $l>0$ it follows that $\W^c(y)\subset \W^s(\W^c(x))$.

By Lemma \ref{lemmaWcWsC1} for every $x\in M$ the set $\W^s(\W^c(x))$ is a $C^1$ injectively immersed submanifold tangent to $E^s\oplus E^c$. As $\W^s(\W^c(x))$ is saturated by $\W^s$ and $\W^c$ leaves it follows that its intrinsic metric is complete and that, if $y \in \W^s(\W^c(x))$, then $\W^s(\W^c(y))=\W^s(\W^c(x))$.

Then $\{\W^s(\W^c(x))\}_{x\in M}$ defines a partition of $M$ whose elements are the leaves of an $f$-invariant foliation tangent to $E^s\oplus E^c$ and subfoliated by leaves of $\W^s$ and $\W^c$. Thus a center-stable invariant foliation $\W^{cs}$ whose leaves are \emph{complete} (meaning that $\W^{cs}(x)=\W^s(\W^c(x))$ for every $x\in M$). The same arguments show that the sets $\{\W^u(\W^c(x))\}_{x\in M}$ define an invariant center-unstable foliation with complete leaves.
\end{proof}

Note that acting quasi-isometrically on a center foliation is preserved under finite lifts and finite powers. One can build other examples of quasi-isometrically center actions as follows: 

\begin{example} Let $\varphi_t:M\to M$ be an Anosov flow, $\pi:N\to M$ be a finite cover of $M$ and $\tilde{\varphi_t}:N\to N$ be the lift of $\varphi_t$ to $N$. Note that $\tilde{\varphi_t}$ is also an Anosov flow in $N$. One can define $f:N\to N$ as the composition of the time $1$ map of $\tilde{\varphi_t}$ with a non-trivial deck transformation of order $k>1$. It follows that $f$ is a partially hyperbolic diffeomorphism acting quasi-isometrically on the center (in fact, isometrically) that is not a discretized Anosov flow or a partially hyperbolic skew-product, but such that the power $g=f^k$ is a discretized Anosov flow. 
\end{example}

A construction from \cite{BPP} gives an example of a system $f\in \PH_{c=1}(M^3)$ acting quasi-isometrically on an $f$-invariant center foliation $\W^c$ such that $f^k$ is not a discretized Anosov for every 
$k \neq 0$ nor $\W^c$ is uniformly compact. This is done via a \emph{$h$-transversality surgery} over the time 1 map of a non-transitive Anosov flow. One can easily check from its construction that this example is not transitive.

One more type of examples of partially hyperbolic diffeomorphisms acting quasi-isometrically on a center foliation can be constructed by taking the product $f\times A :M\times N \to M \times N$ of a discretized Anosov flow $f:M\to M$ and an Anosov map $A:N\to N$.

In view of the above known examples of quasi-isometrically center actions we may ask the following:

\begin{quest}\label{questQIisDAForSkewP}
Suppose $f\in \PH_{c=1}(M)$ acts quasi-isometrically on a center foliation $\W^c$. If $\W^c$ is transitive (i.e. has a dense leaf) then does there exists $k\in \mathbb{Z}^+$ such that $f^k$ is a discretized Anosov flow?
\end{quest}

In \cite{Zh} and \cite{BZ} the notions of partially hyperbolic diffeomorphisms that are \emph{neutral along center} and \emph{topologically neutral along center} were introduced. One can check easily that these systems act quasi-isometrically on the center. However, these notions are strictly stronger than quasi-isometrical action on the center as they forbid, for example, the existence of hyperbolic periodic points (see \cite[Remark 2.11.]{GM}).

In \cite{BZ} a positive answer to Question \ref{questQIisDAForSkewP} in dimension 3 is obtained for systems that are (topologically) neutral along center.

\subsection{Uniqueness of invariant foliations}

The goal of this subsection is to show uniqueness of invariant center-stable and center-unstable foliation for discretized Anosov flows (Theorem \ref{thmB} item (\ref{thm2ii})), and more generally for partially hyperbolic systems acting quasi-isometrically on a one-dimensional center foliation (Theorem \ref{thmB'} item (\ref{thm2ii})). 

We will rely on the following lemma. 

\begin{lemma}\label{lemmaetac}
Suppose $f\in \PH(M)$ admits an $f$-invariant center-stable foliation $\W^{cs}$. If $\eta$ is a $C^1$ curve that is not contained in a leaf of $\W^{cs}$ then $\lim_{n\to +\infty}\length(f^n\circ \eta)=\infty$.
\end{lemma}

\begin{proof}
Let $\delta>0$ be a constant as in Lemma \ref{lemmaWcWsC1}. As the invariant bundles vary continuously in $M$ we can suppose that $\delta$ is small enough such that at scale $\delta$ the invariant bundles are `nearly constant' so that for every $0<\delta'\leq \delta$, if $x,y\in M$ satisfy $d(x,y)<\delta'$, then $\W^\sigma_{2\delta'}(x)$ and $\W^\sigma_{2\delta'}(y)$ intersect and the intersection point is unique for every $(\sigma,\sigma')\in\{(cs,u),(cu,s)\}$. More precisely, one can consider, for example, a constant $\delta=\delta(f)>0$ and a metric in $M$ as in Lemma \ref{rmkUdelta}.

Suppose $\eta$ is a $C^1$ curve that is not contained in a leaf of $\W^{cs}$. Let us see that that $\lim_{n\to +\infty}\length(f^n\circ \eta)=\infty$. Note that it is enough to show this for $\length(\eta)<\delta/4$ since otherwise one can divide $\eta$ is finite pieces of length less than $\delta/4$ and argue from there. Then, suppose from now on that $\length(\eta)<\delta/4$. 

Let $x$ be a point in $\eta$. For every $y\in \W^u_{\delta/4}(x)$ let $\D^{cs}(y)$ be the intersection of $\W^{cs}_\delta(y)$ with $\W^u_\delta(\W^{cs}_{\delta/4}(x))$. It follows that $D:=\bigcup_{y\in \W^u_{\delta/4}(x)}\D^{cs}(y)$ is an open subset of $M$ that is subfoliated by $u$-plaques and $cs$-plaques. The latter being the plaques $\{\D^{cs}(y)\}_{y\in \W^u_{\delta/4}(x)}$. Analogously $f^n(D)$ is subfoliated by $u$-plaques and the $cs$-plaques $\{f^n\D^{cs}(y)\}_{y\in \W^u_{\delta/4}(x)}$ for every $n>0$.

Note that, since $\length(\eta)<\delta/4$ then $\eta$ is contained in $D$. Informally, forwards iterates of $f$ will separate indefinitely the $cs$-plaques of $D$. If $\eta$ is not contained in a unique $cs$-plaque this will force the length of $\eta$ to increase indefinitely.

We will work with the intrinsic metric in $D$ and in its forward iterates $\{f^nD\}_{n> 0}$. Given $\D^{cs}(y)$ and $\D^{cs}(y')$ two different $cs$-plaques in $D$ let us denote $d_u(\D^{cs}(y),\D^{cs}(y'))$ the infimum length among all unstable arcs inside $u$-plaques of $D$ joining $\D^{cs}(y)$ and $\D^{cs}(y')$. Analogously for every $f^nD$.

Note that, as backwards iterates of $f$ contract distances uniformly inside $\W^u$-leaves, then for every pair of disjoint $cs$-plaques $\D^{cs}(y)$ and $\D^{cs}(y')$ in $D$ there exists $N>0$ such that $d_u(f^n\D^{cs}(y),f^n\D^{cs}(y'))>2\delta$ for every $n\geq N$.

Moreover, we claim that if for some $n>0$ one has that the distance $d_u(f^n\D^{cs}(y),f^n\D^{cs}(y'))$ is greater that $\delta$ and $\W^u_\delta(f^n\D^{cs}(y)$ is contained in $f^nD$ then in the intrinsic metric of $f^nD$ every point of $\D^{cs}(y)$ is at distance greater than $\delta/2$ from every other point in $f^n\D^{cs}(y')$. Indeed, by contradiction, if $z\in f^n\D^{cs}(y)$ and $z'\in f^n\D^{cs}(y')$ are at distance less than $\delta/2$ and $\W^u_\delta(f^n\D^{cs}(y))\subset f^nD$
then $\W^u_\delta(z)$ intersects $\W^{cs}_\delta(z')$ and this intersection point needs to be a point in $f^n\D^{cs}(y')$ since $\W^u_\delta(z)$ is contained in $f^n\D^{cs}(y')$. It follows that $d_u(f^n\D^{cs}(y),f^n\D^{cs}(y'))<\delta$ and we get to a contradiction. This proves the claim.

Finally, given any constant $L>0$, let $K>0$ be an integer larger than $L/2\delta$. As $\eta$ is not contained in $\W^{cs}(x)$, there exist $K$ different $cs$-plaques in $D$ intersecting $\eta$. Let us denote them as $\D^{cs}(y_1)$, \ldots, $\D^{cs}(y_{K})$. There exists $N>0$ such that $d_u(f^n\D^{cs}(y_i),f^nD^{cs}(y_j))>2\delta$ for every $n\geq N$ and $i\neq j$. 

Moreover, for every $1\leq i \leq K$ there exist $\epsilon_i$ such that $\W^u_{\epsilon_i}(\D^{cs}(y_i))$ is contained in $D$. By taking $N$ larger, if needed, one can ensure that $\W^u_\delta(f^n\D^{cs}(y_i))$ is contained in $f^nD$ for every $n\geq N$.

It follows that $\length(f^n\circ \eta)>L$ for every $n\geq N$ since $f^n\circ \eta$ must contain at least $K$ disjoint subsegments of length at least $\delta/2$, each one of them corresponding to an intersection of $f^n\circ \eta$ with $f^n\D^{cs}(y_i)$ for every $1\leq i \leq K$.
\end{proof}

\begin{remark}\label{rmkWsinWcs}
From Lemma \ref{lemmaetac} one can easily justify that every $f\in \PH(M)$ admitting an $f$-invariant center-stable foliation $\W^{cs}$ satisfies that the leaves of $\W^{cs}$ are saturated by leaves of $\W^s$. 

Indeed, for every $x\in M$ and $y\in \W^s(x)$ one can join $x$ and $y$ by a $C^1$ curve $\eta$ contained in $\W^s(x)$. Since $\eta$ gets contracted uniformly by forward iterates of $f$ it follows that $\eta$ must be contained in $\W^{cs}(x)$. Then $\W^s(x)\subset \W^{cs}(x)$.
\end{remark}

\begin{proof}[Proof of Theorem \ref{thmB} item (\ref{thm2ii}) and Theorem \ref{thmB'} item (\ref{thm2ii})]
By Remark \ref{rmkDAFQI} if $f$ is a discretized Anosov flow then $f$ acts quasi-isometrically on the center foliation $\W^c$ given by the flow lines of the flow appearing in the definition of $f$. Hence for proving Theorem \ref{thmB} item (\ref{thm2ii}) it is enough to show that if $f$ acts quasi-isometrically on a one-dimensional center foliation $\W^c$ then $f$ admits a unique $f$-invariant center-stable foliation and a unique $f$-invariant center-unstable foliation. That is, it is enough to  show that Theorem \ref{thmB'} item (\ref{thm2ii}) is true.

Suppose $f$ acts quasi-isometrically on a one-dimensional center foliation $\W^c$. By Proposition \ref{propdyncoh} there exist $f$-invariant foliations $\W^{cs}$ and $\W^{cu}$ whose leaves are characterized as $\W^{cs}(x)=\W^s(\W^c(x))$ and $\W^{cu}(x)=\W^u(\W^c(x))$ for every $x\in M$.

Suppose $\W^{cs}_1$ is an $f$-invariant center-stable foliation. As $f$ acts quasi-isometrically in $\W^c$, one has that by Lemma \ref{lemmaetac} the leaf $\W^c(y)$ needs to be contained in $\W^{cs}_1(x)$ for every $x\in M$ and $y\in \W^{cs}_1(x)$. Moreover, as pointed out in Remark \ref{rmkWsinWcs} the leaf $\W^s(y)$ must also be contained in $\W^{cs}_1(x)$ for every $y\in \W^{cs}_1(x)$. It follows that $\W^{cs}(x)=\W^s(\W^c(x))$ needs to be a subset of $\W^{cs}_1(x)$ for every $x\in M$. 

For the intrinsic leaf metric induced by the Riemannian metric in $M$ each leaf of $\W^{cs}$ and $\W^{cs}_1$ is a complete metric space. This imply that the boundary of $\W^{cs}(x)$ in $\W^{cs}_1(x)$ needs to be empty. We conclude that $\W^{cs}(x)=\W^{cs}_1(x)$ for every $x\in M$.

Analogously for $f$-invariant center-unstable foliations.
\end{proof}

\subsection{Flow center foliation}

If $f$ is a discretized Anosov flow of the form $f(x)=\varphi_{\tau(x)}(x)$ as in Definition \ref{defDAFintro}, then by Proposition \ref{prop1DAFs} the flow lines of $\varphi_t$ form a center foliation $\W^c$ whose leaves are individually fixed by $f$.

By Remark \ref{rmkDAFQI} and Theorem \ref{thmB} items (\ref{thm2ii}) and (\ref{thm2iii}) we can deduce the following characterization of $\W^c$.

\begin{remark}
The foliation $\W^c$ is:
\begin{itemize}
\item The only foliation tangent to $E^c$ that is the intersection of $f$-invariant foliations $\W^{cs}$ and $\W^{cu}$. 
\item The only $f$-invariant foliation tangent to $E^c$ such that $f$ acts quasi-isometrically on it.
\end{itemize}
In particular, if $f$ is of the form $f(x)=\varphi'_{\tau'(x)}(x)$ for some other flow $\varphi'_t$ as in Definition \ref{defDAFintro}, then $\varphi'_t$ need to be a reparametrization of $\varphi_t$.
\end{remark}

In light of the above, we will designate from now on $\W^c$ as the \emph{flow center foliation of $f$}. And in view of the statement of Remark \ref{rmkDAFreparam}, if not otherwise stated we may implicitly assume from now on that the \emph{center flow} $\varphi_t:M\to M$ is parametrized by arc-length.

It would be interesting to know if, in general, the flow center foliation of a discretized Anosov is the only $f$-invariant center foliation, or at least if it is the only center foliation whose leaves are individually fixed by $f$. We do not have a proof for either of these statements.

\subsection{Topological Anosov flows}

\begin{definition}\label{deftopAnosov}
We say that a flow $\varphi_t:M\to M$ is a \emph{topological Anosov flow} if it is a continuous flow, with $\frac{\partial \varphi_t}{\partial t}|_{t=0}$ a continuous vector field without singularities, such that it preserves two topologically transverse continuous foliations $\F^{ws}$ and $\F^{wu}$ satisfying the following:

\begin{itemize}
\item[(i)] The foliation $\F^{ws}\cap\F^{wu}$ is the foliation given by the orbits of $\varphi_t$.
\item[(ii)] Given $x$ in $M$ and $y\in \F^{ws}(x)$ (resp. $y\in \F^{wu}(x)$) there exists an increasing continuous reparametrization $h:\mathbb{R}\to \mathbb{R}$ such that $d(\varphi_t(x),\varphi_{h(t)}(y))\to 0$ as $t\to +\infty$ (resp. $t\to -\infty$).
\item[(iii)] There exists $\epsilon>0$ such that for every $x\in M$ and $y\in \F^{ws}_\epsilon(x)$ (resp. $y\in \F^{wu}_\epsilon(x)$), with $y$ not in the same orbit as $x$, and for every increasing continuous reparametrization $h:\mathbb{R}\to \mathbb{R}$ with $h(0)=0$, there exists $t\leq 0$ (resp. $t\geq 0$) such that $d(\varphi_t(x),\varphi_{h(t)}(y))>\epsilon$.
\end{itemize}
\end{definition}

It is worth noting that Definition \ref{deftopAnosov} is a priori more restrictive than other definitions of topological Anosov flows appearing in the literature since we are asking for $\frac{\partial \varphi_t}{\partial t}|_{t=0}$ to be a continuous vector field.

It has been a long standing problem to determine whether in general every topological Anosov flow is orbit equivalent to an Anosov flow. Just recently in \cite{Sha} every transitive topological Anosov flow in dimension 3 (for a more general definition of topological Anosov flow that covers Definition \ref{deftopAnosov}) has been shown to be orbit equivalent to a smooth Anosov flow.

\begin{prop}[Theorem \ref{thmB} item (\ref{thm2i})]\label{propDAFtopAnosov} Let $f$ be a discretized Anosov flow and $\varphi_t$ be the flow appearing in the definition of $f$. Then $\varphi_t:M\to M$ is a topological Anosov flow.
\end{prop}
\begin{proof}
The map $f$ is of the form $f(x)=\varphi_{\tau(x)}(x)$ for some $\tau:M\to \R$ continuous. By Proposition \ref{prop1DAFs} the function $\tau$ has constant sign. Without loss of generality we can assume that $\tau$ is positive, otherwise we can argue analogously using $f^{-1}$ instead of $f$.

The flow $\varphi_t$ is a continuous flow with $\frac{\partial \varphi_t}{\partial t}|_{t=0}$ a continuous vector field. By Proposition \ref{propdyncoh} the map $f$ is dynamically coherent with center-stable foliation $\W^{cs}$ and center-unstable foliation $\W^{cu}$ such that $\W^c=\W^{cs}\cap \W^{cu}$ is the flow center foliation of $f$. Hence property (i) in  the definition of topological Anosov flow is immediately satisfied for $\F^{ws}=\W^{cs}$ and $\F^{wu}=\W^{cu}$.

Let us see property (ii). Suppose $x$ and $y$ are points in $M$ such that $y$ belongs to $\W^{cs}(x)$. By Proposition \ref{propdyncoh} the leaf $\W^{cs}(x)$ coincides with $\W^s(\W^c(x))$. Then $y$ belongs to $\W^s(z)$ for some $z\in \W^c(x)$.  

Let us assume first that $z=x$. Consider $\gamma_y:\R\to M$ the continuous curve in $\W^c(y)$ such that $\gamma_y(0)=y$ and $\gamma_y(t)\in \W^s(\varphi_t(x))$ for every $t$. The curve $\gamma_y$ is the transport by center holonomy of $y$ along stable transversals with respect to the $\varphi_t$-orbit of $x$.

The key property to note is that $\gamma_y(\tau(x))=f(y)$ for every $y\in \W^s(x)$. In fact, more generally, if $y'=\gamma_y(t)$ for some $t\in \R$ then \begin{equation}\label{eqTopAflow}
f(y')=\gamma_y(\tau(\varphi_t(x))).\end{equation}
This property follows immediately from the continuity of $\tau$ and local product structure of the foliations $\W^c$ and $\W^s$ restricted to $\W^{cs}(x)$. For more details, one can see Section \ref{sectioncenterholonomy} for a precise characterization of discretized Anosov flows in terms of center holonomy.

Let $R>0$ denote a constant such that $\gamma_y(t)\in \W^s_R(\varphi_t(x))$ for every $t\in [0,\tau(x)]$. This constant exists since the stable distance $d_s(\gamma_y(t),\varphi_t(x))$ varies continuously with $t$. Then, as $f$ contracts distances uniformly inside stable leaves, it follows from (\ref{eqTopAflow}) that $\lim_{t\to +\infty}d(\gamma_y(t),\varphi_t(x))=0$. Defining $h_y:\R\to \R$ as the increasing reparametrization such that $\varphi_{h_y(t)}(y)=\gamma_y(t)$ for every $t$ we obtain (ii) for the case $z=x$.

If $z$ is different from $x$ consider some $t_0>0$ and $h:(-\infty,t_0]\to \R$ continuous and increasing so that $h(0)=0$ and $y'=\varphi_{h(t_0)}(y)$ lies in  $\W^s(x')$ for $x'=\varphi_{t_0}(x)$. Defining as above $h_{y'}:\R \to \R$ so that $\varphi_{h_{y'}(t)}(y')=\varphi_t(x')$ for every $t$, then the function $h$ can be extended to $h:\R \to \R$ by the formula $h(t)=h(t_0)+h_{y'}(t-t_0)$ for every $t>t_0$. It follows that $\lim_{t\to +\infty}d(\varphi_t(x),\varphi_{h(t)}(y))=0$ as above. 

In the case $y$ lies in $\W^{cu}(x)$ one argues analogously for $\varphi_t$-past iterations. This settles property (ii).

Finally, let us see property (iii). As the bundles $E^c$ and $E^u$ vary continuously there exists a small constant $\epsilon>0$ such that for every $z$ and $z'$ satisfying $z'\in \W^{cu}_\epsilon(z)$ it follows that $\W^c_{2\epsilon}(z')$ and $\W^u_{2\epsilon}(z)$ intersect and that this intersection point is unique.

Let $x$ and $y$ be points in $M$ such that $y\in \W^{cu}_\epsilon(x)$. Suppose that $h:\mathbb{R}\to\mathbb{R}$ is an increasing continuous reparametrization with $h(0)=0$ such that $d(\varphi_t(x),\varphi_{h(t)}(y))\leq \epsilon$ for every $t\geq 0$. Let $y'$ denote the intersection $\W^c_{2\epsilon}(y)$ and $\W^u_{2\epsilon}(x)$ and let $\gamma_{y'}(t)=\W^c_{2\epsilon}(\varphi_{h(t)}(y))\cap \W^u_{2\epsilon}(\varphi_t(x))$ for every $t\geq 0$. The curve $\gamma_{y'}$ is no other than the transport by center holonomy of $y'$ along unstable transversals with respect to the $\varphi_t$-orbit of $x$. In analogy with (\ref{eqTopAflow}) is follows that $f(y')=\gamma_{y'}(\tau(x))$, so $f(y')$ lies in $\W^u_{2\epsilon}(f(x))$. Inductively, $f^n(y')$ lies in $\W^u_{2\epsilon}(f^n(x))$ for every $n>0$. Iterating  $n$ times backwards and taking limit with $n$ we conclude that $y'$ needs to coincide with $x$. Then $y$ lies in $\W^c_{2\epsilon}(x)$ and, in particular, lies in the $\varphi_t$-orbit of $x$.

In the case $x$ and $y$ are points such that $y\in \W^{cs}_\epsilon(x)$ one can argue analogously for past iterates of $f$ and $\varphi_t$. Property (iii) is settled.
\end{proof}

Let us end this subsection with a proposition showing that some classical properties for Anosov flows are satisfied (by means of the same type of arguments) by the topological Anosov flows arising as center foliations of discretized Anosov flows. Some of these properties will be needed later in the text. For the sake of completeness we will sketch their proofs.

We say that a leaf of a foliation of dimension $d>0$ is a \emph{plane} if it is homeomorphic to $\R^d$, and that it is a \emph{cylinder} if it is homeomorphic to a fiber bundle over the circle whose fibers are homeomorphic to $\R^{d-1}$. We say that two foliations $\W$ and $\W'$ have \emph{global product structure} if $\W(x)$ and $\W'(y)$ intersect for every pair $x$ and $y$, and this intersection is a unique point. 

The following proves Theorem \ref{thmB} item (\ref{thm2iv}).

\begin{prop}\label{propsTopAF} Suppose $f$ is a discretized Anosov flow. Let $\varphi_t$ and $\W^c$ denote the flow and center foliation appearing in the definition of $f$. Let $\W^{cs}$ and $\W^{cu}$ denote the center-stable and center-unstable foliations such that $\W^c=\W^{cs}\cap\W^{cu}$. Then:
\begin{enumerate}
\item\label{item1propTAFs} Every leaf of $\W^{cs}$ and $\W^{cu}$ is a plane or a cylinder.
\item\label{item2propTAFs} If a leaf $\W^{cs}(x)$ is a plane then $\W^c$ and $\W^s$ restricted to $\W^{cs}(x)$ have global product structure. Analogously for $\W^{cu}$-leaves.
\item\label{item3propTAFs} If a leaf $\W^{cs}(x)$ is a cylinder then $\W^c$ restricted to $\W^{cs}(x)$ contains a unique compact leaf $L$ and the omega limit set under $\varphi_t$ of every point $y$ in $\W^{cs}(x)$ is $L$. Analogously for $\W^{cu}$-leaves and alpha limit sets.
\item\label{item4propTAFs} There exists at least one compact leaf for $\W^c$.
\end{enumerate}

\end{prop}
\begin{proof}
Let $x$ be a point in $M$. For every $y \in \W^s(x)$ we can define $\gamma_y:\R\to M$ as the continuous curve in $\W^c(y)$ such that $\gamma_y(0)=y$ and $\gamma_y(t)\in \W^s(\varphi_t(x))$ for every $t$. The curve $\gamma_y$ is a transport by center holonomy of $y$ with respect to the $\varphi_t$-orbit of $x$. As in the previous proposition, note the key property: $\gamma_y(\tau(x))=f(y)$ for every $y\in \W^s(x)$.

If $x$ is a periodic point for $\varphi_t$ of period $t_x>0$ let us denote $H(y)\in \W^s(x)$ to the point $\gamma_y(t_x)$ for every $y\in \W^s(x)$. For some $N>0$ large enough $H^N:\W^s(x)\to \W^s(x)$ is a contraction with $x$ the unique fixed point. In this case $\W^{cs}$ is a cylinder and it is immediate to check that the $\varphi_t$-omega limit of every point in $\W^{cs}(x)$ is the orbit of $x$.

If $x$ is not periodic for $\varphi_t$ but some point $y$ in $\W^s(x)$ is periodic then we can argue as above and conclude that $\W^{cs}(x)$ is a cylinder and that the $\varphi_t$-omega limit of every point in $\W^{cs}(x)$ is the orbit of $y$.

If none of the points in $\W^s(x)$ are periodic for $\varphi_t$ then for every $y\in \W^s(x)$ the point $\gamma_y(t)$ lies in $\W^s(x)$ if and only if $t=0$, otherwise a contraction $H^N:\W^s(x)\to\W^s(x)$ as above can be constructed and some $\varphi_t$-periodic point in $\W^s(x)$ should be found. It follows that $\bigcup_{y\in \W^s(x)}\gamma_y(t)=\W^s(\varphi_t(x))$ for every $t$ and, since $\W^{cs}(x)=\W^s(\W^c(x))$ by Proposition \ref{propdyncoh}, then $\W^{cs}(x)$ is a plane and $\W^c$ and $\W^s$ have a global product structure inside $\W^{cs}(x)$.

Properties (\ref{item1propTAFs}), (\ref{item2propTAFs}) and (\ref{item3propTAFs}) are settled. Let us see that $\varphi_t$ must have at least one periodic orbit and this will settle the last property.

For some $x$ in $M$ let $z$ be a point in the $\varphi_t$-omega limit of $x$. Consider $D$ a small $C^1$ disc transverse to $\W^c$ and containing $z$ in its interior. Let $D$ be such that the leaves of $\W^{cs}$ and $\W^{cu}$ intersect $D$ in $C^1$ discs. For every $z'\in D$ let $w^s(z')$ and $w^u(z')$ denote the connected components of $\W^{cu}(z')\cap D$ and $\W^{cs}(z')\cap D$ containing $z'$, respectively.

Let $D'\subset D$ be such that if $z',z''\in D'$ then $w^s(z')\cap w^u(z'')\neq \emptyset$ and $w^u(z')\cap w^s(z'')\neq \emptyset$. For every $z'\in D'$ let $\pi^u(z')$  denote the point in $w^s(z)$ such that $w^u(z')\cap w^s(z)=\pi^u(z')$.

Let $t_x>0$ be a time such that $\varphi_{t_x}(x)$ lies in $D'$ close to $z$ and let $T_x>t_x$ be a large enough time so that $\varphi_{T_x}(x)$ lies also in $D'$, is close to $z$ and the Poincaré return map $P$ from $w^s(\varphi_{t_x}(x))$ to $D'$ is well defined. Then $\pi^u \circ P$ needs to be a contraction if $T_x$ is large enough. Let $z'$ denote the fixed point of this contraction. It follows that $P(z')$ lies in $w^u(z')$ so there exists some positive time $t_{z'}$ close to $T_x-t_x$ such that $\varphi_{t_{z'}}(z')$ lies in $\W^u(z')$. By (\ref{item3propTAFs}) it follows that $\W^{cu}(z')$ has to be a cylinder leaf and, as a consequence, it has to contain a periodic orbit for $\varphi_t$.
\end{proof}

\subsection{Equivalence with other definitions}\label{sectionequivalences}

Discretized Anosov flows have been richly studied in the literature, though not always under this name. Without trying to be exhaustive, it is worth establishing that many of these classes studied before are in fact discretized Anosov flows as in Definition \ref{defDAFintro}. This is one of the primary goals of this text.

In \cite{BFFP}, \cite{BFP}, \cite{BaGo} and \cite{GM} a map $f\in \PH_{c=1}(M)$ was called a `discretized Anosov flow' if it satisfied the following: there exist a topological Anosov flow $\varphi_t:M\to M$ and a continuous function $\tau:M\to \mathbb{R}_{>0}$ such that $f(x)=\varphi_{\tau(x)}(x)$ for every $x$ in $M$.

As a direct consequence of Theorem \ref{thmB} item (\ref{thm2i}) and Proposition \ref{prop1DAFs} item \ref{item2prop2} we obtain:

\begin{cor}\label{corequivDAFS}
The definition of discretized Anosov flow given in \cite{BFP},  \cite{BFFP}, \cite{BaGo} and \cite{GM} is equivalent with Definition \ref{defDAFintro}.
\end{cor}

It is worth noting the following two other classes of systems studied before that are also discretized Anosov flows. 

\begin{remark}\label{rmkBW05}
Partially hyperbolic diffeomorphisms on $3$-manifolds were investigated in the seminal article \cite{BW}. The statement of \cite[Theorem 2. items 1. and 2.]{BW} can be paraphrased as the following criterion for detecting discretized Anosov flows (in particular, using Proposition \ref{propcenterfixingL} to conclude):

\emph{Suppose $f\in \PH_{c=1}(M^3)$ is transitive and dynamically coherent with invariant foliations $\W^{cs}$, $\W^{cu}$ and $\W^c=\W^{cs}\cap \W^{cu}$. Then $f^n$ is a discretized Anosov flow for some $n>0$ if and only if there exists a periodic compact leaf $\eta\in\W^c$ and every center leaf through $\W^s_{loc}(\eta)$ is also periodic by $f$.}
\end{remark}

\begin{remark} In \cite{BG1} and \cite{BG2} diffeomorphisms in $\PH_{c=1}(M)$ that are Axiom A and admit a center foliation tangent to an Anosov vector field $X^c$ were studied. In \cite{BG1} it is shown that these systems can be written as $f(x)=X^c_{\tau(x)}(x)$ for some $\tau:M\to \R^+$ continuous. It follows that, in particular, they are all discretized Anosov flows.
\end{remark}

, we can establish the equivalence with the notion of \emph{flow-type partially hyperbolic diffemorphism}. In \cite{BFT} a diffeomorphism $f\in \PH_{c=1}(M)$ is called \emph{flow-type} if it is dynamically coherent with orientable center foliation $\W^c=\W^{cs}\cap \W^{cu}$ admitting a compact leaf and such that $f$ can be written as $f(x)=\varphi_{\tau(x)}(x)$ for every $x\in M$, where $\varphi_t$ is the flow of unit positive speed along the leaves of $\W^c$ and $\tau:M\to \mathbb{R}_{>0}$ is some continuous function.

As a consequence of what we have seen so far we get the following:

\begin{cor}\label{corequivflowtypeDAF}
The definition of flow-type partially hyperbolic diffeomorphism as given in \cite{BFT} is equivalent with Definition \ref{defDAFintro} of a discretized Anosov flow.
\end{cor}
\begin{proof} It is immediate to check that every flow-type partially hyperbolic diffemorphism is a discretized Anosov flow as in Definition \ref{defDAFintro}.

Conversely, suppose $f$ is a discretized Anosov flow and let $\varphi_t$ and $\W^c$ denote the flow and center foliation appearing in the definition of $f$. Theorem \ref{thmB} item (\ref{thm2ii}) shows that every discretized Anosov is dynamically coherent with center-stable foliation $\W^{cs}$ and center-unstable foliation $\W^{cu}$ such that $\W^c=\W^{cs}\cap\W^{cu}$. Moreover, modulo reparametrization and inverting the time of $\varphi_t$, Proposition \ref{prop1DAFs} and Remark \ref{rmkDAFreparam} show that $f$ can be written down as $f(x)=\varphi_{\tau(x)}(x)$ where $\varphi_t$ is parametrized by arc-length and $\tau:M\to \mathbb{R}$ is continuous and positive. Finally, Proposition \ref{propsTopAF} shows that $\W^c$ has a compact leaf. Thus $f$ needs to be a flow-type partially hyperbolic diffeomorphism.
\end{proof}

\subsection{Discretized Anosov flows and center holonomy}
\label{sectioncenterholonomy} Let us end this section by pointing out a characterization of discretized Anosov flows in terms of center holonomy. 

Recall the definition of a holonomy map for a foliation:

\begin{remark}[Holonomy map along a curve]\label{rmkDAFiffcenterholonomy} Suppose $\W$ is a foliation with $C^1$ leaves tangent to a continuous subbundle in the compact Riemannian manifold $M$. The construction that follows is standard to check.

Suppose $x$ in $M$, $y$ in  $\W(x)$ and $\gamma:[0,1]\to \W(x)$ a $C^1$ curve such that $\gamma(0)=x$ and $\gamma(1)=y$. Suppose $D_x$ and $D_y$ are $C^1$ discs transverse to $\W$, containing $x$ and $y$ respectively. Let $\delta>0$ be a constant such that every ball in $M$ of radius $2\delta$ is contained in a foliation box neighborhood of $\W$.

Every small enough $C^1$ disc $D'_x\subset D_x$ containing $x$ has the property that for every $z\in D_x'$ there exists a $C^1$ curve $\gamma_z:[0,1]\to\W(z)$ such that $\gamma_z(0)=z$, $\gamma_z(1)\in D_y$ and $d(\gamma(t),\gamma_z(t))<\delta$ for every $t\in [0,1]$. Moreover, the point $\gamma_z(1)$ in $D_y$ is independent of the choice of such a $\gamma_z$. In particular, there exists a well defined \emph{holonomy map along $\gamma$} $$H:D_x'\to D_y$$ given by $H(z)=\gamma_z(1)$ for every $z\in D_x'$. 

Furthermore, one can chose the curves $\gamma_z$ so that $z\mapsto \gamma_z$ varies continuously in the $C^1$ topology as $z$ varies continuously in  $D'_x$.
\end{remark}

The following characterizes discretized Anosov flows in terms of center holonomy:

\begin{prop}\label{propDAFiffH=f}
Suppose $f\in \PH_{c=1}(M)$. The following are equivalent:
\begin{enumerate}[label=(\roman*)]
\item\label{item1propHol} The map $f$ is a discretized Anosov flow.

\item\label{item2propHol} The bundle $E^c$ integrates to an $f$-invariant foliation $\W^c$ such that for every $x\in M$ there exist:
\begin{itemize}
\item A curve $\gamma:[0,1]\to \W^c(x)$ with $\gamma(0)=x$ and $\gamma(1)=f(x)$,
\item A $C^1$ disc $D$ transverse to $\W^c$ with $x\in D$ such that the $\W^c$ holonomy map $H$ along $\gamma$ is well defined from $D$ to $f(D)$ and satisfies $$H(y)=f(y)$$ for every $y \in D$.
\end{itemize}
\end{enumerate}
\end{prop}
\begin{proof}
Suppose that $f$ is a discretized Anosov flow. By Definition \ref{defDAFintro}, Proposition \ref{prop1DAFs} and Remark \ref{rmkDAFreparam} the map $f$ can be written down as $f(x)=\varphi_{\tau(x)}(x)$, where $\tau:M\to \R_{>0}$ is continuous and $\varphi_t:M\to M$ is a unit speed flow whose flow lines coincide with the leaves of the flow center foliation $\W^c$ of $f$.

Given $x\in M$ let $\gamma:[0,1]\to \W^c(x)$ be the reparametrization of the piece of $\varphi_t$ orbit from $x$ to $f(x)$ so that $\|\frac{\partial \gamma_t}{\partial t}\|=\frac{1}{\tau(x)}$ for every $t\in [0,1]$. Let $D_x$ be a $C^1$ disc containing $x$ and transverse to $\W^c$. Then $f(D_x)$ contains $f(x)$ and is also a $C^1$ disc transverse to $\W^c$. 

Let $\delta>0$ be a constant such that every ball of radius $2\delta$ is contained in a foliation box neighborhood of $\W^c$. As in Remark \ref{rmkDAFiffcenterholonomy}, let $D_x'\subset D_x$ be such that $x\in D_x'$ and the holonomy map along $\gamma$ $$H:D_x'\to f(D_x)$$ is well defined.

For every $z\in D_x'$ let $\gamma_z:[0,1]\to \W^c$ denote the piece of $\varphi_t$ orbit from $z$ to $f(z)$ reparametrized so that $\|\frac{\partial (\gamma_y)_t}{\partial t}\|=\frac{1}{\tau(y)}$ for every $t\in [0,1]$.

We can assume that $D_x'$ is small enough so that $d(\gamma_z(t),\gamma(t))<\delta$ for every $z\in D_x'$ and $t\in [0,1]$. It follows that $f(z)=\gamma_z(1)$ for every $z\in D_x'$. This shows that \ref{item1propHol} implies \ref{item2propHol}.

Conversely, suppose that \ref{item2propHol} is satisfied. In particular, $f$ individually fixes each leaf of $\W^c$. Given $x\in M$ let $\gamma$ and $D$ be as in \ref{item2propHol}. Let us see that locally in a neighborhood of $x$ the condition $f(w)\in \W^c_L(w)$ is satisfied for some $L>0$.

Let $\delta>0$ be such that every ball of radius $2\delta$ is contained in a foliation box neighborhood of $\W^c$. As in Remark \ref{rmkDAFiffcenterholonomy} let $D'\subset D$ be a $C^1$ disc containing $x$ so that its closure is a subset of $D$ and such that for every $y\in D'$ a $C^1$ holonomy curve $\gamma_y:[0,1]\to \W^c(y)$ with the following properties is well defined: $\gamma_y(0)=y$, $\gamma_y(1)=f(y)$ and $d(\gamma(t),\gamma_y(t))<\delta$ for every $t\in [0,1]$. 

Moreover, the curves $\gamma_y$ can be considered so that $y\mapsto \gamma_y$ varies continuously with $y$. Then $y\mapsto \length (\gamma_y)$ varies continuously and as consequence there exists $K>0$ a constant larger than $\sup_{y\in D'} \length(\gamma_y)$.

Let $U$ be a foliation box neighborhood of $\W^c$ obtained as $\bigcup_{z\in f(D')}\W^c_{\epsilon_1}(z)$ for some small $\epsilon_1>0$. Let $\epsilon_2>0$ be such that $f(\W^c_{\epsilon_2}(y))$ is a subset of $\W^c_{\epsilon_1}(f(y))$ for every $y\in D'$ and let $U'$ be the neighborhood  $\bigcup_{y\in D'}\W^c_{\epsilon_2}(y)$. It follows that $f(w)$ lies in $\W^c_{K+\epsilon_1+\epsilon_2}(w)$ for every $w\in U'$.

Let us rename $U'$ as $U_x$ and $K+\epsilon_1+\epsilon_2$ as $L_x$ to highlight the dependence on the point $x$. We conclude that for every $x\in M$ there exists a neighborhood $U_x$ and a constant $L_x$ so that $f(w)\in \W^c_{L_x}(y)$ for every $w\in U_x$. By taking a finite subcover $\{U_{x_i}\}_{i\in I}$ of the cover $\{U_x\}_{x\in M}$ of $M$ it follows that $f(w)\in \W^c_L(w)$ for every $w\in M$ and $L=\max_{x\in I}L_{x_i}$. Then \ref{item1propHol} follows as a consequence of Proposition \ref{propcenterfixingL}.
\end{proof}

\section{Continuation of normally hyperbolic foliations revisited}\label{normhypfolrevisited}\label{appendixA}

In this section we revisit the stability of normally hyperbolic foliations of \cite{HPS} (see also \cite{PSW}). The main goal is to prove Theorem \ref{thmgraphtransf} which guarantees that, in a certain sense, the continuation of a normally hyperbolic foliation can be carried out along sets of uniform size in $\PH_{c=1}(M)$. The immediate antecedent for this result is \cite[Theorem 4.1]{BFP} (see also  \cite[Section 4.1]{BFP} and \cite[Appendix B]{BFP}).

Everything in this section is independent from the previous one.

\subsection{Statements}

From now on throughout this section let $M$ be a closed (compact and without boundary) Riemannian manifold. 

Suppose $\C^1$ and $\C^2$ are continuous cone fields in $M$ of complementary dimension. 
Given constants $\epsilon,\delta>0$ we will say that the metric in $M$ and the cone fields $(\C^1,\C^2)$ are \emph{$\epsilon$-nearly euclidean at scale $\delta$} if for every $x\in M$ the exponential map $\exp_x:T_xM\to M$ restricted to $B_\delta(0)\subset T_xM$ is a diffeomorphism onto its image $B_\delta(x)\subset M$ satisfying that, if one identifies $T_xM$ isometrically with the euclidean space $\R^n$ by a linear map $A:T_xM\to \R^n$, then $$\big| \|D(A\circ \exp_x^{-1})_y(v_y)\|-1\big|<\epsilon$$ and $$\big| \measuredangle \big(D(A\circ \exp_x^{-1})_y(v_y^1), D(A\circ \exp_x^{-1})_{y'}(v_{y'}^2)\big)-\frac{\pi}{2}\big|<\epsilon \pi$$ 
for every $y,y'\in B_{\delta}(x)$, every unit vector $v_y$ in $T_yM$ and every unit vectors $v_y^1\in\C^1(y)$ and $v_{y'}^2\in\C^2(y')$.

Informally, for $\epsilon>0$ small the property of being $\epsilon$-nearly euclidean at scale $\delta$ indicates that in restriction to balls of radius $\delta$ the metric is close to being euclidean and the cone fields are fairly narrow, almost constant and almost pairwise orthogonal.

\begin{lemma}\label{rmkUdelta} Suppose $f_0\in \PH_{c=D}(M)$ for some $D>0$. There exists a Riemannian metric in $M$, constants $0<\lambda<1$, $\kappa>1$ and $\delta(f_0)>0$, and for every $\delta$ with $0 <\delta\leq\delta(f_0)$ a $C^1$-neighborhood $\mathcal{U}_\delta(f_0)\subset\PH_{c=D}(M)$ of $f_0$, such that:

\begin{itemize}

\item[(P1)]\label{P1} One has $\max\{\|Df_x\|,\|Df^{-1}_x\|\}<\kappa$ for every $x\in M$ and every $f \in \U_\delta(f_0)$.

\item[(P2)]\label{P2} One has $\|Df|_{E^s(x)}\|<\lambda$ and $\|Df^{-1}|_{E^u(x)}\|<\lambda$ for every $x\in M$ and every $f \in \U_\delta(f_0)$.

\item[(P3)]\label{P3} There exist continuous cone fields $\mathcal{C}^s$, $\mathcal{C}^{cs}$, $\mathcal{C}^u$ and $\mathcal{C}^{cu}$ on $M$ such that for every $f\in \U_\delta(f_0)$ and $x\in M$:  
\begin{enumerate}

\item The dimension of $\C^\sigma$ is equal to $\dim(E^\sigma)$ and the bundle $E^\sigma_f(x)$ is contained in $\C^\sigma(x)$ for every $x\in M$ and $\sigma\in\{s,cs,u,cu\}$.

\item The cones $\C^s$ and $\C^{cs}$ are $f^{-1}$-invariant and satisfy $E^\sigma_f(x)=\bigcap_{n\geq 0} Df^{-n}\mathcal{C}^\sigma_{f^n(x)}$ for every $x\in M$ and $\sigma\in\{s,cs\}$.

\item The cones $\C^u$ and $\C^{cu}$ are $f$-invariant and satisfy that $E^\sigma_f(x)=\bigcap_{n\geq 0} Df^n\mathcal{C}^\sigma_{f^{-n}(x)}$ for every $x\in M$ and $\sigma\in\{u,cu\}$.
\end{enumerate}

\item[(P4)]\label{P4} The metric and the cone fields $(\C^s,\C^{cu})$ and $(\C^{cs},\C^u)$ are $\frac{1}{16}$-nearly euclidean at scale $20\delta$.

\item[(P5)]\label{P5}
The $C^0$ distance $d_0(f,g)$ is smaller than $\frac{\delta}{64\kappa^2}(1+\lambda+\lambda^2+\cdots)^{-1}$ and smaller than $\frac{1}{10}(\lambda^{-1}-1)$ for every $f,g\in \mathcal{U}_\delta(f_0)$.
\end{itemize}
\end{lemma}

Recall from the preliminaries that by `$f$-invariant cone field' we mean the usual containment property given by $Df^N\C(x)\subset \intt \C(f^N(x))$ for every $x\in M$ and some uniform $N>0$.

\begin{proof}[Proof of Lemma \ref{rmkUdelta}] Let us start by considering $\U$ a $C^1$ open neighborhood of $f_0$ in $\PH_{c=D}(M)$. If $\U$ is small enough then property (P1) is automatically satisfied for some constant $\kappa>1$.

By \cite{Gou} there exists a constant $0<\lambda<1$ and an \emph{adapted metric} $\mathit{g_1}$  in $M$ such that $f_0$ satisfies $\|D{f_0}|_{E^s(x)}\|<\lambda$ and $\|D{f_0}^{-1}|_{E^u(x)}\|<\lambda$ for every $x\in M$. 

Let $\mathit{g_2}$ be the metric that makes the subbundles $E^s_{f_0}$, $E^c_{f_0}$ and $E^u_{f_0}$ pairwise orthogonal and coincides with $\mathit{g_1}$ in restriction to each of them. Note that since the invariant bundles of $f_0$ vary a priori only continuously with respect to the point in $M$ we can not guarantee that $\mathit{g_1}$ has better regularity than continuous. Nevertheless, if we consider $\mathit{g}$ a $\C^\infty$ metric close enough to $\mathit{g_2}$ we can ensure that $\|D{f_0}|_{E^s(x)}\|<\lambda$ and $\|D{f_0}^{-1}|_{E^u(x)}\|<\lambda$ is still satisfied for every $x\in M$ and that the pairwise angles between the subbundles $E^s_{f_0}$, $E^c_{f_0}$ and $E^u_{f_0}$ lie in $(\pi/2-\pi/64,\pi/2+\pi/64)$.

Since the invariant bundles vary continuously in the $C^1$ topology we can shrink $\U$, if necessary, so that (P2) is satisfied for every $f\in\U$ with respect to the same constant $\lambda$ and such that the pairwise angles between the subbundles $E^s_f$, $E^c_f$ and $E^u_f$ also lie in $(\pi/2-\pi/64,\pi/2+\pi/64)$ for every $f\in\U$.

In order to obtain (P3) and (P4) let $\bar{\C}^s$, $\bar{\C}^{cs}$, $\bar{\C}^u$ and $\bar{\C}^{cu}$ be invariant cone fields, given by the partial hyperbolicity of $f_0$, satisfying that $E^\sigma_{f_0}(x)=\bigcap_{n\geq 0} Df_0^{-n}\bar{\C}^\sigma_{f_0^n(x)}$
for every $x\in M$ and every $\sigma\in\{s,cs\}$, and $E^\sigma_{f_0}(x)=\bigcap_{n\geq 0} Df_0^n\bar{\C}^\sigma_{f_0^{-n}(x)}$ for every $x\in M$ and every $\sigma\in\{u,cu\}$.

Let us define $\C^s_x=Df_0^{-N}\bar{\C}^s_{f_0^N(x)}$, $\C^u_x=Df_0^{N}\bar{\C}^u_{f_0^{-N}(x)}$, $\C^{cu}_x=Df_0^N\bar{\C}^{cu}_{f_0^{-N}(x)}$ and $\C^{cs}_x=Df_0^{-N}\bar{\C}^{cs}_{f_0^N(x)}$ for $N>0$ large enough so that the angle between every vector of $\C^\sigma_x$ and $E^\sigma_{f_0}(x)$ is less than $\pi/64$, for every $x\in M$ and every $\sigma\in \{s,u,cs,cu\}$.

By shrinking $\U$ even more in the $C^1$ topology, if necessary, one obtains that (1), (2) and (3) of property (P3) need to be fulfilled by every $f\in \U$.

Moreover, it is not difficult to check that for every $x\in M$ there exists $\delta_x>0$ such that for every $z\in B_{\delta_x}(x)$ the exponential map $\exp_z:T_zM\to M$ restricted to $B_{\delta_x}(0)\subset T_zM$ is a diffeomorphism onto its image $B_{\delta_x}(z)\subset M$ and, if one identifies isometrically $T_zM$ with euclidean $\mathbb{R}^n$ by a linear map 
$A:T_zM\to \R^n$, then 

$$\big| \|D(A\circ \exp_z^{-1})_y(v_y)\|-1\big|<\frac{1}{16}$$ and $$\big| \measuredangle \big(D(A\circ \exp_z^{-1})_y(v_y^\sigma), D(A\circ \exp_z^{-1})_{y'}(v_{y'}^{\sigma'})\big)-\frac{\pi}{2}\big|<\frac{\pi}{16}$$ 
for every $y,y'\in B_{\delta_x}(z)$, every unit vector $v_y$ in $T_yM$ and every unit vectors $v_y^\sigma\in \mathcal{C}^{\sigma}_y$ and $v_{y'}^{\sigma'}\in \mathcal{C}^{\sigma'}_{y'}$ for every pair $(\sigma,\sigma')\in\{(s,cu),(cs,u)\}$.

By taking a finite subcover $\{B_{\delta_{x_i}}(x_i)\}_{1\leq i\leq k}$ of $M$ it follows that $\delta(f_0)=\frac{1}{20}\min\{\delta_{x_i}\}_{1\leq i\leq k}$ guarantees that property (P4) is satisfied by every $f\in \U$ for $\delta=\delta(f_0)$.

Given $0<\delta\leq\delta(f_0)$, properties (P1),\ldots , (P4) are still fulfilled for every $f\in\U$. It is enough now to shrink $\U$ in the $C^0$ topology even more, if necessary, to a neighborhood $\U_{\delta}(f_0)$ so that property (P5) is satisfied for every $f,g\in \U_{\delta}(f_0)$.
\end{proof}

\begin{remark}\label{rmkafterUdelta} It is worth pointing out that, according to the order in which each property of Lemma \ref{rmkUdelta} is proven, one can always obtain what follows. Given $f_0\in \PH_{c=D}(M)$ for some $D>0$, there exist constants $0<\lambda<1$ and $\kappa>1$, a sequence of metrics $\mathit{g_n}$ in $M$, a non-increasing sequence of constants $\delta_n(f_0)>0$, for every $n\geq 0$ and $\delta$ with $0 <\delta\leq\delta_n(f_0)$ a $C^1$-neighborhood $\mathcal{U}^{(n)}_\delta(f_0)\subset\PH_{c=D}(M)$ of $f_0$, and sequences $C_n\xrightarrow{n}+\infty$ and $\epsilon_n\xrightarrow{n} 0$ so that for each $n$ properties (P1), (P2), (P3) and (P5) always satisfied for every $f,g\in \mathcal{U}^{(n)}_\delta(f_0)$, and property (P4) states that the cone fields $(\C^s,\C^{cu})$ and $(\C^{cs},\C^u)$ are $\epsilon_n$-nearly euclidean at scale $C_n\delta$. This will be used in Section \ref{sectionproofthmA} where a narrower version of property (P4) is needed.
\end{remark}

Suppose $E$ is a continuous subbundle of $TM$. If $N$ is a connected manifold of dimension $\dim(E)$ we say that $\eta:N\to M$ is a \emph{complete $C^1$ immersion tangent to $E$} if $\eta$ is a (not necessarily injective) $C^1$ map such that $D_x\eta(T_xN)=E(\eta(x))$ for every $x\in N$ and such that the pull-back metric in $N$ is complete. Moreover, if $L\subset M$ denotes the image of $\eta$ we say that $L$ is a \emph{complete $C^1$ immersed submanifold tangent to $E$}.

\begin{thm}\label{thmgraphtransf} (Uniform continuation of normally hyperbolic foliations).
Suppose $f_0\in \PH_{c=1}(M)$. Consider a metric in $M$ and a constant $\delta(f_0)>0$ as in Lemma \ref{rmkUdelta}. Then for every $\delta$ with $0<\delta\leq\delta(f_0) $ a $C^1$ neighborhood $\U_\delta(f_0)$ as in Lemma \ref{rmkUdelta} satisfies the following properties.

For every pair $f$ and $g$ in  $\U_\delta(f_0)$, if $\W^c$ is an $f$-invariant center foliation, then there exists \begin{itemize}
\item A map $h:M\to M$ continuous, surjective and $\delta$-close to identity,

\item A homeomorphism $\rho:M\to M$ such that for every leaf $L\in\W^c$, one has $\rho(L)=L$ and the map $\rho|_L:L\to L$ is a $C^1$ diffeomorphism that is $\delta$-close to the identity on $L$,

\end{itemize}
such that
\begin{enumerate}

\item\label{(1)thmplaqueexp} 
For every leaf $L\in\W^c$ the set $h(L)$ is a complete $C^1$ immersed submanifold tangent to $E^c_{g}$. Furthermore, the map $h|_L:L\to M$ is $C^1$ with respect to the inner differentiable structure of $L$, the derivative $D(h|_L)|_{E^c_f}$ varies continuously in $M$ and satisfies $D(h|_L)_x(E^c_f(x))=E^c_{g}(h(x))$ and $\frac{1}{2}<\|D(h|_L)_x|_{E^c_f(x)}\|<2$ for every $x\in L$. 

\item  The equation $h \circ \rho \circ f(x)=g\circ h(x)$ is satisfied for every $x\in M$. In particular,  $h\circ f(L)= g \circ h(L)$ for every $L\in\W^c$.
\end{enumerate}
\end{thm}

From the proof of Theorem \ref{thmgraphtransf} we will also derive the following.

\begin{thm}[Uniform continuation of complete $C^1$ center immersions]\label{thmunifcontC1completecenterimm}

Suppose $f_0\in \PH_{c=1}(M)$. Consider a metric in $M$ and a constant $\delta(f_0)>0$ as in Lemma \ref{rmkUdelta}. Then for every $\delta$ with $0<\delta\leq\delta(f_0)$ a $C^1$ neighborhood $\U_\delta(f_0)$ as in Lemma \ref{rmkUdelta} satisfies the following properties.

If $f$ and $g$ are maps in  $\U_\delta(f_0)$ then for every $\eta:\R \to M$ a complete $C^1$ immersion tangent to $E^c_f$ there exists a sequence $\{\gamma_n:\R\to M\}_{n\in \mathbb{Z}}$ of complete $C^1$ immersions tangent to $E^c_{g}$ such that 
\begin{equation}\label{eq1thmC1imm}d(f^n\circ \eta(t),\gamma_n(t))<\delta\end{equation} and
\begin{equation}\label{eq2thmC1imm}\gamma_{n+1}\mbox{ is a reparametrization of }g\circ \gamma_n
\end{equation}
for every $t\in \mathbb{R}$ and $n\in \mathbb{Z}$.

Moreover, if $\{\gamma_n':\R\to M\}_{n\in \mathbb{Z}}$ is another sequence of complete $C^1$ immersions tangent to $E^c_{g}$ satisfying (\ref{eq1thmC1imm}) and (\ref{eq2thmC1imm}), then $\gamma_n'$ is a reparametrization of $\gamma_n$ for every $n\in \mathbb{Z}$.
\end{thm}

\subsection{Plaque expansivity and leaf-conjugacies}\label{subsectionplaqueexpappA} 

It is worth noting in this subsection some consequences of Theorem \ref{thmgraphtransf} before getting into its proof.

\begin{remark}\label{rmkleafconj}(Leaf-conjugacy)  Note that if $h$ is injective then $h(\W^c)$ is a $g$-invariant center foliation and $h$ is a homeomorphism taking leaves of $\W^c$ into leaves of $g$ such that $$h\circ f(L)= g \circ h(L)$$ for every $L$ in $\W^c$. That is, $(f,\W^{c}_{f})$ and $(g,\W^{c}_{g})$ are \emph{leaf-conjugate}.
\end{remark}

As detailed in Lemma \ref{lemmadeltaplaqueexp=>leafconj} below, a sufficient condition for $h$ to be injective is given by the following property.

\begin{definition}\label{defdeltape} Suppose $f\in \PH(M)$ admits an $f$-invariant center foliation $\W^c$. Assume that a metric in $M$ has been fixed. We say that $(f,\W^c)$ is \emph{$\delta$-plaque expansive} if every pair of $\delta$-pseudo orbits $(x_n)_n$ and $(y_n)_n$ satisfying 
\begin{itemize}
\item $x_{n+1}\in \W^c_\delta(f(x_n))$ for every $n \in \mathbb{Z}$
\item $y_{n+1} \in \W^c_{\delta}(f(y_n))$ for every $n \in \mathbb{Z}$,
\item $d(x_n,y_n)<2\delta$  for every $n \in \mathbb{Z}$,
\end{itemize}
also satisfy $y_0\in \W^c_{3\delta}(x_0)$.
\end{definition}

\begin{lemma}\label{lemmadeltaplaqueexp=>leafconj}
In the context of Theorem \ref{thmgraphtransf}, if $f\in \U_{\delta}(f_0)$ is $\delta$-plaque expansive then $h$ is a homeomorphism and $(f,\W^c)$ and $(g,h(\W^c))$ are leaf-conjugate.
\end{lemma}
\begin{proof} Suppose $h(x_0)=h(y_0)=z_0$ for some $x_0,y_0\in M$. The orbit of $z_0$ by $g$ defines two $\delta$-pseudo orbits for $f$ with `jumps' in $\W^c$-plaques as follows. 

Since $h\circ \rho \circ f=g\circ h$, the points $x_1=\rho(f(x_0))$ and $y_1=\rho(f(y_0))$, and inductively the points $x_{n+1}=\rho(f(x_n))$ and $y_{n+1}=\rho(f(y_n))$ for every $n\in \mathbb{Z}$, satisfy $$g^n(z_0)=h(x_n)=h(y_n)$$ for every $n\in \mathbb{Z}$. 

As $h$ and $\rho$ are $\delta$-close to the identity, the sequences $(x_n)_{n\in\mathbb{Z}}$ and $(y_n)_{n\in \mathbb{Z}}$ satisfy $x_{n+1}\in \W^c_\delta(f(x_n))$, $y_{n+1}\in \W^c_\delta(f(y_n))$ and $d(x_n,y_n)<2\delta$.

If $f$ is $\delta$-plaque expansive the above implies that $y_0$ belongs to $\W^c_{3\delta}(x_0)$. By (\ref{(1)thmplaqueexp}) in Theorem \ref{thmgraphtransf} and property (P4) in Lemma \ref{rmkUdelta} the image by $h$ of $\W^c_{3\delta}(x_0)$ is a $C^1$ arc tangent to $E^c_{g}$ and $h$ restricted to $\W^c_{3\delta}(x_0)$ is a $C^1$ diffeomorphism over its image. As $y_0$ belongs to $\W^c_{3\delta}(x_0)$ and $h(x_0)=h(y_0)$ it follows that $x_0=y_0$. 

This proves the global injectivity of $h$. By Remark \ref{rmkleafconj} one concludes that $(f,\W^{c}_{f})$ and $(g,h(\W^c))$ are leaf-conjugate.
\end{proof}

It is important to note that, in contrast with the usual definition of plaque-expansivity (as given in the introduction and below), the notion of $\delta$-plaque expansivity is sensible to the metric one chooses for $M$. 

Note also that, for $\delta>0$ small, if $(f,\W^c)$ is $\delta$-plaque expansive with respect to some metric, then $(f,\W^c)$ is $\delta'$-plaque expansive with respect to the same metric for every $0<\delta'\leq\delta$.

Recall that $(f,\W^c)$ is called \emph{plaque expansive} if for some metric and some $\delta>0$ every pair of sequences $(x_n)_{n\in \mathbb{Z}}$ and $(y_n)_{n\in \mathbb{Z}}$ satisfying that $x_{n+1}\in \W^c_\delta(f(x_n))$, $y_{n+1}\in \W^c_\delta(f(y_n))$ and $d(x_n,y_n)<\delta$ for every $n\in \mathbb{Z}$ must also satisfy $y_0\in \W^c_{loc}(x_0)$. Note that here $\W^c_{loc}(x)$ should be understood as $\W^c_\epsilon(x)$ for some small $\epsilon>0$ independent of $x\in M$.

It is immediate to check that:

\begin{remark}\label{rmkdeltaplaqueexpisplaqueexp}
If $(f,\W^c)$ is $\delta$-plaque expansive with respect to some metric then $(f,\W^c)$ is plaque expansive.
\end{remark}
\begin{proof}
It is enough to consider $0<\delta'\leq \delta$ so that $\W^c_\delta(x)\subset \W^c_{loc}(x)$ for every $x\in M$. Then $(f,\W^c)$ being $\delta'$-plaque expansive automatically implies that $(f,\W^c)$ is plaque expansive.
\end{proof}

Conversely, the following is also satisfied.

\begin{lemma}\label{lemmaplaqueexp=>deltaplaqueexp}
Suppose $(f,\W^c)$ is a plaque expansive system in $M$. Given a metric in $M$ there exists $\delta>0$ such that $(f,\W^c)$ is $\delta$-plaque expansive with respect to that metric.
\end{lemma}
\begin{proof}
Suppose $(f,\W^c)$ is plaque expansive. Then for some metric in $M$ and some small $\epsilon>0$ there exists $\delta>0$ such that every pair of sequences $(x_n)_{n\in \mathbb{Z}}$ and $(y_n)_{n\in \mathbb{Z}}$ satisfying $x_{n+1}\in \W^c_\delta(f(x_n))$, $y_{n+1}\in \W^c_\epsilon(f(y_n))$ and $d(x_n,y_n)<\delta$ for every $n\in \mathbb{Z}$ must also satisfy $y_0\in \W^c_\epsilon(x_0)$. Note that $\delta>0$ can be considered as small as wanted so that the previous property remains to be true. At first, let $\delta$ be smaller than $\epsilon$.


Suppose we consider another metric in $M$  and let us denote by $d'$ the distance induced by this new metric (in contrast with $d$ for the first one). As $M$ is a compact manifold there exists $C\geq 1$ such that $\frac{1}{C}d'(x,y)\leq d(x,y) \leq C d'(x,y)$ for every $x,y\in M$. 

Let $d_c$ and $d'_c$ denote the distances inside center leaves with respect to $d$ and $d'$, respectively. Note that we can consider $C$ so that it also satisfies $\frac{1}{C}d'_c(x,y)\leq d_C(x,y) \leq C d'_c(x,y)$ for every $x$ and $y$ in the same center leaf.

Finally, suppose that $\delta$ is small enough so that for every $0<\delta'\leq C\delta$, if $d_c'(x,y)<C\delta$ and $d'(x,y)<\delta'$, then $d_c'(x,y)<(3/2)\delta'$.

Under these conditions it is immediate to check that $f$ needs to be $\frac{\delta}{C}$-plaque expansive with respect to the new metric. Indeed, let $(x_n)_{n\in \mathbb{Z}}$ and $(y_n)_{n\in \mathbb{Z}}$ be such that $d_c'(x_{n+1},f(x_n))<\delta/C$, $d_c'(y_{n+1},f(y_n))<\delta/C$ and $d'(x_n,y_n)<2\delta/C$ for every $n\in \mathbb{Z}$. It follows that $d_c(x_{n+1},f(x_n))<\delta$, $d_c(y_{n+1},f(y_n))<\delta$ and $d(x_n,y_n)<2\delta$ for every $n\in \mathbb{Z}$. Then $y_0$ lies in $\W^c_\epsilon(x_0)$. Since $\W^c_\epsilon(x_0)$ is a subset of $\W^c_\delta(x_0)$, one has $d_c(x_0,y_0)<\delta$. Which in turns implies  $d_c'(x_0,y_0)<C\delta$. As $d'(x_0,y_0)<2\delta/C$ then from the last constraint imposed to $\delta$ it follows that $d'_c(x_0,y_0)<3\delta/C$.  
\end{proof}

Note that from the proof of the previous lemma one can also deduce the following:
\begin{lemma}\label{lastlemma?}
Consider two distinct metrics in $M$. Given $\delta>0$ there exists $C>0$ such that, if $(f,\W^c)$ is $\delta C$-plaque expansive with respect to the first metric, then $(f,\W^c)$ is $\delta$-plaque expansive with respect to the second one.
\end{lemma}

As a corollary we obtain:

\begin{cor}
Theorem \ref{thmunifHPSintro} follows from Theorem \ref{thmgraphtransf}.
\end{cor}
\begin{proof}
The statement of Theorem \ref{thmunifHPSintro} presupposes a metric in $M$. In parallel, let us consider $\delta(f_0)>0$ and the metric in $M$ as in Lemma \ref{rmkUdelta}. 

Let $C>0$ be as in Lemma \ref{lastlemma?} so that, if $(f,\W^c)$ is $\delta(f_0)C$-plaque expansive with respect to the first metric, then $(f,\W^c)$ is $\delta(f_0)$-plaque expansive with respect to the second one. It is enough to consider now $\delta:=\delta(f_0)C$ and $\U:=\U_{\delta(f_0)}(f_0)$. The rest follows by Theorem \ref{thmgraphtransf}.
\end{proof}

We recover also the classical stability statement for normally hyperbolic foliations (see \cite[Theorem 7.1]{HPS}):

\begin{cor}\label{lemmaplaqueexpopen}
Suppose $(f_0,\W^c)$ is a plaque expansive system in $\PH_{c=1}(M)$. There exists a neighborhood $\U\subset \PH_{c=1}(M)$ of $f_0$ such that every $f\in \U$ admits a $f$-invariant center foliation $\W^c_{f}$ such that $(f,\W^c_{f})$ is plaque expansive and leaf conjugate to $(f_0,\W^c)$.
\end{cor}
\begin{proof}
Let $(f_0,\W^c)$ be a plaque expansive system in $\PH_{c=1}(M)$. Let us consider the metric in $M$ and the constant $\delta(f_0)>0$ given by Lemma \ref{rmkUdelta}. 

By Lemma \ref{lemmaplaqueexp=>deltaplaqueexp} there exists $\delta>0$ such that $(f_0,\W^c)$ is $\delta$-plaque expansive (with respect to the metric we have just fixed). We can suppose that $\delta$ is smaller than $\delta(f_0)$.

Let $\delta'>0$ be such that $3\delta'<\delta$. Let $\U_{\delta'}(f_0)$ be the $C^1$ neighborhood of $f_0$ given by Lemma \ref{rmkUdelta}. If $f$ is a system in $\U_{\delta'}(f_0)$ then by Lemma \ref{lemmadeltaplaqueexp=>leafconj} the map $h$ given by Theorem \ref{thmgraphtransf} is a homeomorphism, and $(f_0,\W^c)$ and $(f,h(\W^c))$ are leaf conjugate. Let $\W^c_f$ denote $h(\W^c)$. It remains to see that $(f,\W^c_f)$ is plaque expansive.

Always with respect to the metric in $M$ given by Lemma \ref{lemmaplaqueexp=>deltaplaqueexp} suppose that $(x'_n)_{n\in \mathbb{Z}}$ and $(y'_n)_{n\in \mathbb{Z}}$ are $\delta'$ pseudo-orbits for $f$ so that $x'_{n+1}\in (\W^c_f)_{\delta'}(f(x'_n))$, $y'_{n+1} \in (\W^c_f)_{\delta'}(f(y_n))$ and $d(x'_n,y'_n)<2\delta'$  for every $n \in \mathbb{Z}$. Let us see that $y'_0$ must lie in $(\W^c_f)_{3\delta'}(x'_0)$.

Consider $x_n=h^{-1}(x_n')$ and $y_n=h^{-1}(y_n')$ for every $n\in \mathbb{Z}$. Let $\rho$ be the map given by Theorem \ref{thmgraphtransf}. As $\rho\circ f_0 (x_n)=h^{-1}\circ f(x'_n)$ and $\rho$ is $\delta'$ close to the identity it follows that $h^{-1}\circ f(x'_n)$ lies in $\W^c_{\delta'}(f_0(x_n))$. Moreover, as $1/2<\|Dh|_{E^c}\|<2$ and $x'_{n+1}\in (\W^c_f)_{\delta'}(f(x_n'))$, the point $h^{-1}\circ f(x'_n)$ lies in $\W^c_{2\delta'}(x_{n+1})$. It follows that $$x_{n+1}\in\W^c_{3\delta'}(f_0(x_n))$$ for every $n\in \mathbb{Z}$. Analogously for $(y_n)_{n\in \mathbb{Z}}$.

Moreover, as $d(h^{-1}(x_n'),x_n')<\delta'$ and $d(h^{-1}(y_n'),y_n')<\delta'$ because $h$ is $\delta'$-close to the identity, then $d(x_n',y_n')<2\delta'$ implies $$d(x_n,y_n)<4\delta'$$
for every $n\in \mathbb{Z}$.

As $3\delta'<\delta$ and $4\delta'<2\delta$ it follows from the $\delta$-plaque expansivity of $(f_0,\W^c)$ that $y_0$ needs to lie in $\W^c_{3\delta}(x_0)$. Then $x_0'$ needs to lie in $\W^c_{6\delta}(x_0')$ because of $1/2<\|Dh|_{E^c}\|<2$. Since $d(x_0',y_0')<2\delta'$ and because at scale $20\delta(f_0)$ the center bundles are almost constant (property (P4) in Lemma \ref{rmkUdelta}) it follows that $y_0'$ needs to lie in $(\W^c_f)_{3\delta'}(x_0')$. 

This shows that $(f,\W^c_f)$ is $\delta'$-plaque expansive. Then $(f,\W^c_f)$ is plaque expansive.
\end{proof}

\begin{remark}\label{rmkuniformdelta'plaqexp} Note that from the proof of the previous corollary the following statement can also be deduced: If $(f_0,\W^c)$ in $\PH_{c=1}(M)$ is plaque expansive and a metric as in Lemma \ref{rmkUdelta} has been fixed, then the $C^1$ neighborhood $\U\subset \PH_{c=1}(M)$ of $f_0$ given by Corollary \ref{lemmaplaqueexpopen} can be chosen so that there exists $\delta'>0$ such that $(f,\W^c_f)$ is $\delta'$-plaque expansive for every $f\in \U$ (with respect to the metric that has been fixed).
\end{remark}

\begin{remark}
To show the open property in Theorem \ref{thmA} and Theorem \ref{thmAprima} it is enough to show that the systems are plaque expansive.

In contrast, the following brief remark concerning the closed property is worth mentioning. In general, suppose $f_0\in \PH_{c=1}(M)$ is the limit of a sequence $f_n$ in $\PH_{c=1}(M)$ such that $(f_n,\W^c_{f_n})$ is plaque expansive for some invariant center manifold $\W^c_{f_n}$. Consider a metric in $M$, a constant $\delta(f_0)>0$ and, for every $0<\delta\leq \delta(f_0)$, a neighborhood $\U_{\delta}(f_0)$ as in Lemma \ref{rmkUdelta}. 

As $(f_n,\W^c_{f_n})$ is $\delta'$-plaque expansive for every small enough $\delta'>0$ we can consider $\delta_n>0$ the largest constant such that $(f_n,\W^c_{f_n})$ is $\delta'$-plaque expansive for every $\delta'\in (0,\delta_n)$. 

The key point to note is that, a priori, we can not rule out that for every $0<\delta\leq\delta(f_0)$ and $f_n\in \U_\delta(f_0)$ the constant $\delta_n$ may be smaller than $\delta$. Thus a priori we can not conclude that $f_0$ has to admit a center foliation and that there exists a leaf-conjugacy with some $(f_n,\W^c_{f_n})$. To show the closed property Theorem \ref{thmA} and Theorem \ref{thmAprima} an extra `uniform' argument will be needed.
\end{remark}

\subsection{Stability of unique integrability for plaque expansive systems}\label{appendixsubsectionuniqueint} It is also worth noting the following consequences of Theorem \ref{thmgraphtransf} and Theorem \ref{thmunifcontC1completecenterimm}.

\begin{lemma}\label{lemmaintuniq} In the context of Theorem \ref{thmgraphtransf} and Theorem \ref{thmunifcontC1completecenterimm} suppose $0<\delta\leq \delta_0(f_0)$ and $f,g\in \U_\delta(f_0)$. If $E^c_f$ is uniquely integrable then for every $C^1$ curve $\gamma$ tangent to $E^c_{g}$ there exists $L\in \W^c_f$ such that $\gamma\subset h(L)$.
\end{lemma}
\begin{proof}
Suppose $\gamma:(0,1)\to M$ is a $C^1$ curve tangent to $E^c_{g}$. By a little abuse of notation we denote both the curve and its image by $\gamma$. By Peano's existence theorem we can extend $\gamma$, if needed, and redefine its domain so that $\gamma:\R\to M$ is a complete $C^1$ immersion tangent to $E^c_{g}$. Let us see that $\gamma$ needs to be contained in $h(L)$ for some leaf $L\in \W^c$. 

By Theorem \ref{thmunifcontC1completecenterimm} (with the names of $f$ and $g$, and the etas and gammas, interchanged) there exists a sequence $\eta_n:\R \to M$ of complete $C^1$ immersions tangent to $E^c_f$ such that $\eta_{n+1}$ is a reparametrization of $f\circ\eta_n$ for every $n\in \mathbb{Z}$ and \begin{equation}\label{equsubsuniqint} d(g^n\circ \gamma(t),\eta_n(t))<\delta\end{equation} for every $t\in \R$ and $n\in\mathbb{Z}$.

Since $E^c_f$ is uniquely integrable the key observation to note is that each $\eta_n$ needs to be the $C^1$ parametrization of a leaf of $\W^c$ (as these are the only $C^1$ curves tangent to $E^c_f$). If $L$ denote the leaf of $\W^c_f$ whose parametrization is $\eta_0:\R \to M$, let us see that $\gamma$ must be contained in the continuation $h(L)$ of $L$.

On the one hand, as $\eta_{n+1}$ is a reparametrization of $f\circ \eta_n$ then (\ref{equsubsuniqint}) implies that $g^n\circ \gamma$ can be reparametrized to a $C^1$ curve $\gamma_n$ satisfying that $$d(f^n\circ \eta_0(t),\gamma_n(t))<\delta$$ for every $t\in \R$ and $n\in\mathbb{Z}$. It is immediate to check that, in addition, the curve $\gamma_{n+1}$ is a reparametrization of $g\circ \gamma_n$ for every $n\in \mathbb{Z}$.

On the other hand, since $h\circ f^n(L)=g^n\circ h (L)$ for every $n\in \mathbb{Z}$ and $h$ is $\delta$-close to the identity the curves $\gamma_n':=h\circ \eta_n$ satisfy that $\gamma_{n+1}'$ is a reparametrization of $g\circ \gamma_n'$ for every $n
\in \mathbb{Z}$ and $$d(f^n\circ \eta_0(t),\gamma_n'(t))<\delta$$ for every $t\in \R$ and $n\in\mathbb{Z}$.

By the uniqueness part of Theorem \ref{thmunifcontC1completecenterimm} (for $f$ and $g$ not interchanged) it follows that $\gamma$ is a reparametrization of $h\circ \eta_0$. In particular, $\gamma$ is contained in $h(L)$ for $L\in \W^c_f$ the image of $\eta_0$.

\end{proof}

As an immediate consequence of Lemma \ref{lemmaintuniq} one gets the following.

\begin{cor}\label{corintuniq}In the context of Theorem \ref{thmgraphtransf}, if $h$ is a homeomorphism and $E^c_f$ is uniquely integrable then $E^c_{g}$ is uniquely integrable.
\end{cor}
\begin{proof}
If $h$ is a homeomorphism then $h(\W^c_f)$ is a center foliation for $g$. If $\gamma$ is a $C^1$ curve tangent to $E^c_{g}$ then by Lemma \ref{lemmaintuniq} it has to be contained in a leaf of $h(\W^c_f)$. We conclude that through every point of $M$ there exists a unique $C^1$ curve tangent to $E^c_{g}$, modulo reparametrizations.
\end{proof}

As a consequence of the previous corollary one can show the following proposition.

\begin{prop}\label{propstableuniqint}
Let $(f,\W^c)$ be a plaque expansive system in $\PH_{c=1}(M)$.
There exists a $C^1$ neighborhood $\U\subset \PH_{c=1}(M)$ of $f$ such that, if $E^c_g$ is uniquely integrable for some $g\in \U$, then $E^c_{g'}$ is uniquely integrable for every $g'\in \U$.
\end{prop}
\begin{proof}

Suppose $(f,\W^c)$ is a plaque expansive system in $\PH_{c=1}(M)$. Consider $\delta(f)>0$ and a metric in $M$ as in Lemma \ref{rmkUdelta} for $f_0=f$. 

Consider $\U\subset \PH_{c=1}(M)$ a $C^1$ neighborhood of $f$ and $\delta'>0$ given by Corollary \ref{lemmaplaqueexpopen} and Remark \ref{rmkuniformdelta'plaqexp} so that every $g\in \U$ is $\delta'$-plaque expansive.

We can suppose without loss of generality that $\delta'\leq \delta(f)$. Consider $\U_{\delta'}(f)\subset \PH_{c=1}(M)$ the $C^1$ neighborhood of $f$ given by Lemma \ref{rmkUdelta} with respect to the metric already fixed. Consider $\U'=\U_{\delta'}(f)\cap \U$. Let us see that if $\U'$ contains a systems with uniquely integrable center bundle then every system in $\U'$ has this property.

Suppose $E^c_g$ is uniquely integrable for some $g\in \U'$ and let $\W^c_g$ denote the corresponding center foliation for $g$. As $\U'\subset \U_{\delta'}(f)$,	 for $g'\in \U'$ we can consider $h:M\to M$ given by Theorem \ref{thmgraphtransf} so that $h(L)$ is a complete $C^1$ immersion tangent to $E^c_{g'}$ for every $L\in \W^c_g$. Since $(g,\W^c_g)$ is $\delta'$-plaque expansive, $h$ needs to be a homeomorphism (Lemma \ref{lemmadeltaplaqueexp=>leafconj}). By Corollary \ref{corintuniq} we conclude that $E^c_{g'}$ has to be uniquely integrable.
\end{proof}

\subsection{Proof of Theorem \ref{thmgraphtransf}} From now on throughout this subsection let us fix a metric in $M$, a $C^1$ open set $\U_{\delta}(f_0)$  and a pair of partially hyperbolic diffeomorphisms $f,g\in\U_{\delta}(f_0)$ as in the hypothesis of Theorem \ref{thmgraphtransf}. Let $0<\lambda<1$ and $\kappa>1$ denote the constants, and $\mathcal{C}^\sigma$ for every $\sigma\in \{s,u,cs,cu\}$ the invariant cone fields, given by Lemma \ref{rmkUdelta}. Note that $f$ and $g$ satisfy properties (P1),\ldots , (P5) from Lemma \ref{rmkUdelta}. We will refer to properties (P1),\ldots , (P5) implicitly referring to the ones from Lemma \ref{rmkUdelta}.

Informally, for every leaf $L$ of $\W^c$ we will consider $U(L)$ an `unfolded' $\delta$-wide tubular neighborhood of $L$ (see next subsection for the formal construction) and a manifold $V(L)$ which is the disjoint union of the manifolds $U(f^n(L))$ for every integer $n$. As $f$ and $g$ are $C^0$ close enough we will be able to `lift' the map $g$ to $V(L)$ in a neighborhood of $\bigcup_n f^nL$, sending points of each connected component $U(f^nL)$ to the `next' connected component $U(f^{n+1}L)$. By `transverse hyperbolicity' and the constraints imposed by Lemma \ref{rmkUdelta} there will exists a non empty set $L'$ in $U(L)$ whose points are exactly those ones whose $g$ orbit remains in $V(L)$ for every backwards and forwards iterate. We will call $L'$ the \emph{continuation} of $L$.

The set of points in $U(L)$ whose whole $g$ backwards orbit remains in $V(L)$ will be obtained as the limit set in $n$ of the `$cu$-strips' $\W^u_\delta(f^{-n}L)$ iterated $n$ times forwards by  $g$, where $\W^u$ stands for the unstable foliation of $f$. As $\W^u_\delta(f^{-n}L)$ is tangent to the $cu$-cone and $g$ contracts uniformly this cone for positive iterates, the limit set would be a $C^1$ submanifold tangent to $E^c_{g}\oplus E^u_{g}$. The same argument shows that the points whose $g$ forwards orbit is well defined in $V(L)$ is a $C^1$ strip tangent $E^s_{g}\oplus E^c_{g}$. Hence $L'$, the intersection of both sets, would be tangent to $E^c_{g}$.

Once the continuation of every center leaf has been constructed it will remain to define the maps $h$ and $\rho$ that coherently identify each leaf $L$ with it continuation $L'$ so that the identity $h\circ\rho\circ f=g\circ h$ holds.

\subsubsection{Good cover of every center leaf}\label{subsubsectiongoodcover}

For every leaf $L$ of $\W^c$ let us consider the set which is the disjoint union of the balls $\{B_\delta(x)\}_{x\in L}$. Namely $\bigcup_{x\in L}\{(y,x):y\in B_\delta(x)\}$. On this set let us identify two points $(y,x)$ and $(y',x')$ if and only if $y=y'$ and  $x'\in L_{6\delta}(x)$. (Recall the notation $L_r(x)$ for the points in the leaf $L$ at intrinsic distance less than $r$ from $x$). Note that by property (P4) every ball of radius $10\delta$ lies in a foliation box neighborhood of $\W^c$ and if two balls $B_\delta(x)$ and $B_\delta(x')$ intersect for some $x'\in L_{6\delta}(x)$ then $x'$ must lie in $L_{3\delta}(x)$ (from which the transitive property for the above equivalence relation is derived straightforwardly). Let us denote by $U(L)$ the space obtained after the above identification.

Let us see how $U(L)$ has a natural structure of (open) Riemannian manifold. The space $U(L)$ has a natural projection $\pi:U(L)\to M$, defined explicitly by $\pi(y,x)=x$. For every $x\in L$ let $\phi_x:B_\delta(x)\to U(L)$ be such that $\pi\circ \phi_x$ is the identity in $B_\delta(x)$. The topology given to $U(L)$ is the one such that every open set in $U(L)$ is of the form $\bigcup_{x\in F} \phi_{x}(O_x)$, where each $O_x$ is an open subset of $B_\delta(x)$ and $F$ is any subset of $L$. By considering $\{x_n\}_{n\geq 0}$ a countable dense subset of $L$ for the intrinsic topology in $L$, and $\{O_{x_n}^m\}_{m\geq 0}$ a countable base of $B_\delta(x_n)$ for each $n\geq 0$, one obtains that $\{\phi_{x_n}(O_{x_n}^m)\}_{n\geq 0,m\geq 0}$ is a countable base of $U(L)$.

A differentiable atlas for $U(L)$ is given by $\{\phi_x\}_{x\in L}$. The transition functions for this atlas are identity maps in sets of the form $B_\delta(x)\cap B_\delta(x')$ for $x'\in L_{6\delta}(x)$. Finally, the Riemannian metric in $U(L)$ is obtained by taking push-forward of the metric in $M$ by the maps $\{\phi_x\}_{x\in L}$.

As informally stated before, let $V(L)$ be the manifold which is the disjoint union of the manifolds $U(f^n(L))$ for every integer $n$. Note that in the case $L$ is periodic by $f$ of period $N>0$ then $V(L)$ has exactly $N$ connected components, namely  $\{U(f^nL)\}_{0\leq n\leq N-1}$. Otherwise $V(L)$ has countable connected components, namely $\{U(f^nL)\}_{n\in \mathbb{Z}}$. Note also that the projection $\pi:V(L)\to M$ is well defined as it is well defined on each connected component.

For every $\epsilon<\delta$ let us denote by $U_{\epsilon}(L)$  the subset of $U(L)$ given by the points at distance less than $\epsilon$ from $L$. That is, $U_{\epsilon}(L)=\bigcup_{x\in L}B_\epsilon(x)\subset U(L)$. Accordingly let $V_{\epsilon}(L)$ be the subset of $V(L)$ whose connected components are $\{U_\epsilon(f^nL)\}_{n\in \mathbb{Z}}$.

Recall that the $C^0$ distance $d_0(f,g)$ is smaller than $\delta/2$ by property (P5). Recall also that by property (P1) the constant $\kappa>1$ satisfies the condition $\max\{\|Dg_x\|,\|Dg^{-1}_x\|\}<\kappa$ for every $g\in \U_\delta$. Let us fix from now on $\delta_1=\frac{\delta}{2\kappa}$.

\begin{af}\label{claimlifts} For every leaf $L$ of $\W^c$ the maps $f$ and $g$ lift to maps $$f,g:V_{\delta_1}(L)\to V(L)$$
such that the connected component $U_{\delta_1}(f^nL)$ is sent by $f$ and by $g$ inside $U_\delta(f^{n+1}L)$, and is sent by $f^{-1}$ and $(g)^{-1}$ inside $U_\delta(f^{n-1}L)$, for every $n\in \mathbb{Z}$. 
\end{af}
\begin{proof}
First of all, note that $f$ lifts directly to $\bigcup_{n\in\mathbb{Z}}f^nL\subset V_{\delta_1}(L)$.

For $y\in U_{\delta_1}(f^nL)$ let $x$ be a point in $f^nL\subset U_{\delta_1}(f^nL)$ such that $d(x,y)<\delta_1$. Let us denote $y'=\pi(y)$ and $x'=\pi(x)$.
Since $y'\in B_{\delta_1}(x')$, one has $d(f(y'),f(x'))<\kappa\delta_1$. Moreover, as $f$ and $g$ are $\frac{\delta}{2}$-close by property (P5) then $d(f(y'),g(y'))<\delta/2$. We conclude that $f(y')$ and $g(y')$ lie $B_\delta(f(x'))$ as $\delta_1+\delta/2<\delta$. 

As $\pi$ is bijective from $B_\delta(f(x'))$ to $B_\delta(f(x))$, the points $f(y)$ and $g(y)$ can be lifted to $B_\delta(f(x))$ to points $f(y')$ and $g(y')$, respectively. In this way, it is easy to check that $f$ and $g$ are well defined $C^1$ maps from $V_{\delta_1}(L)\to V(L)$.
\end{proof}

The proof of the theorem is going to show that $h(L)$, the \emph{continuation of $L$}, will be the projection by $\pi$ of set of points in $U_{\delta_1}(L)$ whose $g$ orbit in $V(L)$ is well defined for every future and past iterate (see Remark \ref{rmkcharactcontinuation}).

\vspace{0.15cm}
\noindent \emph{Notations.} We will denote by $E^\sigma$ and $E^{\sigma,g}$ the $f$-invariant and $g$-invariant bundles in $M$, respectively, for every $\sigma \in \{s,c,u,cs,cu\}$. Analogously for the $f$ and $g$-invariant foliations $\W^{\sigma}$ and $\W^{\sigma,g}$. Note that we can lift these bundles and leaves to $V(L)$. Let us denote these lifted bundles as $\tilde{E}^\sigma$ and $\tilde{E}^{\sigma,g}$, and the lifted foliations as $\tilde{\W}^\sigma$ and $\tilde{\W}^{\sigma,g}$, respectively. Note that they are (locally) invariant wherever $f$ and $g$ are well defined. The same for the $f$ and $g$-invariant cone-fields $\mathcal{C}^\sigma$ lifting to cone-fields $\tilde{\mathcal{C}}^\sigma$.
\vspace{0.15cm}

\subsubsection{Graph transform for $cu$-strips}

Let us fix from now on the constant $\delta_2>0$ such that $\delta_2=\frac{\delta_1}{2}=\frac{\delta}{4\kappa}$.

For every leaf $L$ of $\W^c$ and every $\epsilon\leq \delta_2$ let us define $U^{su}_{\epsilon}(L)$ in $U(L)$ as $$U^{su}_{\epsilon}(L)=\tilde{\W}^{s,g}_{\epsilon}(\tilde{\W}^u_{\epsilon}(L)).$$ Note that the unstable plaques are considered with respect to $f$ and the stable plaques with respect to $g$. This is not essential but will make some arguments simpler.

Recall that by Lemma \ref{lemmaWcWsC1} the sets $\tilde{\W}^u_{\epsilon}(L)$ are $C^1$ submanifolds tangent to $\tilde{E}^{cu}$. These sets are what we call \emph{$cu$-strips}. As the $g$ stable local manifolds  $\tilde{\W}^{s,g}_{\epsilon}(x)$ are transverse to $\tilde{\W}^u_{\epsilon}(L)$ for every $x\in \tilde{\W}^u_{\epsilon}(L)$ it is easy to check that $U^{su}_{\epsilon}(L)$ is an open subset of $U(L)$.

Note that $U^{su}_{\epsilon}(L)$ is a subset of $U_{2\epsilon}(L)$ since every point in $U^{su}_{\epsilon}(L)$ can be joined to a point in $L$ by a concatenation of an $g$-stable and an $f$-unstable arc of lengths less than $\epsilon$.  Moreover, by property (P4) it follows that $U_{\epsilon/2}(L)$ is contained in $U^{su}_{\epsilon}(L)$.

Let us define $V^{su}_{\epsilon}(L)$ as the subset of $V_{\delta_1}(L)$ which is the union of the sets $U^{su}_{\epsilon}(f^nL)$ for every integer $n$. And let us define $$\pi^s:V^{su}_{\delta_2}(L)\to \bigcup_{n\in \mathbb{Z}} \tilde{\W}^u_{\delta_2}(f^nL)$$ the projection along local stable $g$-plaques.

Let $\delta_3>0$ be the constant $\delta_3=\frac{\delta_2}{4\kappa}$. Recall that the $C^0$ distance $d_0(f,g)$ is smaller than $\frac{\delta_2}{2}$ by (P5). By the same arguments as in Claim \ref{claimlifts}, the image by $g$ of $U_{2\delta_3}(f^nL)$ is a subset of $U_{\delta_2}(f^{n+1}L)$ for every $n\in \mathbb{Z}$. Because  $U^{su}_{\delta_3}(f^{n+1}L)$ is contained in $U_{2\delta_3}(f^nL)$, we have the following:

\begin{remark} The map $g$ from $V^{su}_{\delta_3}(L)$ to $V^{su}_{\delta_2}(L)$ is well defined.
\end{remark}

Let us consider the set of continuous functions
$$\Pi^{cu}(L)=\{ \xi:
\bigcup_{n\in \mathbb{Z}} \tilde{\W}^u_{\delta_3}(f^nL) \to V^{su}_{\delta_2}(L) \mbox{ such that }\pi^s\circ \xi= \id\}.$$ 
Note that if $V(L)$ has $N>0$ (resp. countably many) connected components then $\xi\in \Pi^{cu}(L)$ is given by functions $\xi|_{\tilde{\W}^u_{\delta_3}(f^nL)}:\tilde{\W}^u_{\delta_3}(f^nL) \to U^{su}_{\delta_2}(f^nL)$ for each $0\leq n \leq N-1$ (resp. for each $n\in\mathbb{Z}$).

Given two maps $\xi,\xi'$ in $\Pi^{cu}(L)$ we can define a distance between them $$d(\xi,\xi')=\sup d_s(\xi(x),\xi'(x))$$ where $d_s$ denotes the distance inside the plaque $\WW^{s,g}_{\delta_2}(x)$ and the supremum is taken over all $x$ in $\bigcup_{n\in\mathbb{Z}}\tilde{\W}^u_{\delta_3}(f^nL)$. 

The \emph{zero-section} is the function $\xi^0$ in $\Pi^{cu}(L)$ defined by $\xi^0(x)=x$ for every $x$. For every $\xi$ in $\Pi^{cu}(L)$ we denote by $\graph (\xi)$ the set which is the image of $\xi$. For simplicity, let us denote from now on by $\delta'$ the $C^0$ distance $d_0(f,g)$.

\begin{af}[Graph transform]\label{claimgraphtransf} \noindent

\begin{enumerate} 
\item\label{claimgraphtransf_i} The image by $g$ of $\graph(\xi^0)$ induces a new map $g\xi^0$ in $\Pi^{cu}(L)$ such that $\graph(g\xi^0) \subset g\graph(\xi^0)$ and $d(\xi^0,g\xi^0)<2\delta'$. 

\item\label{claimgraphtransf_ii} Moreover, for every $\xi$ in $\Pi^{cu}(L)$ such that $d(\xi^0,\xi)< \delta_3$ the image by $g$ of $\graph(\xi)$ induces a new map $g\xi$ in $\Pi^{cu}(L)$ such that $\graph(g\xi)\subset g\graph(\xi)$ and  $d(\xi^0,g\xi)<2\delta'+\lambda d(\xi^0,\xi)$.

\item\label{claimgraphtransf_iii} Finally, for every $\xi,\xi'$ in $\Pi^{cu}(L)$ with $d(\xi,\xi^0)<\delta_3$ and $d(\xi',\xi^0)<\delta_3$ we have $d(g\xi,g\xi')<\lambda d(\xi,\xi')$.
\end{enumerate}
\end{af}

\begin{proof}
Let us start by looking at the image by $g$ of the zero section $\xi^0$. Recall that $\graph(\xi^0)$ is the union in $n\in \mathbb{Z}$ of the $C^1$ submanifolds $\tilde{\W}^u_{\delta_3}(f^nL)$. Let $n$ be any fixed integer. Let us see that the image by $\pi^s\circ g$ of $\tilde{\W}^u_{\delta_3}(f^nL)$ covers $\tilde{\W}^u_{\delta_3}(f^{n+1}L)$, and that $\pi^s\circ g$ restricted to $\tilde{\W}^u_{\delta_3}(f^nL)$ is injective. As a consequence the map $g\xi^0$ at any point $y\in \tilde{\W}^u_{\delta_3}(f^{n+1}L)$ will be unambiguously defined as the unique point in the image by $g$ of $\tilde{\W}^u_{\delta_3}(f^nL)$  whose projection by $\pi^s$ is $y$. It will be clear from the construction that $g\xi^0(y)$ defined in this way will vary continuously with $y$.

Note that by property (P2) the set $\tilde{\W}^u_{\lambda^{-1}\delta_3}(f^{n+1}L)$ is contained in the image by $f$ of  $\tilde{\W}^u_{\delta_3}(f^nL)$. Thus for every $y\in \tilde{\W}^u_{\lambda^{-1}\delta_3}(f^{n+1}L)$ there exists $y'$ in $\tilde{\W}^u_{\delta_3}(f^nL)$ such that $f(y')=y$. 

As $\tilde{\W}^u_{\delta_3}(f^nL)$ is a $C^1$ submanifold tangent to the cone field  $\tilde{\C}^{cu}$ it follows that its image by $g$ is also a $C^1$ submanifold tangent to $\tilde{\C}^{cu}$. By property (P4) 	it follows that $\pi^s\circ g$ has to be injective restricted to $\tilde{\W}^u_{\delta_3}(f^nL)$.

Since $f(y')$ and $g(y')$ are at distance less than $\delta'=d_0(f,g)$, again by property (P4) it follows that $\tilde{\W}^{s,g}_{2\delta'}(g(y'))$ and $\tilde{\W}^{cu}_{2\delta'}(y)$ intersect. In particular $\pi^s\circ g(y')$ and $y$ need to be at distance less than $2\delta'$ for the intrinsic metric of $\tilde{\W}^u_{\delta_2}(f^{n+1}L)$. We conclude that $\pi^s\circ g\circ f^{-1}$ is a well defined continuous and injective function from $\tilde{\W}^u_{\lambda^{-1}\delta_3}(f^{n+1}L)$ to $\tilde{\W}^u_{\delta_2}(f^{n+1}L)$ that is $2\delta'$-close to the identity.

For every $y\in \tilde{\W}^u_{\delta_3}(f^{n+1}L)$ the ball of radius $10\delta'$ in $\tilde{\W}^u_{\delta_2}(f^{n+1}L)$ is contained in $\tilde{\W}^u_{\lambda^{-1}\delta_3}(f^nL)$ by property (P5). By a standard topology argument using that $\pi^s\circ g \circ f$ is $2\delta'$-close to the identity we obtain that $y$ needs to be in the image of this ball. So the image by $\pi_s\circ g$ of $\tilde{\W}^u_{\delta_3}(f^nL)$ covers $\tilde{\W}^u_{\delta_3}(f^{n+1}L)$ as we wanted to prove. This settles (\ref{claimgraphtransf_i}).

In order to see (\ref{claimgraphtransf_ii}) suppose $\xi$ is not the zero section but $d(\xi^0,\xi)<\delta_3$. For simplicity let $d$ denote $d(\xi^0,\xi)$. For every $w$ in $\tilde{\W}^u_{\delta_3}(f^nL)$ the point $\xi(w)$ lies in $\tilde{\W}^{s,g}_d(w)$ so $g\circ \xi(w)$ needs to lie in  $\tilde{\W}^{s,g}_{\lambda d }(g(w))$. Moreover, as seen before, the point $g(w)$ lies in $\W^{s,g}_{2\delta'}(\pi^s\circ g(w))$. It follows that $g\circ \xi(w)$ lies in $\W^{s,g}_{2\delta'+\lambda d}(\pi^s\circ g(w))$. 

As the image of $\pi^s\circ g\circ \xi$ coincides with that of $\pi^s\circ g$ it follows that $\graph g\circ \xi$ defines a function  $g\xi$ in $\Pi^{cu}(L)$ such that $d(\xi^0,g\xi)<2\delta' + \lambda d(\xi^0,\xi)$. This proves (\ref{claimgraphtransf_ii}).

Finally, (\ref{claimgraphtransf_iii}) follows immediately from the previous arguments.
\end{proof}

\noindent \emph{Notations.} Let us denote $g(g\xi^0)$ by $g^2\xi^0$ and, inductively, $g(g^n\xi^0)$ by $g^{n+1}\xi^0$ for every $n>0$.
\vspace{0.15cm} 

From (\ref{claimgraphtransf_i}) and (\ref{claimgraphtransf_ii}) of the previous claim it follows that $d(\xi^0,g\xi^0)<2\delta'$, then $d(\xi^0,g^2\xi^0)<2\delta'+\lambda 2\delta'=2\delta'(1+\lambda)$, and inductively $$d(\xi^0,g^n\xi^0)<2\delta'(1+\lambda+\cdots  +\lambda^{n-1})
$$ for every $n>0$. Note that $g^n\xi^0\in\Pi^{cu}(L)$ is well defined for every $n>0$ since $\delta'=d_0(f,g)$ satisfies $\delta'(1+\lambda+\cdots )<\frac{\delta}{64\kappa^2}<\delta_3/2$ by property (P5). 

Moreover, by (\ref{claimgraphtransf_iii}) of the previous claim it follows from $d(\xi^0,g\xi^0)<2\delta'$ that $d(g\xi^0,g^2\xi^0)<2\delta'\lambda$, and inductively $$d(g^n\xi^0,g^{n+1}\xi^0)<2\delta'\lambda^n
$$ for every $n>0$. 

Hence we obtain a well defined limit function $\xi^\infty\in \Pi^{cu}(L)$ given by $$\xi^\infty(x):=\lim_n g^n\xi^0(x)$$
for every $x\in \bigcup_{n\in \mathbb{Z}} \tilde{\W}^u_{\delta_3}(f^nL)$. Clearly $\xi^\infty$ satisfies $d(\xi^0,\xi^\infty)\leq 2\delta'(1+\lambda+\ldots) <\delta_3/2$. 

Moreover, note that $g\xi^\infty=\xi^\infty$
since the image by $g$ of $g^n\xi^0(x)$ coincides with $g^{n+1}\xi^0(\pi^s\circ f(x))$ and the image by $g$ of $\lim_n g^n\xi^0(x)$ coincides with $\lim_n g^{n+1}\xi^0(\pi^s\circ f(x))$. In particular $$\graph{g\xi^\infty}\subset g(\graph{\xi^\infty}).$$
As $g^{-1}$ expands $g$-stable arcs uniformly, the points in $\graph(\xi^\infty)$ are precisely the points in $V_{\delta_3}(L)$ whose $g$ backwards orbit is well defined for every past iterate in $V_{\delta_3}(L)$.

\begin{af}\label{claimgraphC1} The set $\graph\xi^\infty$  is a $C^1$-submanifold tangent to $\tilde{E}^{cu}_{g}$. 
\end{af}
\begin{proof}
We will make a local argument near every $x$ in $L$. Let us consider the local exponential map $\exp_x:B^{cu}_\delta\times B^{s}_\delta\subset T_xM \to M$ where $B^{cu}_\delta$ and $B^{s}_\delta$ denote the balls of center $x$ and radius $\delta$ in $\tilde{E}^{cu}(x)$ and $\tilde{E}^s(x)$, respectively.

Let $\tilde{E}^{cu}$ and $\tilde{\C}^{cu}$ denote the pull-back by $\exp_x$ of the bundle $\tilde{E}^{cu}_{g}$ and the cone field $\tilde{\C}^{cu}$.

Let $S_n\subset B^{cu}_\delta\times B^{s}_\delta$ denote the preimage by $\exp_x$ of $\graph g^n\xi^0$ for every $n>0$. Since $\graph g^n\xi^0$ is a $C^1$ submanifold tangent to the cone field $g^n\C^{cu}$, by property (P4) there exists $\epsilon>0$ small enough so that the sets $\{z\}\times B^{s}_\delta(0)$ intersects $S_n$ and this intersection point is a unique point for every $z\in B^{cu}_{\epsilon}$.

This defines $C^1$ functions $$\psi_n: B^{cu}_{\epsilon} \to B^s_\delta$$ for every $n>0$ given by  $\psi_n(z):=(\{z\}\times B^{s}_\delta)\cap S_n$.

For every $z \in B^{cu}_\epsilon$ it is immediate to check that the limit $\psi_\infty(z):=\lim_n \psi_n(z)$ exists and defines a function $\psi_\infty:B^{cu}_\epsilon \to B^s_\delta$. Moreover, by property (P3)(3) the sequence $D(\psi_n)_z(\tilde{E}^{cu}(x))$ needs to converge uniformly to $\tilde{E}^{cu}(\psi_\infty(z))$ for $z\in B^{cu}_{\epsilon}$.

We obtain that $\psi_\infty$ is of class $C^1$ and that $D_z\psi_\infty(\tilde{E}^{cu}(x))$ is equal to $\tilde{E}^{cu}_{g}(\psi_\infty(z))$ for every $z\in B^{cu}_\epsilon$  by the following standard fact from multivariable calculus that is a consequence of  Arzelà-Ascoli's theorem:

\begin{itquote}
If $\psi_n:U\subset \R^{d_1}\to \R^{d_2}$ is a sequence of $C^1$ maps defined in an open subset $U\subset  \R^{d_1}$ such that:
\begin{itemize}\item The limit $\psi_\infty(x):=\lim_n\psi_n(x)$ exists for every $x\in U$.

\item The limit $A(x):=\lim_n D_x\psi_n$ given by the rule $(A(x))_{ij}=\lim_n (D_x\psi_n)_{ij}$ exists for every $x\in U$, varies continuously with $x$ and $\sup_{x\in U}\|D_x\psi_n-A(x)\|\xrightarrow{n}0$.
\end{itemize}
Then $\psi_\infty:U \subset\R^{d_1}\to \R^{d_2}$ is a $C^1$ map and $D_x\psi_\infty(x)=A(x)$ for every $x\in U$.
\end{itquote}

This proves that $\graph\xi^\infty$  is a $C^1$-submanifold tangent to $\tilde{E}_{g}^{cu}$. 

\end{proof}

\subsubsection{Construction of $h$ and $\rho$}\label{subsectionhandrho}

For every leaf $L$ of $\W^c$ we have constructed a limit map $\xi^\infty$ in $\Pi^{cu}(L)$ such that $d(\xi^0,\xi^\infty)<\delta_3/2$. As this limit map corresponds to a limit graph for $cu$-strips let us rename it as $\xi^\infty_{cu}$. And let us also rename by $\xi^0_{cu}$ the zero-section $\xi^0$.

Analogously as before one can define neighborhoods $U^{su}_{\delta_3}(f^nL)$ for every $n\in \mathbb{Z}$, a map $\pi^u$, a family of maps $\Pi^{cs}(L)$ and a limit map $\xi^\infty_{cs}$ for $cs$-strips satisfying analogous properties than the $cu$ ones (interchanging the roles of $g$ and $g^{-1}$).

Following Claim \ref{claimgraphC1} we obtain that the intersection $$\graph(\xi_{cs})\cap \graph(\xi_{cu})$$ defines a $C^1$ manifold in $V(L)$ that is $g$-invariant and tangent to $\tilde{E}^{c}_{g}$. Let us denote by $L'$ the connected component of this intersection that lies in $U(L)$, and in general let us denote by $(f^nL)'$ the one that lies in $U(f^nL)$. 

\begin{remark}\label{rmkcharactcontinuation} Note that from the properties of $\xi^\infty_{cs}$ and $\xi^\infty_{cu}$ (see, in particular, the discussion before Claim \ref{claimgraphC1}) the points in $L'$ are characterized as the points in $V(L)$ for which its $g$-orbit is well defined for every future and past iterate. 
\end{remark}

The projection $\pi(L')$ in $M$ is going to be $h(L)$, the continuation of $L$. Let us see how we can construct $h:M\to M$ and $\rho:M\to M$ so that the properties detailed in the statement of the theorem are verified.

For every $L$ in $\W^c$ let us start by defining a map $h_1$ from $L$ to $L'$ in $U(L)$. For every $x\in L$ we define $h_1(x)\in L'$ by $$h_1(x):=\xi^\infty_{cs}\circ  \pi^u \circ \xi^\infty_{cu}(x).$$  In other words, $h_1(x)$ is the unique point in $L'$ such that $\WW^{s,g}_{\delta_3}(x)$ and $\WW^{u,g}_{\delta_3}(h_1(x))$ intersect. As $L'$ is tangent to $\tilde{E}^{c,g}$, the leaf $\WW^{u,g}_{\delta_3}(L')$ is tangent to $\tilde{E}^{cu,g}$ by Lemma \ref{lemmaWcWsC1}. This justifies why the intersection of $\WW^{u,g}_{\delta_3}(L')$ with $\WW^{s,g}_{\delta_3}(x)$ is a unique point.

It is immediate that $h_1$ is continuous. Moreover, by property (P4) it is easy to check that for every $x,y\in L$:
\begin{equation}\label{eq1h1}d_L(x,y)=\delta_3\hspace{0.2cm}\text{ implies }\hspace{0.2cm}\delta_3/2<d_{L'}(h_1(x),h_1(y))<2\delta_3
\end{equation} In particular, $h_1$ continuous and (\ref{eq1h1}) imply that $h_1$ from $L$ to $L'$ is also surjective.

Let us see now what happens when we iterate by $g$. Since $g \graph(\xi_{cs})\subset \graph(\xi_{cs})$ and $g^{-1} \graph(\xi_{cu})\subset \graph(\xi_{cu})$ it follows that $$gL'=(fL)'.$$
Given $x$ in $L$ the point $h_1(x)$ lies in $L'$ and the point $f(x)$ lies in $fL$. Then $g\circ h_1(x)$ and $h_1\circ f(x)$ both lie in $gL'=(fL)'$. We want to justify that the distance between $g\circ h_1(x)$ and $h_1\circ f(x)$ inside $(fL)'$ needs to be small.

Indeed, note first that $d(f(x),g(x))\leq \delta'$ (recall that $\delta'$ denotes $d_0(f,g)$). Then, on the one hand $h_1\circ f(x)$ is given as the unique point in $(fL)'$ such that $\WW^{s,g}_{\delta_3}(f(x))$ and $\WW^{u,g}_{\delta_3}(h_1\circ f(x))$ intersect. On the other hand, $h_1(x)$ is given as the unique point in $L'$ such that $\WW^{s,g}_{\delta_3}(x)$ and $\WW^{u,g}_{\delta_3}(h_1(x))$ intersect, and then by the $g$-invariance of the foliations $\W^{s,g}$ and $\W^{cu,g}$ one obtains that $\WW^{s,g}_{\delta_3}(g(x))$ intersects $\WW^{u,g}_{\delta_3}(g\circ h_1(x))$. That is, $g\circ h_1(x)$ is given as the unique point (unique by the same reasons a before) such that $\WW^{s,g}_{\delta_3}(g(x))$  and $\WW^{u,g}_{\delta_3}(g\circ h_1(x))$ intersect. 

By property (P4) one can derive the following two properties. If two points $z$ and $w$ satisfy $d(z,w)\leq\delta'$ then $\WW^{s,g}_{\delta_3}(z)$ and $\WW^{s,g}_{\delta_3}(w)$ are at Hausdorff distance less than $\frac{3\delta'}{2}$. And if two points $w$ and $z$ lie in $(fL)'$ at distance not smaller than $2\delta'$ then $\WW^{u,g}_{\delta_3}(z)$ and $\WW^{u,g}_{\delta_3}(w)$ are Hausdorff distance greater than $\frac{3\delta'}{2}$.

By applying the two properties above together with the properties that $d(f(x),g(x))\leq\delta'$, that $\WW^{s,g}_{\delta_3}(f(x))$ has no trivial intersection with $\WW^{u,g}_{\delta_3}(h_1\circ f(x))$ and that $\WW^{s,g}_{\delta_3}(g(x))$ has no trivial intersection with $\WW^{u,g}_{\delta_3}(g\circ h_1(x))$, one obtains that 
\begin{equation}\label{eq2h1} g\circ h_1(x)\in (fL)'_{2\delta'}(h_1\circ f(x))\end{equation} for every $x\in L$.

A priori $h_1$ from $L$ to $L'$ may not be injective. However, by a `regulating' process we can rely on $h_1$ to construct the desired $C^1$ diffeomorphism $h$ from $L$ to $L'$. Let $\gamma:\R\to L$ and $\gamma':\R\to L'$ be parametrizations by arc-length and let $\Psi_1:\R \to \R$ denote the map $$\Psi_1(t)=\gamma'^{-1}\circ h_1\circ \gamma(t).$$ We can assume that $L$ and  $L'$ are parametrized with the same orientation, that is, such that $\lim_{t\to+\infty} \Psi_1(t)=+\infty$. Note that by (\ref{eq1h1}) it follows that $$\frac{\delta_3}{2}<\Psi_1(t+\delta_3)-\Psi_1(t)<2\delta_3$$ for every $t\in \R$. If we define $\Psi:\R \to \R$ as $$\Psi(t)=\frac{1}{\delta_3}\int_{t-\frac{\delta_3}{2}}^{t+\frac{\delta_3}{2}}\Psi_1(s)\,ds$$ it follows that the derivative $\D\Psi(t)$ exists everywhere, varies continuously with $t$ and satisfies $\frac{1}{2}<\D\Psi(t)<2$. Defining $h$ as $$h(x)=\gamma' \circ \Psi \circ \gamma^{-1}(x)$$ for every $x\in L$ we conclude that $h|_L:L\to L'$ is a $C^1$ diffeomorphism satisfying $$\frac{1}{2}<\|Dh(v^c)\|<2$$ for every unit vector $v^c$ in $\tilde{E^c}$. 

Moreover by (\ref{eq1h1}) it follows that
$$h(x)\in L'_{2\delta_3}(h_1(x))$$ for every $x\in L$. Since $g\circ h_1 (x)$ lies in $(fL)'_{2\delta'}(h_1\circ f(x))$ by (\ref{eq2h1}) and $2\delta'+2\delta_3<\delta$ then $$g\circ h (x)\in (fL)'_{\delta}(h\circ f(x))$$ for every $x\in L$.  

If we define $\rho:L\to L$ as $$\rho(x)=h^{-1}\circ g \circ h\circ f^{-1}(x)$$ it follows that $\rho$ is a $C^1$ diffeomorphism that is $\delta$-close to the identity and satisfies $$h\circ \rho \circ f(x) = g \circ h(x)$$ for every $x\in L$.

It remains to `descend' $h$ and $\rho$ to $M$. By a little abuse of notation let us denote by $h$ and $\rho$ the maps in $M$ given by $\pi\circ h \circ \pi^{-1}$ and $\pi\circ \rho \circ \pi^{-1}$, respectively. 

All the properties claimed for $h$ and $\rho$ are immediately satisfied except maybe for the ones contained in the following two claims which may require further justification. 

\begin{af} The map $h:M\to M$ is continuous, surjective and $\delta$-close to the identity.
\end{af}
\begin{proof}

It is easy to check that $h$ is $\delta$-close to the identity: Since $\WW^{s,g}_{\delta_3}(x)$ and $\WW^{u,g}_{\delta_3}(h_1(x))$ intersect for every $x$ it follows that $h_1$ is $2\delta_3$ close to the identity. Since $h(x)$ lies in $L_{2\delta_3}(h_1(x))$ we conclude that $h$ is $4\delta_3<\delta$ close to the identity in $M$.

The remaining of the proof is devoted to show that $h$ is continuous. The surjectivity of $h$ is a direct consequence of $h$ continuous and $\delta$-close to the identity 
\footnote{Recall that in a closed Riemannian manifold $X$ every continuous map $F:X\to X$ that is $\delta$-close to the identity for $\delta>0$ small enough needs to be surjective. To show this one can aproximate $F$ by a smooth map $G:X\to X$  that is $C^0$ $\epsilon$-close to it. For $\delta>0$ and $\epsilon>0$ small enough $G$ is homotopic to the identity. Since the identity has degree mod 2 equal to 1, then the same happens to $G$. As as consequence $G$ is surjective. Since $F$ is $\epsilon$-close to a surjective map for every $\epsilon>0$ small enough then $F$ itself needs to be surjective. See for example \cite{Mil}.}.

Note first that $h_1$ also descends naturally to $h_1:M\to M$ and that if $h_1$ is continuous in $M$ then $h$ will also be continuous as the regulating process has to preserve continuity. So we will show the continuity of $h_1$.

The map $h_1$ has been defined by means of maps $\xi^\infty_{cs}$, $\pi^u$ and $\xi^\infty_{cu}$ depending on the `unfolded' tubular neighborhoods $U(L)$ for each leaf $L\in \W^c$. These neighborhoods are a priori disjoint for different leaves of $\W^c$. We need to somehow merge them in $M$ to be able to compare them.

For every $x\in M$ let $L(x)$ denote the leaf of $\W^c$ through $x$. 
Let us define the map $\xi^\infty_{cu,x}:\W^u_{\delta_3}(L_{\delta_3}(x))\subset M\to M$ as the map such that $\xi^\infty_{cu}|_{\WW^u_{\delta_3}(L_{\delta_3}(x))}$ in $U^{su}_{\delta_3}(L(x))$ is a lift of it. Analogously we define the map $\xi^\infty_{cs,x}:\W^s_{\delta_3}(L_{\delta_3}(x))\to M$ for every $x\in M$.

Let $\pi^u_{x}: B_{\delta_3/2}(x)\to \W^s_{\delta_3}(L_{\delta_3}(x))$ be such that $\pi^u_{x}(z)$ is the intersection of $\W^{u,g}_{\delta_3}(z)$ with $\W^s_{\delta_3}(L_{\delta_3}(x))$ for every $z\in B_{\delta_3/2}(x)$. Again, $\pi^u|_{B_{\delta_3/2}(x)}$ in $U^{su}_{\delta_3}(L(x))$ is a lift of $\pi^u_x$. 

For every $x\in M$ we have $$h_1(x)=\xi^\infty_{cs,x} \circ \pi^u_x \circ \xi^\infty_{cu,x}(x).$$ Let us see that if $x_n\xrightarrow{n} x$ then $$\xi^\infty_{cs,x_n} \circ \pi^u_{x_n} \circ \xi^\infty_{cu,x_n}(x_n)\xrightarrow{n} \xi^\infty_{cs,x} \circ \pi^u_x \circ \xi^\infty_{cu,x}(x).$$

It will be enough to show that:
\begin{enumerate}
\item\label{hcont1} If $x_n\xrightarrow{n} x$ and $y_n\xrightarrow{n} y$ then $\pi^u_{x_n}(y_n)\xrightarrow{n} \pi^u_x(y)$.
\item\label{hcont2} If $x_n\xrightarrow{n}  x$ and $z_n\xrightarrow{n}  z$ with $z_n\in \W^u_{\delta_3}(\W^c_{\delta_3}(x_n))$ and $z\in \W^u_{\delta_3}(\W^c_{\delta_3}(x))$, then $\xi^\infty_{cu,x_n}(z_n)\xrightarrow{n}  \xi^\infty_{cu,x}(z)$. Analogous property hold for $cs$ maps.
\end{enumerate}

Let us see first why (\ref{hcont1}) and (\ref{hcont2}) are enough for proving $h_1$ continuous. Indeed, if $x_n \xrightarrow{n} x$ in $M$ then by (\ref{hcont2}) for $x_n=z_n$ and $x=z$ it follows that $\xi^\infty_{cu,x_n}(x_n)\xrightarrow{n} \xi^\infty_{cu,x}(x)$. Then $\pi^u_{x_n}\circ \xi^\infty_{cu,x_n}(x_n)$ converges with $n$ to $\pi^u_x\circ \xi^\infty_{cu,x}(x)$ by (\ref{hcont1}). And again by (\ref{hcont2}) for $cs$ maps with $z_n=\pi^u_{x_n}\circ \xi^\infty_{cu,x_n}(x_n)$ and $z=\pi^u_x\circ \xi^\infty_{cu,x}(x)$ we conclude that $\xi^\infty_{cs,x_n}\circ \pi^u_{x_n}\circ \xi^\infty_{cu,x_n}(x_n)$ converges with $n$ to $\xi^\infty_{cs,x} \circ \pi^u_x\circ \xi^\infty_{cu,x}(x)$.

The proof of (\ref{hcont1}) is immediate by the regularity of the foliations $\W^{u,g}$ and $\W^{cs}$.

The remaining is devoted to showing (\ref{hcont2}). Informally, the key property we will use is that, by the regularity of $\W^c$, for every $R>0$ and $\mu>0$ the sets $L_R(x)$ and $L_R(x_n)$ are at Hausdorff distance less than $\mu$ for every large enough $n$. This will enable us to `lift' to $U(L(x))$ long pieces of the leaf $L(x_n)$ and to `see' in $U(L(x))$ the first iterates of the $cu$ graph transform for $L(x_n)$.

Suppose from now on $x_n\xrightarrow{n} x$ and $z_n\xrightarrow{n} z$ with $z_n\in \W^u_{\delta_3}(\W^c_{\delta_3}(x_n))$ and $z\in \W^u_{\delta_3}(\W^c_{\delta_3}(x))$. Given $\epsilon>0$ let us see that $d(\xi^\infty_{cu,x_n}(z_n), \xi^\infty_{cu,x}(z))<\epsilon$ for every $n$ large enough. The proof for $cs$ maps is analogous. 

Let us assume without loss of generality that $\epsilon<\delta_3/2$. Recall the inclusions $U_{\delta_3/2}(L(y))\subset U^{su}_{\delta_3}(L(y))\subset U^{su}_{2\delta_3}(L(y))\subset U_{\delta_2}(L(y))$ for every $y\in M$.

For every $n$ large enough the point $x_n$ lies in $B_{\delta_3/2}(x)\subset M$ so we can lift it to $B_{\delta_3/2}(x)\subset U_{\delta_3/2}(L(x))$. For simplicity, let us call these lifts of $x_n$ with the same name, $x_n$.

For every $y\in M$ recall that $\tilde{\W}^u_{\delta_3}(L(y))$ denotes the $cu$-strip in $U^{su}_{\delta_3}(L(y))$. For every $R>0$ let $\tilde{\W}^u_{\delta_3}(L_R(y))$ denote the `truncated' $cu$-strip that is the set $\bigcup_{z\in L_R(y)}\tilde{\W}^u_{\delta_3}(z)$. By the regularity of the foliations $\W^c$ and $\W^u$ the following is immediate to check:

\begin{remark}\label{rmkR>0truncated} Suppose $R>0$ and $\mu>0$. For every $x_n$ close enough to $x$  the projection to $M$ of the truncated $cu$-strip $\tilde{\W}^u_{\delta_3}(L_R(x_n))$ can be lifted to be a subset of $U^{su}_{2\delta_3}(L(x))$ that is at Hausdorff distance less than $\mu$ from $\tilde{\W}^u_{\delta_3}(L_R(x))$ and such that $x_n$ lifts close to $x$ in $U^{su}_{2\delta_3}(L(x))$.
\end{remark}

Recall that for every $y\in M$ the map $\xi^\infty_{cu}$ in $\Pi^{cu}(L(y))$ is defined as a limit of the maps $g^k\xi^0_{cu}$ for $\xi^0_{cu}$ the zero-section in $\Pi^{cu}(L(y))$. Let us denote the zero section as $\xi^0_{cu,y}$ and the maps $g^k\xi^0$ as $\xi^k_{cu,y}$ to highlight the dependence on the point $y$. 

Recall that by Claim \ref{claimgraphtransf} it follows that  $$d\big(\xi^k_{cu,y},\xi^\infty_{cu,y}\big)<2\delta'(\lambda^k+\lambda^{k+1}+\ldots)$$ for every $k>0$. Hence for every $\mu>0$ there exists $K>0$ independent of $y$ such that
$$d\big(\xi^k_{cu,y},\xi^\infty_{cu,y}\big)<\mu$$ for every $k \geq K$. 

Let $K_0>0$ be such that \begin{equation}\label{hcont(a)} 2\delta'(\lambda^{K_0}+\lambda^{K_0+1}+\ldots)<\epsilon/6.\end{equation} 

For every $w\in M$ and $z\in \WW^u_{\delta_3}(L(w))$ let $z^{-1}$ denote the point in the $cu$-strip of $f^{-1}L(z)$ such that $g(z^{-1})=\xi^1_{cu,w}(z)$. Analogously, let $z^{-k}$ be the point in the $cu$-strip of $f^{-k}L(w)$ such that $g^k(z^{-k})=\xi^k_{cu,w}(z)$. From the construction of $\xi^k_w$ it follows that $z^{-k-1}$ is contained in $\WW^{cu}_{2\delta'}(f^{-1}(w^{-k}))$ for $\delta'=d_0(f,g)$. That is, $(z^{-k})_{k\geq 0}$ is a backwards $2\delta'$-pseudo orbit for $f$ with jumps in local $cu$-plaques.

Recall that the maximal expansion possible for $df^{-1}$ is given by the constant $\kappa>1$. It follows that, independently of $w$, for every $K>0$ there exists a constant $R(K)>0$ (in terms of $\kappa$ and $\delta'$) such that $z^{-k}$ lies in the truncated $cu$-strip $\tilde{\W}^u_{\delta_3}(L_{R(K)}(f^{-k}(w)))$ for every $k\in \{0,\ldots,K\}$. This is satisfied for every $w\in M$ and $z\in \WW^u_{\delta_3}(L(w))$.

Let $N>0$ be such that, by Remark \ref{rmkR>0truncated} for $R_0=R(K_0)$ and $\mu_0=\delta_3/100$, the truncated $cu$-strip $\tilde{\W}^u_{\delta_3}(L_{R_0}(f^{-k}(x_n)))$ can be projected to $M$ and then lifted to $U^{su}_{2\delta_3}(L(f^{-k}(x)))$ so that it gets at Hausdorff distance less than $\mu_0$ from $\tilde{\W}^u_{\delta_3}(L_{R_0}(f^{-k}(x)))$ for every $k\in \{0,\ldots,K_0\}$ and every $n\geq N$. For simplicity, let us call these projection-lifts to $U^{su}_{2\delta_3}(L(f^{-k}(x)))$ of the truncated $cu$-strips $\tilde{\W}^u_{\delta_3}(L_{R_0}(f^{-k}(x_n)))$ with the same name $\tilde{\W}^u_{\delta_3}(L_{R_0}(f^{-k}(x_n)))$.

It follows that the maps $\xi^\infty_{cu,x_n}$ can also be lifted to $U^{su}_{2\delta_3}(L(x))$ for every $n\geq N$. The domains of such maps being contained in the truncated $cu$-strips $\tilde{\W}^u_{\delta_3}(L_{R_0}(x_n))$. Analogously, the maps $\xi^\infty_{cu,z_n}$ can be lifted to $U^{su}_{2\delta_3}(L(x))$ with domain contained in $\tilde{\W}^u_{\delta_3}(L_{R_0}(z_n))$. Again, for simplicity let us call these lifted maps with the same names $\xi^\infty_{cu,x_n}$ and $\xi^\infty_{cu,z_n}$.

We have to show that $d(\xi^\infty_{cu,x_n}(z_n), \xi^\infty_{cu,x}(z))<\epsilon$ is satisfied  in $U^{su}_{2\delta_3}(L(x))$ for every $n\geq N$.

Note that, modulo taking $N$ larger, for every $n\geq N$ the set $\WW^{s,g}_{\delta_3}(z_n)$ intersects $\WW^u_{\delta_3}(L(x))$ and that this intersection point is unique. Let us call it $w_n$. 

As $z_n\xrightarrow{n} z$ then $w_n\xrightarrow{n} z$. Since $\xi^\infty_{cu,x}$ is continuous in $\WW^u_{\delta_3}(L(x))$ it follows that $d\big(\xi^\infty_{cu,x}(z),\xi^\infty_{cu,x}(w_n)\big)<\epsilon/2$ for every $n\geq N$, by taking $N$ larger if needed. It remains to show that $d\big(\xi^\infty_{cu,x_n}(z_n),\xi^\infty_{cu,x}(w_n)\big)<\epsilon/2$ for every $n\geq N$.

The points $z_n^{-k}$ are well defined points in $U^{su}_{2\delta_3}(f^{-k}L(x))$ satisfying that $g^k(z_n^{-k})=\xi^k_{cu,x_n}(z_n)$ for every $k\in \{0,\ldots, K_0\}$. The points $w_n^{-k}$ are well defined points in $U^{su}_{2\delta_3}(f^{-k}L(x))$ satisfying that $g^k(w_n^{-k})=\xi^k_{cu,x}(w_n)$ for every $k\in \{0,\ldots, K_0\}$. The key point to note is that the above implies \begin{equation}\label{hcont(b)}w_n^{-k}\in\WW^{s,g}_{4\delta_3}(z_n^{-k})\end{equation} for every $k\in \{0,\ldots,K_0\}$ and $n\geq N$. This is because, as $w_n$ lies in $\WW^{s,g}_{\delta_3}(z_n)$, then $\xi^k_{cu,x_n}(z_n)$ and $\xi^k_{cu,x}(w_n)$ lie in the same $\WW^{s,g}$-plaque of $U^{su}_{2\delta_3}(L(x))$. Then $z_n^{-1}$ and $w_n^{-1}$, which are two points in $U^{su}_{2\delta_3}(f^{-1}L(x))$ satisfying that $g(z_n^{-1})$ and $g(w_n^{-1})$ are in the same $\W^{s,g}$ plaque of $U^{su}_{2\delta_3}(L(x))$, need to lie in the same $\W^{s,g}$ plaque too. Inductively, $z_n^{-k}$ and $w_n^{-k}$ need to lie in the same $\WW^{s,g}$-plaque of $U^{su}_{2\delta_3}(f^{-k}L(x))$. As all of these $\WW^{s,g}$-plaques have diameter less than $4\delta_3$ then (\ref{hcont(b)}) follows.

For $k=K_0$ in (\ref{hcont(b)}) it follows that $w_n^{-K_0}$ lies in $\WW^{s,g}_{4\delta_3}(z_n^{-K_0})$. Then $g^{K_0}(z_n^{-{K_0}})$ lies in $\WW^{s,g}_{4\delta_3\lambda^{K_0}}(g^{K_0}(z_n^{-K_0}))$. Recall that $g^{K_0}(z_n^{-{K_0}})=\xi^{K_0}_{cu,x_n}(z_n)$ and $g^{K_0}(w_n^{-K_0})=\xi^{K_0}_{cu,x}(w_n)$. Using (\ref{hcont(a)}) and the fact that $2\delta'<4\delta_3$ by property (P5) it follows that  $$d(\xi^{K_0}_{cu,x_n}(z_n),\xi^{K_0}_{cu,x}(w_n))<\epsilon/6$$ for every $n\geq N$.

Again by (\ref{hcont(a)}) it follows that
$$d(\xi^{K_0}_{cu,x}(z),\xi^\infty_{cu,x}(w_n))<\epsilon/6 \hspace{0.2cm}\mbox{ and }\hspace{0.2cm} d(\xi^{K_0}_{cu,x_n}(z_n),\xi^\infty_{cu,x_n}(z_n))<\epsilon/6$$ for every $n\geq N$ . 

By triangular inequality (two times) we conclude that $$d\big(\xi^\infty_{cu,x}(z_n),\xi^\infty_{cu,x}(w_n)\big)<\epsilon/2$$ for every $n\geq N$.

This shows that $h$ is continuous and ends the proof of the claim.
\end{proof}

\begin{af}
The map $\rho:M\to M$ is a homeomorphism.
\end{af}
\begin{proof}
Recall that we have already seen that $\rho(L)=L$ and $\rho|_L:L\to L$ is injective and $\delta$-close to the identity for every leaf $L\in \W^c$. And that $h\circ \rho\circ f=g\circ h$ is satisfied.

It remains to show that $\rho$ is continuous as $\rho$ continuous, injective and $\delta$ close to the identity implies $\rho$ homeomorphism. 

Suppose $x_n\xrightarrow{n} x$ in $M$. Let us see that $\rho \circ f(x_n)  \xrightarrow{n} \rho \circ f(x)$. As $f$ is continuous this implies $\rho$ continuous.

As $x_n\xrightarrow{n} x$ then $g\circ h(x_n) \xrightarrow{n} g\circ h(x)$ by the continuity of $h$ and $g$. Since $h\circ \rho\circ f=g\circ h$ it follows that $h\circ \rho\circ f(x_n)$ has limit $h\circ \rho\circ f(x)$.

As $x_n\xrightarrow{n} x$ and $\rho\circ f(x_n)$ lies in $\W^c_\delta(f(x_n))$ for every $n$ it follows that every accumulation point of the sequence $\rho\circ f(x_n)$ must lie in $\W^c_\delta(\rho\circ f(x))$. Because of $\frac{1}{2}<\|Dh|_{E^c}\|<2$ the map $h$ is injective restricted to $\W^c_\delta(\rho\circ f(x))$. Hence the only way  that $h\circ \rho\circ f(x_n)$ has limit $h\circ \rho\circ f(x)$ is that $\rho \circ f(x_n)$ converges to $\rho \circ f(x)$. This shows the continuity of $\rho$ and ends the proof of the claim and of Theorem \ref{thmgraphtransf}.
\end{proof}

\subsection{Proof of Theorem \ref{thmunifcontC1completecenterimm}: Continuation of complete $C^1$ center immersions} 

Suppose $\eta:\R\to M$ is a complete $C^1$ immersion tangent to $E^c_f$ as in the hypothesis of Theorem \ref{thmunifcontC1completecenterimm}. As in the proof of Theorem \ref{thmgraphtransf} one can construct an abstract manifold $U(\eta)$, informally an `unfolded neighborhood' of $\eta$, given as the disjoint union of the sets $\{B_\delta(\eta(t))\}_{t\in \R}$ with the points $y\in B_\delta(\eta(t))$ and $y'\in B_\delta(\eta(s))$ identified if and only if $y=y'$ and the piece of $\eta$-orbit from $\eta(t)$ to $\eta(s)$ has length less than $4\delta$. Then in $U(\eta)$ there exists a natural projection $\pi:U(\eta)\to M$ which is a local diffeomorphism at any point and $U(\eta)$ can be given the structure of abstract Riemannian manifold by taking pull-back of the structure in $M$ by the restrictions of $\pi:U(\eta)\to M$ to the sets $B_\delta(\eta(t))$ for $t\in \R$.

Analogously as for Theorem \ref{thmgraphtransf} one can construct a manifold $V(\eta)$ whose connected components are $U(f^n\circ \eta)$ for every $n\in \mathbb{Z}$, so that $g$ can be `lifted' to $V(\eta)$ (sending points from one connected component $U(f^n\circ \eta)$ to the next one $U(f^{n+1}\circ \eta)$ for every $n\in\mathbb{Z}$) so that the graph transform method (Lemma \ref{claimgraphtransf}) can be performed in $V(\eta)$. 

One obtains the existence of a sequence $\gamma_n:N\to M$ of complete $C^1$ immersions tangent to $E^c_{g}$ satisfying (\ref{eq1thmC1imm}) and (\ref{eq2thmC1imm}) of Theorem \ref{thmunifcontC1completecenterimm} by exactly the same arguments already seen for the continuation of $\W^c$-leaves in the proof of Theorem \ref{thmgraphtransf}.

For the uniqueness part of Theorem \ref{thmgraphtransf}, modulo reparametrizations, note that by the same arguments showing that the continuation $h(L)$ of a center leaf $L$ is characterized as the only points in $U(L)$ for which its $g$ orbit is well defined for every backwards and forwards iterate (see Remark \ref{rmkcharactcontinuation}) it follows that the image of every lift of $\gamma_0$ to $U(\eta)$ coincides with the points in $U(\eta)$ whose $g$-orbit is well defined in $V(\eta)$ for every backwards and forwards iterate (in particular, there exits a unique lift).

Then if $\gamma_n':\R\to M$ is another sequence of complete $C^1$ immersions tangent to $E^c_{g}$ satisfying (\ref{eq1thmC1imm}) and (\ref{eq2thmC1imm}) one deduces that the lift of $\gamma_0'$ to $U(\eta)$ has to have the same image as the lift of $\gamma_0$. It follows that $\gamma_0'$ is a reparametrization of $\gamma_0$ and, by (\ref{eq2thmC1imm}), that $\gamma_n'$ is a reparametrization of $\gamma_n$ for every $n\in\mathbb{Z}$.

\section{Global stability: Proof of Theorem \ref{thmA}}\label{sectionproofthmA}

As discussed below, Theorem \ref{thmA} is a consequence of the following.

\begin{prop}\label{propopenclosed}
For every $f_0\in \PH_{c=1}(M)$ there exists a $C^1$-neighborhood $\U$ of $f_0$ such that, if $f\in \mathcal{U}$ is a discretized Anosov flow, then every $g\in \mathcal{U}$ is also a discretized Anosov flow.

Moreover, if $\W^c_f$ and $\W^c_{g}$ denote the flow center foliations of $f$ and $g$, respectively, then $(f,\W^c_f)$ is plaque expansive and $(f,\W^c_f)$ and $(g,\W^c_{g})$ are leaf-conjugate.
\end{prop}

Let us mention how Theorem \ref{thmA} follows from  Proposition \ref{propopenclosed}.

Suppose Proposition \ref{propopenclosed} is true. If $f_0$ is a discretized Anosov flow then there exists a neighborhood $\U$ of $f_0$ such that every element of $\U$ is a discretized Anosov flow. This proves the open property. 

Moreover, if $f_n$ is a sequence of discretized Anosov flows converging to $f_0\in \PH_{c=1}(M)$ and $\U$ is as in Proposition \ref{propopenclosed} then for some large $N$ the map $f_N$ lies in $\U$ and as consequence $f_0$ is also a discretized Anosov flow. This proves the closed property.

It is immediate from Proposition \ref{propopenclosed} that leaf-conjugacy persists in the whole $C^1$ partially hyperbolic connected component:
\begin{cor}\label{corleafconjDAFs}
Suppose $f$ and $g$ belong to the same connected component of $\PH_{c=1}(M)$. If $f$ is a discretized Anosov flow then $g$ is also a discretized Anosov flow and, if $\W^c_f$ and $\W^c_g$ denote their flow center foliations, then $(f,\W^c_f)$ and $(g,\W^c_g)$ are leaf-conjugate.
\end{cor}
\begin{proof}
Let $\mathcal{V}=\{g\in \PH_{c=1}(M)\mid g$ admits a $g$-invariant center foliation $\W^c_g$ such that $(g,\W^c_g)$ is leaf-conjugate to $(f,\W^c_f)\}$ and let $\mathcal{V}(f)$ be the connected component of $\PH_{c=1}(M)$ containing $f$. By Proposition \ref{propopenclosed} the set $\mathcal{V}$ is open and closed in $\PH_{c=1}(M)$. As $\mathcal{V}(f)$ is connected then either $\mathcal{V}(f)\cap \mathcal{V}$ or $\mathcal{V}(f)\cap (\PH_{c=1}(M)\setminus \mathcal{V})$ are empty. Since $f$ belongs to $\mathcal{V}(f)\cap \mathcal{V}$ one obtains that $\mathcal{V}(f)\subset \mathcal{V}$.

\end{proof}

The rest of the section is devoted to prove Proposition \ref{propopenclosed}. We will crucially use Theorem \ref{thmgraphtransf} and Lemma \ref{lemmadeltaplaqueexp=>leafconj} from Section \ref{appendixA}. And we will assume familiarity with the terminology used in Subsection \ref{subsectionplaqueexpappA}.

\subsection{Plaque expansivity for discretized Anosov flows}\label{subsectionplaqueexp}

The notion of $\delta$-plaque expansivity was introduced in Definition \ref{defdeltape}. For every neighborhood $\U_\delta(f_0)$ as in Lemma \ref{rmkUdelta} we will consider a subset of $\U_\delta(f_0)$ satisfying a stronger version of property (P4). This will allow us to show in Proposition \ref{propplaqueexp} that any discretized Anosov on this new neighborhood has to be $\delta$-plaque expansive.

As discussed in Remark \ref{rmkafterUdelta}, property (P4') in the next lemma can be derived from the proof of Lemma \ref{rmkUdelta}. By taking $\epsilon'>0$ small enough it is immediate to check the rest of the lemma.

\begin{lemma}\label{rmkUdeltaDAF}
Suppose $f_0\in\PH_{c=D}(M)$ for some $D>1$. There exist a metric in $M$, constants $0<\lambda<1$, $\kappa>1$ and $\delta_0(f)>0$, for every $0<\delta\leq \delta_0(f)$ a $C^1$ neighborhood $\U_\delta(f_0)$ of $f_0$, and a constant $\epsilon'>0$ so that properties (P1), \ldots, (P5) of Lemma \ref{rmkUdelta} are satisfied for every $f,g\in \U_\delta(f_0)$, and the following reinforcement of property (P4) is satisfied
\begin{itemize}
\item[(P4')]\label{P4DAF} The metric and the cone fields $(\C^s,\C^{cu})$ and $(\C^{cs},\C^u)$ are $\epsilon'$-nearly euclidean at scale $\kappa 20 \delta$
\end{itemize}
ensuring that for every
\begin{itemize}
\item[$\cdot$] map $f\in\U_{\delta}(f_0)$,
\item[$\cdot$] pair of points $x\in M$ and $x'\in\W^s_f(x)$ with $d_s(x,x')=10\delta$,
\item[$\cdot$] pair of $C^1$ curves $\eta$ and $\eta'$ of length less than $20\delta$, tangent to $\C^{cs}\cap \C^{cu}$ with $x\in \eta$ and $x'\in \eta'$,
\end{itemize}
one has that $$d(\eta,\eta')>\lambda(10\delta),$$ where $d(\eta,\eta')$ denotes the infimum distance between points in $\eta$ and $\eta'$.
\end{lemma}

The following is the goal of this subsection.

\begin{prop}[Uniform plaque expansivity for discretized Anosov flows]\label{propplaqueexp} Suppose $f_0\in \PH_{c=1}(M)$. Consider a metric in $M$ and a $C^1$ neighborhood $\mathcal{U}_\delta(f_0)\subset\PH_{c=1}(M)$ of $f_0$ as in Lemma \ref{rmkUdeltaDAF}. If $f$ is a discretized Anosov flow in $\U_\delta(f_0)$ and $\W^c$ is the flow center foliation of $f$ then $(f,\W^c)$ is $\delta$-plaque expansive.
\end{prop}

\begin{proof}[Proof of Proposition \ref{propplaqueexp}] By Definition \ref{defDAFintro}, Proposition \ref{prop1DAFs} and Remark \ref{rmkDAFreparam} the map $f$ can be written down as $f(x)=\varphi_{\tau(x)}(x)$ for $\tau:M\to \R_{>0}$ continuous and $\varphi_t:M\to M$ a unit speed flow whose flow lines are the flow center foliation $\W^c$ of $f$.

The following is a key claim showing that, even taking into account possible `backwards jumps', every $\delta$-pseudo orbit `advances forward' in the direction of the flow. Roughly speaking this allows us to bring into play the expansivity of the topological Anosov flow $\varphi_t$ to obtain expansivity for pairs of $\delta$-pseudo orbits belonging to different center leaves.

\begin{af}\label{claimplaqueexp1} 
The function $\tau$ is always larger than $10\delta$.
\end{af}
\begin{proof}\let\qed\relax

Suppose by contradiction that $\tau(x)< 10\delta$ for some $x\in M$. 

For every $y\in \W^s_{11\delta}(x)$ let $\gamma_y:[0,1]\to \W^c(y)$ be the constant speed reparametrization of the piece of $\varphi_t$-orbit from $y$ to $f(y)$. Note that by the continuity of $\tau$ the length of $\gamma_y$ varies continuously with $y$.

By property (P4') the image of $\gamma_y$ needs to be a segment from $\W^s_{11\delta}(x)$ to $\W^s_{12\delta}(f(x))$, contained in $B_{20\delta}(x)$ and whose length does not surpass $12\delta$. In particular, the image of $\gamma_y$ is contained in $\W^c_{20\delta}(y)$ for every $y\in \W^s_{11\delta}(x)$.

Let us fix $y_0$ in  $\W^s_{11\delta}(x)$ such that $d_s(x,y_0)=10\delta$. It follows that $\gamma_{y_0}$ is a curve joining $y_0$ to $f(y_0)$, where $d(f(x),f(y_0))\leq d_s(f(x),f(y_0))<\lambda (10\delta)$. 

By Lemma \ref{rmkUdeltaDAF} the sets $\W^c_{20\delta}(x)$ and $\W^c_{20\delta}(y_0)$ must be at distance greater than $\lambda 10\delta$. However, we have just shown that $f(x)\in \W^c_{20\delta}(x)$ and $f(y_0)\in \W^c_{20\delta}(y_0)$ are at distance less than $\lambda10\delta$. This gives us a contradiction and proves Claim \ref{claimplaqueexp1}. See Figure \ref{fig5} for a schematic idea of the argument used.

\begin{figure}[htb]
\def\svgwidth{0.8\textwidth}
\begin{center} 
{\scriptsize
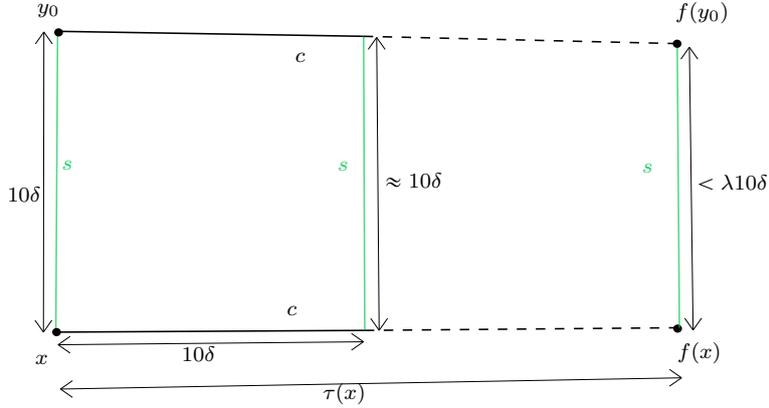
}
\caption{At small scale the bundles $E^\sigma$, $\sigma\in\{s,c,u\}$, are nearly parallel and pairwise disjoint for every $f$ near $f_0$. Hence $\tau$ must greater than $10\delta$ to be able to `see' the contraction of rate $\lambda\in (0,1)$.}\label{fig5} 
\end{center}
\end{figure}

\end{proof}

Recall that by Proposition \ref{propdyncoh} the discretized Anosov flow $f$ is dynamically coherent with center-stable foliation $\W^{cs}$ and center-unstable foliation $\W^{cu}$  such that $\W^c=\W^{cs}\cap \W^{cu}$. As stated in the next claim, dynamical coherence let us obtain $\delta$-plaque expansivity by checking $2\delta$-plaque expansivity inside $\W^{cs}$ and $\W^{cu}$ leaves. 

\begin{af}\label{claimplaqueexp2}
Suppose the following statement is true: For every $(x_n)_{n\geq 0}$ and $(y_n)_{n\geq 0}$ forward $2\delta$-pseudo orbits such that $x_{n+1}\in \W^c_{2\delta}(f(x_n))$, $y_{n+1} \in \W^c_{2\delta}(f(y_n))$ and $y_n\in \W^{cu}_{4\delta}(x_n)$ for every $n\geq 0$, then $y_0\in \W^c_{8\delta}(x_0)$. 

Suppose that the analogous statement for backwards $2\delta$-pseudo orbits inside $\W^{cs}$ leaves is also true.  Then $(f,\W^c)$ is $\delta$-plaque expansive.
\end{af}
\begin{proof}\let\qed\relax
Let $(x_n)_n$ and $(y_n)_n$ be a pair of $\delta$-pseudo orbits satisfying $x_{n+1}\in \W^c_\delta(f(x_n))$, $y_{n+1} \in \W^c_{\delta}(f(y_n))$ and $d(x_n,y_n)<2\delta$ for every $n \in \mathbb{Z}$. Let us see that $y_0\in \W^c_{3\delta}(x_0)$.

All along the proof of the claim we will implicitly use that, by property (P4'), at scale $\kappa 20\delta$ the invariant bundles are nearly pairwise orthogonal. It will be clear on each case that what is stated follows directly from property (P4'). And we will implicitly use that by dynamical coherence $cs$ (resp $cu$) discs are subfoliated by $c$ and $s$ (resp $u$) discs and that $cs$ and $cu$ discs intersect in $c$ discs.

As $d(x_n,y_n)<2\delta$ one can consider $y_n'$ the intersection of $\W^s_{3\delta}(y_n)$ with $\W^{cu}_{3\delta}(x_n)$. 

Moreover, $y_n'\in \W^{cu}_{3\delta}(x_n)$ implies $f(y_n')\in \W^{cu}_{\kappa 3\delta}(f(x_n))$ and $y_n'\in \W^s_{3\delta}(y_n)$ implies $f(y_n')\in \W^{s}_{\kappa 3\delta}(f(y_n))$. It follows that $f(y_n')$ is the intersection point of $\W^s_{\kappa3\delta}(f(y_n))$ with $\W^{cu}_{\kappa3\delta}(f(x_n))$.

The point $x_{n+1}$ lies in $\W^c_\delta(f(x_n))\subset \W^{cu}_{\kappa3\delta}(f(x_n))$ and the point $y_{n+1}'$ is given by the intersection of $\W^s_{3\delta}(y_{n+1})$ with $\W^{cu}_{3\delta}(x_{n+1})\subset \W^{cu}_{\kappa3\delta}(f(x_n))$. We obtain that $f(y_n')$ and $y_{n+1}'$ are both contained in $\W^{cu}_{\kappa3\delta}(f(x_n))$. And both contained in $\W^s_{\kappa3\delta}(\W^c_\delta(f(y_n))$ since $y_{n+1}\in \W^c_\delta(f(y_n))$. It follows that $f(y_n')$ and $y_{n+1}'$, which lie in the intersection of $\W^s_{\kappa3\delta}(\W^c_\delta(f(y_n))\subset \W^{cs}(f(y_0))$ and $\W^{cu}_{\kappa3\delta}(f(x_n))$, are in the same local center manifold. Since $y_{n+1}\in \W^c_\delta(f(y_n))$ it follows that $y_{n+1}'\in\W^c_{2\delta}(f(y_n'))$.

Then $(y_n')_{n\geq 0}$ is a forward $2\delta$-pseudo orbit with jumps in center plaques, as well as $(x_n)_{n\geq 0}$, and they satisfy $y_n'\in \W^{cu}_{3\delta}(x_n)$ for every $n\geq 0$. By the assumption of the statement it follows that $y_0'$ lies $\W^c_{8\delta}(x_0)$. Which in turns imply $y_0\in \W^{cs}_{11\delta}(x_0)$ as $y_0'\in \W^c_{3\delta}(y_0)$.

By defining analogously $(y_n'')_{n \leq 0}$ a backward $2\delta$-pseudo orbit as the intersection of $\W^u_{3\delta}(y_n)$ with $\W^{cs}_{3\delta}(x_n)$ for every $n\leq 0$ we conclude that $y_0$ lies $\W^{cu}_{11\delta}(x_0)$. It follows that $y_0=y_0'=y_0''$, and then that $y_0$ lies in $\W^c_{8\delta}(x_0)$. As we are at scale $\kappa 20 \delta$ then $d(x_0,y_0)<2\delta$ and $y_0\in \W^c_{8\delta}(x_0)$ imply that $y_0$ lies in  $\W^c_{3\delta}(x_0)$. This proves Claim \ref{claimplaqueexp2}.
\end{proof}

Suppose $(x_n)_{n\geq 0}$ and $(y_n)_{n\geq 0}$ are two forward $2\delta$-pseudo orbits such that $x_{n+1}\in \W^c_{2\delta}(f(x_n))$, $y_{n+1} \in \W^c_{2\delta}(f(y_n))$ and $y_n\in \W^{cu}_{4\delta}(x_n)$ for every $n\geq 0$. Let us see that $y_0$ needs to lie in $\W^c_{8\delta}(x_0)$. It will be clear that in a similar fashion one can show the analogous statement for backwards $2\delta$-pseudo orbits if $y_0$ is a point in $\W^{cs}_{4\delta}(x_0)$. By Claim \ref{claimplaqueexp2} this will end the proof of Proposition \ref{propplaqueexp}.

Suppose by contradiction that $y_0$ does not belong to $\W^c_{8\delta}(x_0)$. As $y_n\in \W^{cu}_{4\delta}(x_n)$ we can consider $x'_n\in \W^c_{5\delta}(x_n)$ such that $y_n\in \W^u_{5\delta}(x'_n)$ for every $n\geq 0$. It follows that $y_n\neq x_n'$ for every $n\geq 0$.

Note that $y_0\in \W^u_{5\delta}(x_0')$ implies that $d_u(f^{-n}(y_0),f^{-n}(x_0'))$ tends to $0$ with $n$. Since $\{f^{-n}(y_0)\}_{n\geq 0}$ are points in $\W^c(y_0)$ and $\{f^{-n}(x_0')\}_{n\geq 0}$ are points in $\W^c(x_0)$ it follows that $\W^c(x_0)$ and $\W^c(y_0)$ can not be both compact leaves. As the conditions for $(x_n)_{n\geq 0}$ and $(y_n)_{n\geq 0}$ are symmetric let us assume without loss of generality that $\W^c(x_0)$ is not compact.

For every pair of different points $z,z'\in \W^c(x_0)$ let $[z,z']_c$ and $[z,z')_c$ denote the closed and half-open center segments from $z$ to $z'$ inside $\W^c(x_0)$, respectively. Let us say that a sequence $(z_n)_{n\geq 0}$ in $\W^c(x_0)$ tends to $+\infty_c$ (resp. $-\infty_c$) if $z_n=\varphi_{t_n}(x_0)$ for $t_n \xrightarrow{n} +\infty$ (resp. $-\infty$). 

Since $\tau$ is continuous and positive it is bounded away from zero. It follows that the sequence $(f^k(x_0))_{k\geq 0}$ (resp. $(f^{-k}(x_0))_{k\geq 0}$) tends to $+\infty_c$ (resp. $-\infty_c$). In particular, $\W^c(x_0)$ can be written down as the disjoint union $\bigcup_{k\in \mathbb{Z}}[f^k(x_0),f^{k+1}(x_0))_c$, where each segment $[f^k(x_0),f^{k+1}(x_0))_c$ is a \emph{fundamental domain} for the dynamics of $f$ restricted to $\W^c(x_0)$.

Since $x_{n+1}$ lies in $\W^c_{2\delta}(f(x_n))$ for every $n\geq 0$ and by Claim \ref{claimplaqueexp1} the function $\tau$ is greater than $10\delta$ it follows that $(x_n)_{n\geq 0}$ tends to $+\infty_c$ as well. And, since $x_n'$ lies in $\W^c_{5\delta}(x_n)$ for every $n\geq 0$, the same happens to the sequence $(x_n')_{n\geq 0}$. Let $(K_n)_{n\geq 0}$ be the sequence such that $x_n'$ lies in $\big[f^{K_n}(x_0),f^{K_n+1}(x_0)\big)_c$ for every $n \geq 0$. One concludes that $K_n \xrightarrow{n} +\infty$.

For every $n\geq 0$ the point  $f^{-K_n}(x_n')$ is a point in the closed segment $[x_0,f(x_0)]_c$. Moreover, since $y_n$ lies in $\W^u_{5\delta}(x'_n)$ it follows that $f^{-K_n}(y_n)$ lies in $\W^u_{\lambda^{K_n}5\delta}(f^{-K_n}(x_n'))$. And as $y_n\neq x_n'$, then $f^{-K_n}(y_n)$ is not contained in $[x_0,f(x_0)]_c$. It follows that  \begin{equation}
\label{eqPropplaqexp}f^{-K_n}(y_n)\in\W^u_{\lambda^{K_n}5\delta}([x_0,f(x_0)]_c)\setminus [x_0,f(x_0)]_c\end{equation} for every $n$.

As $\lambda$ is a constant in $(0,1)$ the sequence $\lambda^{-K_n}5\delta$ tends to $0$ with $n$. And since $\W^c(y_0)$ contains $y_n$ for every $n$ and is an invariant leaf by $f$ it follows that $f^{-K_n}(y_n)$ is a sequence contained in $\W^c(y_0)$. Hence if the following claim is true one gets a contradiction with (\ref{eqPropplaqexp}), ending the proof of Proposition \ref{propplaqueexp}. 

\begin{af}\label{claimlastplaqueexp}
There exists $\epsilon>0$ such that $\W^u_{\epsilon}([x_0,f(x_0)]_c)\setminus [x_0,f(x_0)]_c$ is disjoint from $\W^c(y_0)$. 
\end{af}
\begin{proof}\let\qed\relax
Note that a priori one can not rule out that $\W^c(x_0)$ and $\W^c(y_0)$ may be the same leaf. That is why we will show that $\W^c(y_0)$ is disjoint from $\W^u_{\epsilon}([x_0,f(x_0)]_c)\setminus [x_0,f(x_0)]_c$ and not simply disjoint from $\W^u_{\epsilon}([x_0,f(x_0)]_c)$.

Recall from Proposition \ref{propsTopAF} the topological description of the center-unstable leaves of a discretized Anosov flow in terms of planes leaves, cylinder leaves, etc.

If $\W^{cu}(x_0)$ is a plane leaf the claim follows straightforwardly from Proposition \ref{propsTopAF} since in that case the foliations $\W^u$ and $\W^c$ need to have a global product structure inside $\W^{cu}(x_0)$. 

If $\W^{cu}(x_0)$ is a cylinder leaf then Proposition \ref{propsTopAF} shows that the alpha-limit of $y_0$ by the center flow $\varphi_t$ coincides with the unique compact leaf $L$ of $\W^c$ contained in $\W^{cu}(x_0)$. Moreover, as $\W^c(x_0)$ is not compact, then $L\neq \W^c(x_0)$.

In case $\W^c(y_0)$ coincides with $L$ it is enough to consider $\epsilon>0$ smaller than the Hausdorff distance between the compact and disjoints sets $\W^c(y_0)$ and $[x_0,f(x_0)]_c$. 

In the case that $\W^c(y_0)$ does not coincide with $L$, the leaf $\W^c(y_0)$ is not compact and it is immediate to check that the omega-limit of $y_0$ in the intrinsic metric of $\W^{cu}(x_0)$ needs to be empty. This follows from the fact that for every $R>0$ the point $f^n(y_0)$ can not be contained in $\W^u_R(L)$ for arbitrarily large $n>0$. Indeed, if $f^n(y_0)$ lies in $\W^u_R(L)$ for arbitrarily large $n>0$ then $y_0=f^{-n}\circ f^n(y_0)$, which is not contained in $L$, would be at arbitrarily small distance from the compact leaf $L$ getting to a contradiction.

It follows that for some $T>0$ the set $\W^c(y_0)\setminus \W^c_T(x_0)$ is at positive distance from the compact set $[x_0,f(x_0)]_c$ in the intrinsic metric of $\W^{cu}(x_0)$. Say $d>0$. 

If $\W^c_T(y_0)$ is disjoint from $[x_0,f(x_0)]_c$ it is enough to consider $d
>\epsilon>0$ so that $\epsilon$ is smaller than the Hausdorff distance between $\W^c_T(y_0)$ and $[x_0,f(x_0)]_c$.

If $\W^c_T(y_0)$ is not disjoint from $[x_0,f(x_0)]_c$ then for some $x_1,x_2\in \W^c(x_0)$ satisfying $[x_0,f(x_0)]_c\subset (x_1,x_2)_c$ the center segment $[x_1,x_2]_c$ needs to be contained in $\W^c_T(y_0)$ since the endpoints of $\W^c_T(y_0)$ are far from $[x_0,f(x_0)]_c$. It is enough to consider in this case $d>\epsilon>0$ so that $\epsilon $ is smaller than the Hausdorff distance between $\W^c_T(x_0)\setminus (x_1,x_2)_c$ and $[x_0,f(x_0)]_c$.

This proves Claim \ref{claimlastplaqueexp} and ends the proof of Proposition \ref{propplaqueexp}.
\end{proof}
\vspace{-0.2cm}
\end{proof}

\begin{remark}
As was pointed out by a referee, it is worth noting at this point and bringing into relevance the link with \cite[Proposition C.2.1.]{Ber} where, by a priori different methods and in a somehow more general context, has already identified and shown plaque-expansivity in the context of normally expanded (or contracted) quasi-isometrically actions.
\end{remark}

\subsection{Proof of Proposition \ref{propopenclosed}} 

Suppose $f_0\in \PH_{c=1}(M)$. Consider a metric in $M$, a constant $\delta(f_0)>0$ and, for some $0<\delta\leq\delta(f_0)$, a $C^1$ neighborhood $\U_\delta(f_0)\subset\PH_{c=1}(M)$ of $f_0$ as in Lemma \ref{rmkUdeltaDAF}.

Suppose $f$ and $g$ are maps in $\U_\delta(f_0)$ such that $f$ is a discretized Anosov flow. Let us see that $g$ needs to be also a discretized Anosov flow.

Suppose $f$ is of the form $f(x)=\varphi_{\tau(x)}(x)$ and let $\W^c$ denote the flow center foliation whose leaves are the flow lines of $\varphi_t$. 

Let $h:M\to M$ and $\rho:M\to M$ be as in Theorem \ref{thmgraphtransf}. By Proposition \ref{propplaqueexp} the system $(f,\W^c)$ is $\delta$-plaque expansive (in particular it is plaque expansive, see Remark \ref{rmkdeltaplaqueexpisplaqueexp}). By Remark \ref{rmkleafconj} and Lemma \ref{lemmadeltaplaqueexp=>leafconj} it follows that $h$ is a homeomorphism and that $h(\W^c)=\W^c_g$ is a $g$-center foliation such that $(f,\W^c)$ and $(g,\W^c_g)$ are leaf conjugate. In particular, $g(W')=W'$ for every leaf $W'\in \W^c_g$.

By Proposition \ref{propcenterfixingL} there exists $L>0$ such that $f(x)\in \W^c_L(x)$ for every $x\in M$. By Theorem \ref{thmgraphtransf} the maps $h$ and $\rho$ satisfy $\frac{1}{2}<\|Dh|_{E^c_f}\|<2$ and $h\circ \rho \circ f =g \circ h$. Moreover, $\rho(W)=W$ and $\rho$ is a $\delta$-close to the identity map inside $W$ for every leaf $W\in\W^c$. Then $g(x)=\W^c_{g,2(L+\delta)}(x)$ for every $x\in M$.

By denoting $L'= 2(L+\delta)$ we obtain that $g$ individually fixes each leaf of the center foliation $\W^c_{g}$ satisfying $$g(x)\in \W^{c}_{g,L'}(x)$$ for every $x\in M$. By Proposition \ref{propcenterfixingL} we conclude that $g$ is a discretized Anosov flow. Moreover, it is immediate to check from the proof of Proposition \ref{propcenterfixingL} that $\W^c_{g}$ needs to be the flow center foliation of $g$. This ends the proof of Proposition \ref{propopenclosed}, and therefore also that of Theorem \ref{thmA}.

\begin{remark}
By this point, it is worth spending a few words to highlight how Proposition \ref{propopenclosed} compares with \cite[Propositions 4.5 and 4.6]{BFP}. While both are based on the same framework, there are some added difficulties when it comes to showing Proposition \ref{propopenclosed}. One of them is that \cite{BFP} relies many times on the ambient dimension 3, while Proposition \ref{propopenclosed} needs to deal with arbitrary dimension. Another, is that at some point \cite{BFP} may not have to care for the continuation of certain leaves to be disjoint, as the leaves in collapsed Anosov flows are allowed to merge. To show Proposition \ref{propopenclosed} one needs to justify by means of certain `uniform' estimates as in Proposition \ref{propplaqueexp} that the map $h$ from Theorem \ref{thmgraphtransf} is a homeomorphism.
\end{remark}

\section{Global stability for uniformly compact center foliation: Proof of Theorem \ref{thmAprima}}\label{sectionAprima}

As in the case of discretized Anosov flows (see the discussion at the beginning of Section \ref{sectionproofthmA}), Theorem \ref{thmAprima} follows from the following proposition.

\begin{prop}\label{propopenclosedunifcompact}
Suppose $f_0\in \PH_{c=1}(M)$. There exists a $C^1$-neighborhood $\U\subset \PH_{c=1}(M)$ of $f_0$ such that, if $f\in \mathcal{U}$ admits a uniformly compact center foliation, then every $g\in \mathcal{U}$ also admits a uniformly compact center foliation.

Moreover, if $\W^c_f$ and $\W^c_{g}$ denote the uniformly compact center foliations of $f$ and $g$, respectively, then $(f,\W^c_f)$ and $(g,\W^c_{g})$ are leaf-conjugate.
\end{prop}

\begin{proof}[Proof of Proposition \ref{propopenclosedunifcompact}]
Suppose $f_0\in \PH_{c=1}(M)$. Consider a metric in $M$, a constant $\delta(f_0)>0$ and a $C^1$ neighborhood $\U_{\delta_0}(f_0)\subset \PH_{c=1}(M)$ of $f_0$ as in Lemma \ref{rmkUdeltaDAF}.

Suppose there exists $f$ in $\U_{\delta_0}(f_0)$ admitting a uniformly compact center foliation $\W^c_f$. Let us see that every $g\in \U_{\delta_0}(f_0)$ admits a uniformly compact center foliation $\W^c_{g}$ such that $(f,\W^c_f)$ and $(g,\W^c_{g})$ are leaf-conjugate. By Lemma \ref{lemmadeltaplaqueexp=>leafconj} it is enough to show that $(f,\W^c)$ is $\delta$ plaque-expansive as in Definition \ref{defdeltape}. 

By Proposition \ref{propdyncoh} (see also \cite[Theorem 1]{BoBo}) the map $f$ is dynamically coherent admitting $f$-invariant foliations $\W^{cs}$ and $\W^{cu}$ such that $\W^c=\W^{cs}\cap \W^{cu}$. 

Note that, as it was shown in Claim \ref{claimplaqueexp2} during the proof of Proposition \ref{propplaqueexp}, in order to show that $(f,\W^c)$ is $\delta$-plaque expansive it is enough to show that the following property is satisfied (together with its analogous version for backwards orbits and $cs$-leaves): if $(x_n)_{n\geq 0}$ and $(y_n)_{n\geq 0}$ are two forward $2\delta$-pseudo orbits such that $x_{n+1}\in \W^c_{2\delta}(f(x_n))$, $y_{n+1} \in \W^c_{2\delta}(f(y_n))$ and $y_n\in \W^{cu}_{4\delta}(x_n)$ for every $n\geq 0$, then $y_0\in \W^c_{8\delta}(x_0)$.

Suppose by contradiction that in the context above the point $y_0$ does not belong to  $\W^c_{8\delta}(x_0)$. Again, as in the proof of Proposition \ref{propplaqueexp} the fact that $y_n$ lies in $\W^{cu}_{4\delta}(x_n)$ allows us to consider $x'_n\in \W^c_{5\delta}(x_n)$ such that $y_n\in \W^u_{5\delta}(x'_n)$ for every $n\geq 0$. As $y_0\notin \W^c_{8\delta}(x_0)$ it follows that $y_n\neq x_n'$ for every $n\geq 0$.

By defining $w_n=f^{-n}(x_n')$ and $z_n=f^{-n}(y_n)$ we obtain that $w_n$ and $z_n$ are points contained in $\W^c(x_0)$ and $\W^c(y_0)$, respectively, satisfying that $\lim_{n\to +\infty} d(w_n,z_n)= 0$. By considering $w_{\infty}$ an accumulation point of $(w_n)_{n\geq 0}$ and $U(w_{\infty})$ a small $\W^c$-foliation box neighborhood of $w_{\infty}$ we obtain that there exists a subsequence $(z_{n_k})_{k\geq 0}$  
tending to $w_{\infty}$ such that each $z_{n_k}$ corresponds to a different center plaque in $ U(w_{\infty})$. As $(z_{n_k})\subset \W^c(y_0)$ and $\W^c(y_0)$ is compact we get to a contradiction since $\W^c(y_0)$ cannot contain infinitely many disjoint plaques of $U(w_{\infty})$.

\end{proof}

The following corollary is derived from Proposition \ref{propopenclosedunifcompact} by means of the same arguments as Corollary \ref{corleafconjDAFs} was derived from Proposition \ref{propopenclosed}. We can therefore omit its proof.

\begin{cor}\label{corleafconjunifcompact}
Suppose $f$ and $g$ belong to the same connected component of $\PH_{c=1}(M)$. If $f$ admits a uniformly compact center foliation $\W^c_f$ then $g$ admits a uniformly compact center foliation $\W^c_{g}$ such that $(f,\W^c_f)$ and $(g,\W^c_{g})$ are leaf-conjugate.
\end{cor}

Whether there exists $f$ in $\PH(M)$ admitting a compact center foliation with non uniformly bounded volume of leaves is still unknown. Partial non-existence results have been given in \cite{Car}, \cite{Gog} and \cite{DMM}, not exclusively for the one-dimensional center scenario.

Assuming one-dimensional center it is worth noting that the second part of the proof of Proposition \ref{propopenclosedunifcompact} only uses that $\W^c$ is compact and $f$ is dynamically coherent. Moreover, by Theorem \ref{thmgraphtransf} (\ref{(1)thmplaqueexp}), whenever $h$ is a homeomorphism the volume of a compact center leaf $L$ and its continuation $h(L)$ differ at most by a constant factor depending only on the $C^1$ neighborhood $\U_{\delta(f_0)}(f_0)$. 
Thus the following statement follows from the proof of Proposition \ref{propopenclosedunifcompact}.

\begin{prop}\label{propglobalsatbilitynonunifcompact}
Suppose $f\in \PH_{c=1}(M)$ is a dynamically coherent system admitting $f$-invariant foliations $\W^{cs}$ and $\W^{cu}$ such that $\W^c=\W^{cs}\cap\W^{cu}$ is a non-uniformly compact center foliation. 

There exists a $C^1$ neighborhood $\U\subset \PH_{c=1}(M)$ of $f$ satisfying that every $g\in \U$ admits a non-uniformly compact center foliation $\W^c_{g}$ such that $(f,\W^c)$ and $(g,\W^c_{g})$ are leaf-conjugate.
\end{prop}

Proposition \ref{propglobalsatbilitynonunifcompact} could potentially be useful for bringing into play perturbative techniques to the existence problem of non-uniformly compact center foliations with one-dimensional center.

\section{Unique integrability of the center bundle}

Suppose $f\in\PH_{c=1}(M)$. Since $\dim(E^c)=1$ it follows from Peano's existence theorem that through every point of $M$ there exists at least one local $C^1$ curve tangent to $E^c$. We say that $E^c$ is \emph{uniquely integrable} if through every point of $M$ there exists a unique $C^1$ local curve tangent to $E^c$ modulo reparametrizations. That is, if for every $\eta:(-\delta,\delta)\to M$ and $\gamma:(-\epsilon,\epsilon)\to M$ a pair of $C^1$ curves tangent to $E^c$ with $\eta(0)=\gamma(0)$ there exists $\delta'>0$ such that $\eta(-\delta',\delta')$ is a subset of  $\gamma(-\epsilon,\epsilon)$.

It turns out that unique integrability of the center bundle persists along whole connected components of discretized Anosov flows and of systems admitting a uniformly compact center foliation (at least for one-dimensional center). This is the content of Proposition \ref{propC} and Proposition \ref{propC'} stated in the introduction.

\begin{proof}[Proof of Proposition \ref{propC} and Proposition \ref{propC'}]
Suppose $f$ is a discretized Anosov flow such that $E^c_{f}$ is uniquely integrable. Let $\mathcal{V}(f)$ denote the connected component of $\PH_{c=1}(M)$ containing $f$. By Theorem \ref{thmA} every $g$ in $\mathcal{V}(f)$ is a discretized Anosov flow. 

Let $\mathcal{V}=\{g\in \PH_{c=1}(M)\mid E^c_g$ is uniquely integrable$\}$. By Proposition \ref{propstableuniqint} the set $\mathcal{V}$ is open and closed in $\PH_{c=1}(M)$. Since $\mathcal{V}(f)$ is connected it follows that either $\mathcal{V}(f)\cap \mathcal{V}$ or $\mathcal{V}(f)\cap (\PH_{c=1}(M)\setminus\mathcal{V})$ is empty. Since $f$ lies in $\mathcal{V}(f)\cap \mathcal{V}$ one obtains that the latter is empty, thus $\mathcal{V}(f)$ is a subset of $\mathcal{V}$.

In case $f$ was a system admitting a uniformly compact center foliation the argument is analogous using Theorem \ref{thmAprima} and Proposition \ref{propopenclosedunifcompact} in the place of Theorem \ref{thmA} and Proposition \ref{propopenclosed}.
\end{proof}

\begin{cor}\label{cortime1uniqint}
Every discretized Anosov flow in the same $C^1$ connected component of $\PH_{c=1}(M)$ as the time 1 map of an Anosov flow has a uniquely integrable center bundle.
\end{cor}
\begin{proof} Let $\varphi_t:M\to M$ be an Anosov flow. We can approximate $\frac{\partial \varphi_t}{\partial t}|_{t=0}$ by a $C^\infty$ vector field $X$ so that, if $X_t$ denotes the flow generated by $X$, then $f:=\varphi_1$ and $g:=X_1$ are $C^1$-close (in particular, such that they are in the same $C^1$ partially hyperbolic connected component). Since $g$ is a discretized Anosov flow with uniquely integrable center bundle (because $X$ is $C^\infty$) it follows that $f$ and every systems in the $C^1$ connected component of $\PH_{c=1}(M)$ containing $f$ has a uniquely integrable center bundle.
\end{proof}

In \cite{HHU} two types of partially hyperbolic diffeomorphisms in $\mathbb{T}^3$ are built. Ones which are non-dynamically coherent and ones which are dynamically coherent but such that $E^c$ is not uniquely integrable. The following sketches how a discretized Anosov flow with non-uniquely integrable center bundle can be obtained as a simple modification of the second type of examples.

\begin{example}\label{exnonuniqint} Let us start by giving a brief description of a dynamically coherent example from \cite{HHU}. For more details see \cite{HHU} itself.

The partially hyperbolic diffeomorphism $f:\mathbb{T}^3\to \mathbb{T}^3$ can be considered homotopic to $A\times \Id$, where $\mathbb{T}^3$ is identified with $\mathbb{T}^2\times S^1$ and $A:\mathbb{T}^2\to \mathbb{T}^2$ is a linear hyperbolic automorphism with eigenvalues $0<\lambda<1$ and $1/\lambda$.

Denote by $E^s_A$ the contracting eigenspace of $A$ and by $e_s$ a unit vector in $E^s_A$. And identify $S^1$ with $\R/2\mathbb{Z}$. Then the map $f$ can be taken to be of the form $$f(x,\theta)=(A x+v(\theta)e_s,\Psi(\theta))$$ for suitable $v:S^1\to \mathbb{R}$ and $\Psi:S^1\to S^1$ such that $v$ is positive in $(-1,0)\subset S^1$ and negative in $(0,1)\subset S^1$, and $\Psi$ is a Morse-Smale map with $-1$ and $0$ as only fixed points that in addition satisfy $\Psi'(0)<\lambda<1<\Psi'(-1)<1/\lambda$. 

The sets $\mathbb{T}^2\times \{-1\}$ and $\mathbb{T}^2\times \{0\}$ are two invariant tori that are fixed by $f$, with $f$ acting as $A$ on each of them. The torus $\mathbb{T}^2\times \{0\}$ is a $cu$-torus (it is saturated by $\W^c$ and $\W^u$-leaves) and the torus $\mathbb{T}^2\times \{-1\}$ is a repelling $su$-torus (it is saturated by $\W^s$ and $\W^u$-leaves).

The construction given by \cite{HHU} shows the following. The map $f$ admits an $f$-invariant foliation by circles $\W^c$. Each of these circles is homotopic to a circle of the form $\{x\}\times S^1$ and intersects in a unique point each torus $\mathbb{T}^2 \times \theta$ for every $\theta \in S^1$. Thus $f$ is a partially hyperbolic skew-product where $\W^c$ is a foliation by circles that gives to $M$ a structure of fiber bundle.

Moreover, it can be seen that the bundle $E^c$ is uniquely integrable outside of the $cu$-torus $\mathbb{T}^2\times \{0\}$. However, remarkably, through each point of $\mathbb{T}^2\times \{0\}$ there exists more than one local $C^1$ curve tangent to $E^c$. Namely, through each point $y$ of $\mathbb{T}^2\times \{0\}$ one can consider the center arc corresponding to the leaf $\W^c(y)$, but also all the center arcs that are a concatenation of a piece of arc of $\W^c$, a center arc through $y$ contained in the $cu$ torus $\mathbb{T}^2\times \{0\}$ and a third piece of $\W^c$ arc. See Figure \ref{fig4}.

\begin{figure}[htb]
\def\svgwidth{0.7\textwidth}
\begin{center} 
{\scriptsize
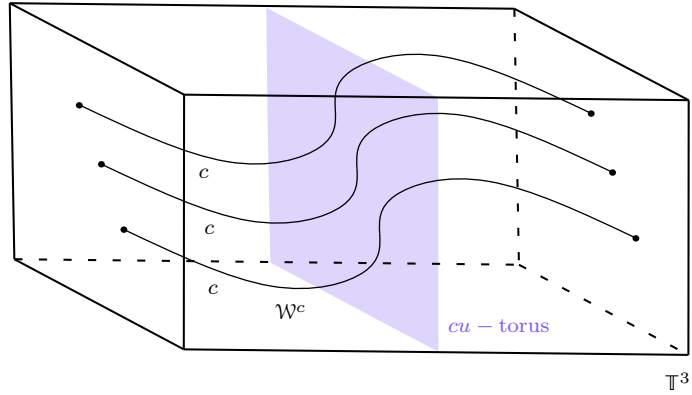
}
\caption{A cu-torus of points where $E^c$ is not uniquely integrable.}\label{fig4} 
\end{center}
\end{figure}

The simple modification of the example proceeds as follows. Let $F:\mathbb{T}^2\times \R\to \mathbb{T}^2\times \R$ be the lift of $f$ such that $F(x,-1)=(Ax,-1)$ for every $x\in \mathbb{T}^2$. It is immediate to check that $F$ commutes with the elements of the group $\Gamma=\{(x,\tilde{\theta})\mapsto (A^nx,\tilde{\theta}+2n)\}_{n\in \mathbb{Z}}$. 
As a consequence, $F$ descends to a partially hyperbolic diffeomorphism $g:N\to N$ in $N=(\mathbb{T}^2\times \R)/\Gamma$.

Let $\tilde{\W}^c$ denote the lift of $\W^c$ to $\mathbb{T}^2\times \R$ and $\W^c_g$ the descended one to $N$. It follows that $\W^c_g$ is a $g$-invariant center foliation for $g$. Since $F(\WW^c(x,-1))=\WW^c(Ax,-1)$ for every $x\in \mathbb{T}^2$ the leaves of $\W^c_g$ are individually fixed by $g$ (that is, $g(W)=W$ for every $W\in\W^c_g$). 

Moreover, for every $z\in N$ the point $g(z)$ lies in $\W^c_{g,L}(z)$ for $L>0$ any constant larger than the maximum length of a leaf in $\W^c$. By Proposition \ref{propcenterfixingL} it follows that $g$ is a discretized Anosov flow.

Finally, the property of non-unique integrability of the center bundle is preserved along the $cu-$torus that is the projection of $\mathbb{T}^2\times \{0\}$ to $N$ since this is a local property that is preserved by lifts and quotients. Hence $E^c_g$ is not uniquely integrable.
\end{example}

It turns out that unique integrability versus non-unique integrability of the center bundle provides a way for distinguishing between different connected components of discretized Anosov flows and partially hyperbolic systems in general. The following questions arise naturally.

\begin{quest}
Is it possible to connect (via a $C^1$-path of discretized Anosov flows) every discretized Anosov flow with uniquely integrable center bundle to the time 1 map of an Anosov flow? 
\end{quest}

\begin{quest}
Are there examples of discretized Anosov flows with a non uniquely integrable center bundle which are transitive or such that the center flow is not orbit equivalent to a suspension flow?
\end{quest}

More generally, we may ask:

\begin{quest} Are there examples of $C^1$-connected component of partially hyperbolic diffeomorphisms containing both systems with uniquely integrable and non-uniquely integrable center bundle?
\end{quest}

\section{Fixed and compact center foliations in dimension 3: Proof of Theorem \ref{thmD}}\label{sectionfixedcenterindim3}

The goal of this section is to prove Theorem \ref{thmD} stated in the introduction. In particular, item (\ref{item1propcenterfixingdim3}) of Theorem \ref{thmD}  will expand further on the discussion initiated in Subsection \ref{sectionfixedcenterfoliation} as it shows that in dimension 3 the transitive partially hyperbolic diffeomorphisms that are discretized Anosov flows coincide with those that leave invariant each leaf of a center foliation.

\begin{remark}\label{rmkcenterfixing3ifdyncohunifcmpct} When $f$ is dynamically coherent, item (\ref{item2propcenterfixingdim3}) of Theorem \ref{thmD} follows from previous results on compact center foliations. Indeed, by \cite{DMM} (also \cite{Gog} in case $E^c$ is uniquely integrable) the center foliation $\W^c$ needs to be uniformly compact. Then by \cite{Boh} one concludes that, modulo double cover, $(f,\W^c)$ is a partially hyperbolic skew product.

Alternatively, one could try to use \cite[Theorem 1]{BW}. We just chose the path described above so that the approach for both results (items (\ref{item1propcenterfixingdim3}) and (\ref{item2propcenterfixingdim3}))  matches.
\end{remark}

The strategy for showing Theorem \ref{thmD} is to first show that the hypothesis imply dynamical coherence (Proposition \ref{propdim3dyncoh} below). Then item (\ref{item2propcenterfixingdim3}) follows as explained in Remark \ref{rmkcenterfixing3ifdyncohunifcmpct} above and item (\ref{item1propcenterfixingdim3}) will follow as a consequence of \cite[Theorem 2]{BW}.

\subsection{Dynamical coherence}

Recall that $f$ is called \emph{transitive} if it has a dense orbit. Recall that the the \emph{non-wandering set} of $f$, denoted by $\Omega(f)$, is the set of all $x$ in $M$ such that for every  neighborhood $U$ of $x$ there exists $k>0$ satisfying $f^k(U)\cap U \neq \emptyset$. It is immediate to check that if $f$ is transitive then $\Omega(f)=M$.

In the context of Theorem \ref{thmD} we obtain that dynamical coherence can be derived from the more general hypothesis `$\Omega(f)=M$' in the place of `$f$ transitive':

\begin{prop}\label{propdim3dyncoh}
Suppose $f\in \PH_{c=1}(M^3)$ with $\Omega(f)=M^3$ admits an invariant center foliation $\W^c$ satisfying one of the following conditions: \begin{enumerate} \item\label{itemipropdim3dyncoh} $f(W)=W$ for every $W\in \W^c$.
\item\label{itemiipropdim3dyncoh} $W$ is compact for every $W\in \W^c$.\end{enumerate} Then $f$ is dynamically coherent with invariant foliations $\W^{cs}$ and $\W^{cu}$ such  that $\W^c=\W^{cs}\cap \W^{cu}$.
\end{prop}

The proof of Proposition \ref{propdim3dyncoh} consists of the following two lemmas.

\begin{lemma}\label{lemmadyncohdriterion} Suppose $f\in \PH_{c=1}(M^3)$ admits an $f$-invariant center foliation $\W^c$. Suppose that the set $\{W\in \W^c \mid W\mbox{ compact and } f^n(W)=W \mbox{ for some }n\neq 0\}$ is dense in $M$. Then $f$ is dynamically coherent and admits $f$-invariant foliations $\W^{cs}$ and $\W^{cu}$ such that $\W^c=\W^{cs}\cap\W^{cu}$.
\end{lemma}
\begin{proof}

We first claim that it is enough to show that there exists $\delta>0$ so that for every $x$ and $y$ in $M$, if $y\in \W^s_\delta(x)$ then $\W^c_\delta(y)\subset \W^s_{2\delta}(\W^c_{2\delta}(x))$. And if $y\in \W^u_\delta(x)$ then $\W^c_\delta(y)\subset \W^u_{2\delta}(\W^c_{2\delta}(x))$.

Indeed, suppose such a $\delta$ exists. For every $x\in M$ one can define $\W^{cs}(x)$ as the set of all points in $M$ that can be joined to $x$ by a finite concatenation of $\W^s$ and $\W^c$ arcs. In this way $\{\W^{cs}(x):x\in M\}$ defines a partition of $M$. On each element of this partition one can consider the distance $d(y,z):=\inf_\gamma \length(\gamma)$ where $\gamma$ varies among all finite concatenations of $\W^s$ and $\W^c$ arcs joining $y$ to $z$. 

By shrinking $\delta$, if necessary, one can ensure by Lemma \ref{lemmaWcWsC1} that the set $\W^s_{2\delta}(\W^c_{2\delta}(x))$ is a $C^1$ submanifold tangent to $E^s\oplus E^c$ for every $x\in M$. As $\W^c_\delta(y)$ is contained in  $\W^s_{2\delta}(\W^c_{2\delta}(x))$ for every $y\in \W^s_\delta(x)$ it follows that for some $\epsilon, \epsilon'>0$ independent of $x$ the ball $B_\epsilon(x)\subset \W^{cs}(x)$ with respect to $d$ is an open subset of the $C^1$ submanifold $\W^s_{2\delta}(\W^c_{2\delta}(x))$ that contains the ball $B_{\epsilon'}(x)\subset \W^s_{2\delta}(\W^c_{2\delta}(x))$ with respect to the inner metric in $\W^s_{2\delta}(\W^c_{2\delta}(x))$ induced by the Riemannian metric of $M$.

In this way one obtains that each element of $\W^{cs}$ is a $C^1$ submanifold tangent to $E^s\oplus E^c$, saturated by $\W^s$ and $\W^c$ leaves and whose inner metric is complete. Hence $\W^{cs}$ is an $f$-invariant center-stable foliation. Analogously one constructs $\W^{cu}$ an $f$-invariant center-unstable foliation. The property $\W^c=\W^{cs}\cap \W^{cu}$ follows immediately. This proves the claim.
 
It remains to show that there exists $\delta>0$ such that for every $x,y\in M$ with $y \in \W^s_\delta(x)$, one has $\W^c_\delta(y)\subset \W^s_{2\delta}(\W^c_{2\delta}(\W^c_{2\delta}(x))$. For $cu$ discs the arguments are analogous.

The key point to note is that two distinct leaves of $\W^c$ that are compact and periodic can not intersect the same leaf of $\W^s$. Indeed, suppose by contradiction that two such leaves $W,W'\in \W^c$ contain points $x\in W$ and $y\in W'$ that belong to the same leaf of $\W^s$. One can consider $N>0$, a multiple of the periods of $W$ and $W'$, so that $f^N(W)=W$ and $f^{N}(W')=W'$. On the one hand, $d\big(f^{kN}(x),f^{kN}(y)\big)$ tends to $0$ as $k\to +\infty$ because $x$ and $y$ belong to the same stable leaf. On the other hand, $f^{kN}(x)$ lies in $W$ and $f^{kN}(y)$ in $W'$ so for every $k>0$ the distance between both points can not be smaller than the positive distance between the disjoint compact sets $W$ and $W'$. This gives us a contradiction.

Consider from now on a metric in $M$ and $\delta>0$ small enough so that the bundles $E^s$, $E^c$ and $E^u$ are almost constant and pairwise orthogonal at scale $\delta$. For a precise construction of such a metric and constant see for example Lemma \ref{rmkUdelta}. And consider $\delta>0$ small enough so that by Lemma \ref{lemmaWcWsC1} the set $\W^u_{4\delta}(\W^c_{4\delta}(x))$ is a $C^1$ submanifold tangent to $E^c\oplus E^u$ for every $x\in M$. In particular, let us consider the above so that for every $x,y\in M$ such that $d(x,y)<3\delta$ the set $\W^s_{4\delta}(y)$ intersects $\W^u_{4\delta}(\W^c_{4\delta}(x))$ and this intersection point is unique.

For every $y$ such that $d(x,y)<3\delta$ let $\pi^s_x(y)$ denote the intersection of $\W^s_{4\delta}(y)$ with $\W^u_{4\delta}(\W^c_{4\delta}(x))$. It is immediate to check that $\pi^s_x(y)$ varies continuously with $y$. For every $x\in M$ let $D(x)$ denote $\W^u_{2\delta}(\W^c_{2\delta}(x))$. By Lemma \ref{lemmaWcWsC1} it is a $C^1$ disc tangent to $E^c\oplus E^u$ for every $x\in M$. The set $D(x)\setminus \W^c_{2\delta}(x)$ has two connected components. Let us denote them by $D^+(x)$ and $D^-(x)$.

Suppose by contradiction that there exists 	$x_0,y_0\in M$  such that $y_0$ lies in $\W^s_\delta(x_0)$ and $\W^c_\delta(y_0)$ is not contained in $\W^s_{2\delta}(\W^c_{2\delta}(x_0))$. Then there exists $z_0\in \W^c_\delta(y_0)$ such that $\pi^s_{x_0}(z_0)$ is not in $\W^c_{2\delta}(x_0)$. By the hypothesis of almost constant and pairwise orthogonal invariant bundles at scale $\delta$ the point $\pi^s_{x_0}(z_0)$ lies in $D(x_0)$. Suppose without loss of generality that $\pi^s_{x_0}(z_0)$ lies in $D^+(x_0)$. See Figure \ref{fig2}.	

\begin{figure}[htb]
\def\svgwidth{0.7\textwidth}
\begin{center} 
{\scriptsize
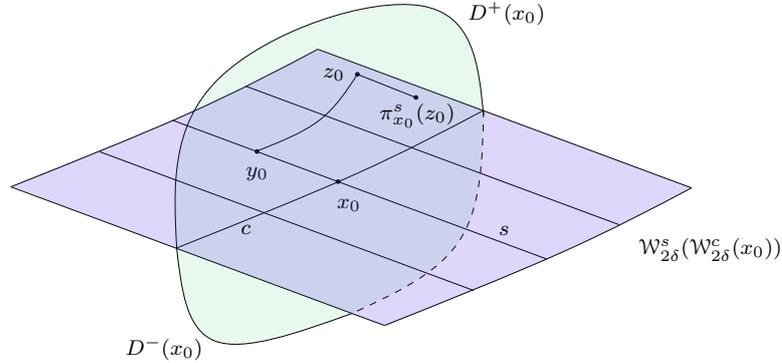
}
\caption{Projection by stable holonomy of the point $z_0$ lying in $\W^c_{loc}(y_0)$ but not in $\W^s_{loc}(\W^c_{loc}(x_0))$, for some $y_0$ in $\W^s_{loc}(\W^c_{loc}(x_0))$.}\label{fig2} 
\end{center}
\end{figure}

On the one hand, there exists $\epsilon>0$ small so that $\pi^s_{x_0}(B_\epsilon(z_0))$ is entirely contained in $D^+(x_0)$. On the other hand, since $\pi^s_{x_0}(y_0)=x_0$ one can consider $y_1$ as close as wanted to $y_0$ so that $\pi^s_{x_0}(y_1)$ lies in $D^-(x_0)$ and $\W^c_{\delta}(y_1)$ intersects $B_\epsilon(z_0)$. In particular, for such a $y_1$ there exists an arc $\gamma\subset \W^c_\delta(y_1)$ joining $y_1$ with a point $z_1\in B_\epsilon(z_0)$.

As $\{W\in \W^c \mid W\mbox{ compact and } f^n(W)=W \mbox{ for some }n\neq 0\}$ is dense in $M$ we can approximate $\gamma$ and $\W^c_{2\delta}(x_0)$ by center arcs contained in compact periodic leaves of $\W^c$. 

By construction $\pi^s_{x_0}(\gamma)$ is an arc in $D(x_0)$ joining a point in $D^+(x_0)$ with a point in $D^-(x_0)$. In particular, $\pi^s_{x_0}(\gamma)$ intersects $\W^c_{2\delta}(x_0)$. One can then approximate $\gamma$ by an arc $\gamma'$ contained in a periodic compact leaf of $\W^c$ so that the $\pi^s_{x_0}(\gamma')$ continues to satisfy the same property, namely that $\pi^s_{x_0}(\gamma')$ intersects $\W^c_{2\delta}(x_0)$ and has each of its endpoints in a different connected component of $D(x_0)\setminus \W^c_{2\delta}(x_0)$.

By approximating $\W^c_{2\delta}(x_0)$ close enough by a center arc $\eta$ contained in a periodic compact leaf of $\W^c$ one obtains that $\pi^s_{x_0}(\gamma')$ and $\pi^s_{x_0}(\eta)$ must intersect. This gives us a contraction with the aforementioned fact that one can not join two different compact periodic leaves of $\W^c$ by an arc contained in a leaf of $\W^s$.
\end{proof}

The criterion above combined with the following lemma ends the proof of Proposition \ref{propdim3dyncoh}. 

\begin{lemma}\label{lemmaperiodiccenters} In the setting of Proposition \ref{propdim3dyncoh} the set $\{W\in \W^c \mid W\mbox{ compact}$ $\mbox{and } f^n(W)=W \mbox{ for some }n\neq 0\}$ is dense in $M$.
\end{lemma}
\begin{proof} 

To avoid possible non-orientation preserving issues let us work with $g:=f^2$ and show that $\{W\in \W^c \mid W\mbox{ compact}$ $\mbox{and } g^n(W)=W \mbox{ for some }n\neq 0\}$ is dense in $M$. 

Recall that, since $\Omega(f)=M$, it follows by standard arguments that $\Omega(g)=M$. Indeed, for every open set $U\subset M$ there exists $k_1,k_2>0$ such that $V:=f^{k_1}(U)\cap U$  and $f^{k_2}(V)\cap V\neq \emptyset$ are non empty. It follows that at least one element $k$ in $\{k_1,k_2,k_1+k_2\}$ is even, and then $g^{
k/2}(U)\cap U$ is non empty. One concludes that $\Omega(g)=M$.


Note that the set of fixed points of $g$, denoted by $\Fix(g)\subset M$, has empty interior in $M$. This follows immediately from the fact that, if $x$ is a fixed point of $g$, then every $y$ in $\W^s_{loc}(x)\setminus\{x\}$ can not be a fixed point of $g$ because its forward $g$-orbit must tend to $x$.

As a consequence of $\Fix(g)$ having empty interior in $M$ it is enough to show that $\{W\in \W^c \mid W\mbox{ compact}$ $\mbox{and } g^n(W)=W \mbox{ for some }n\neq 0\}$ is dense in $M\setminus \Fix(g)$ to obtain that is dense in $M$.

Suppose from now on that $x_0$ is a point in $M\setminus \Fix(g)$. Let us see that for every $\epsilon>0$ small enough there exists $x\in B_\epsilon(x_0)$ and $k>0$ such that $g(\W^c_\epsilon(x))$ is disjoint from  $\W^c_\epsilon(x)$ and $g^k(x)\in \W^c_\epsilon(x)$. This immediately implies that $\W^c(x)$ needs to be compact and periodic (see next paragraph) and shows that $x_0$ can be approximated by periodic compact leaves of $\W^c$.

Indeed, in case that every leaf of $\W^c$ is compact then $\W^c(x)$ is automatically compact and periodic for $g$ and there is no more to say. In case $f(W)=W$ for every leaf $W\in\W^c$ let us suppose by contradiction that $\W^c(x)$ is not compact. Then $f:\W^c(x)\to \W^c(x)$ is a homeomorphism of the line. As a consequence $g:\W^c(x)\to \W^c(x)$ is an orientation preserving homeomorphism of the line. Then $g(\W^c_\epsilon(x))$ disjoint from $\W^c_\epsilon(x)$ impedes $g^k(x)$ from lying in $\W^c_\epsilon(x)$ for some $k>0$ and gives us a contradiction.

Let $\epsilon>0$ be small enough so that $g(B_{2\epsilon}(x_0))$ is disjoint from $B_{2\epsilon}(x_0)$. And small enough so that at scale $\epsilon$ the bundles are almost constant and the distances inside the invariant manifolds $\W^\sigma$, $\sigma\in\{s,c,u\}$ are nearly the same as in the manifold. For a precise construction see for example the scale and metric considered in property (P4) of Lemma \ref{rmkUdelta}. 

Inside $B_\epsilon(x_0)$ let $U$ be a $\W^c$-foliation box neighborhood containing $x_0$ that is obtained as $U:=\W^c_\delta(D)$ for $\delta>0$ some small constant and $D$ some $C^1$ disc transverse to $\W^c$ and nearly tangent to $E^s\oplus E^u$. In particular, let $\delta>0$ be such that $\delta/2$ is smaller than the constant given by Lemma \ref{lemmaWcWsC1}. 

Let us consider $0<\delta'<\delta$ and $0<\epsilon'<\epsilon$ such that $10\epsilon'<\delta'$ and such that for every $y\in B_{\epsilon'}(x_0)$ the set $\W^s_{\delta'}(\W^u_{\delta'}(y))$ is contained in $U$. 

We claim that for every $y\in B_{\epsilon'}(x_0)$ the set $\W^s_{\delta'}(\W^u_{\delta'}(y))$ intersects every center plaque of $U$ in at most one point. This is a consequence of Lemma \ref{lemmaWcWsC1}. Indeed, suppose that $w,w'\in \W^s_{\delta'}(\W^u_{\delta'}(y))$ are points in the same center plaque of $U$. Then $w'\in\W^c_\delta(w)$. Let $z,z'\in \W^u_{\delta'}(y)$ be such that $w\in\W^s_{\delta'}(z)$ and $w'\in\W^s_{\delta'}(z')$. As $w'\in \W^c_{\delta'}(w)$ then both $z$ and $z'$ lie in $\W^s_{\delta'}(\W^c_{\delta'}(w))$. As $\W^s_{\delta'}(\W^c_{\delta'}(w))$ is $C^1$ and tangent to $E^s\oplus E^c$ it follows that $\W^s_{\delta'}(\W^c_{\delta'}(w))$ intersects $\W^u_{\delta'}(y)$ in at most one point. That is, $z=z'$. Then $w=w'$. This proves the claim.

Let $\pi^c:U\to D$ denote the projection along center plaques. It is immediate to check that $\pi^c$ needs to be continuous. The previous paragraph then implies that $\pi^c$ from $\W^s_{\delta'}(\W^u_{\delta'}(y))$ to $D$ is a homeomorphism onto its image for every $y\in B_{\epsilon'}(x_0)$. 

Since $x_0\in \Omega(g)$ there exists $k>0$ such that $g^k(B_{\epsilon'}(x_0))\cap B_{\epsilon'}(x_0)\neq \emptyset$. Moreover, such a $k$ can be considered arbitrarly large. Let us fix such a $k$ large enough so that $\W^u_{2\delta'}(g^k(x))\subset g^k(\W^u_{\delta'}(x))$ and $g^k(\W^s_{\delta'}(x))\subset \W^s_{\delta'/2}(g^k(x))$ for every $x\in M$.

Let us fix from now on $y$ a point in $B_{\epsilon'}(x_0)$ so that $g^k(y)\in B_{\epsilon'}(x_0)$. It follows that there exists a sub arc $\gamma^u_y$ in $\W^u_{\delta'}(y)$ such that $g^k(\gamma^u_y)=\W^u_{2\delta'}(g^k(y))$.

Then $g^k(\W^s_{\delta'}(\gamma^u_y))$ is a subset of $\W^s_{\delta'/2}(\W^u_{2\delta'}(g^k(y))$. Consider $R\subset D$ the closure of the image by $\pi^c$ of $\W^s_{\delta'}(\gamma^u_y))$. It follows that $R$ is a topological disc in $D$. Its boundary can be viewed as a rectangle. Two of its opposite sides, $\Gamma_2$ and $\Gamma_4$, correspond to the projection by $\pi^c$ of the two $s$-arcs $\W^s_{\delta'/2}(y_1)$ and $\W^s_{\delta'/2}(y_2)$ for $y_1,y_2$ each one of the two endpoints of $\gamma^u_y$ in $\W^u(y)$. The other two sides, $\Gamma_1$ and $\Gamma_3$, correspond to the projection by $\pi^c$ of the two segments formed by the endpoints of $\W^s_{\delta'}(y')$ as $y'$ varies in $\gamma^u_y$.

It follows that $h=\pi^c\circ g^k\circ (\pi^c)^{-1}$ is a well defined continuous map from $R$ to $D$. We claim that it is enough to show that $h$ has a fixed point. 
Indeed, if $p$ denotes a fixed point for $h$ then $o:=(\pi^c)^{-1}(p)$ satisfies that $g^k(o)$ and $o$ are in the same center plaque of $U$. Since $g(\W^c_\epsilon(x))$ is disjoint from  $\W^c_\epsilon(x)$, because $g(B_{2\epsilon}(x_0))$ is disjoint from $B_{2\epsilon}(x_0)$, one obtains that $\W^c(o)$ needs to be compact and periodic as discussed before.

The existence of a fixed point for $h$ follows by a classic Lefschetz's index argument. Let $\Gamma$ denote the boundary of $R$. The closed curve $\Gamma$ is the union $\Gamma_1\cup\ldots \cup \Gamma_4$ of the sides of $R$ as explained above. Since $g^k(\W^s_{\delta'}(\gamma^u_y))$ is a subset of $\W^s_{\delta'/2}(\W^u_{2\delta'}(g^k(y))$ and $d(y,g^k(y))<\epsilon'$ for $10\epsilon'<\delta'$ it follows that $h$ sends the rectangle $R$ to a new rectangle $h(R)$
that `crosses' $R$ so that $\Gamma_1$ and $\Gamma_3$ do not intersect $h(R)$ and $h(\Gamma_2)$ and $h(\Gamma_4)$ lie in two different connected components of $R\setminus h(R)$ that are adjacent to $\Gamma_2$ and $\Gamma_4$, respectively (see Figure \ref{fig0} and Figure \ref{fig1}). This is enough for finding a fixed point for $h$. For the sake of completeness we will reproduce this classical argument for finding a fixed point under these hypothesis.

\begin{figure}[htb]
\def\svgwidth{0.8\textwidth}
\begin{center} 
{\scriptsize
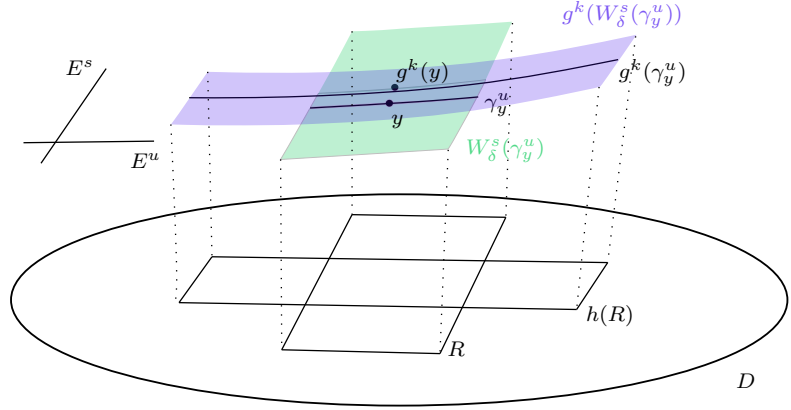
}
\caption{Modulo local center projection, a $su$ rectangle that comes back transverse to itself.}\label{fig0} 
\end{center}
\end{figure}
	
\begin{figure}[htb]
\def\svgwidth{0.6\textwidth}
\begin{center} 
{\scriptsize
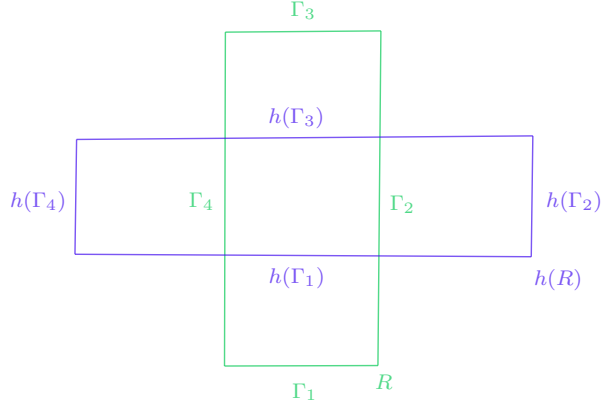
}
\caption{The transverse $su$ return ensures a fixed point modulo center projection, hence a point in $M$ that is sent to its same local center leaf.}\label{fig1} 
\end{center}
\end{figure}

Let $t\mapsto \Gamma(t)$ be an homeomorphism from the circle $S^1$ to $\Gamma$. We can consider a nullhomotopy $\{\Gamma^{(s)}\}_{s\in[0,1]}$ of $\Gamma$ inside $R$ as follows. Let us identify $R$ homeomorphically with $[0,1]
\times [0,1]$. Then let $\Gamma^{(0)}$ be equal to $\Gamma$ and let $\Gamma^{(s)}$, varying continuously with $s\in [0,1]$, be such that the image of $\Gamma^{(s)}$ is the boundary of the square $[0,1-s]\times [0,1-s]$. 

Suppose by contradiction that $h$ has no fixed points in $R$. Let us identify $D$ with the euclidean plane $\R^2$. The no fixed points assumption implies that the continuous family of maps $\rho_s:S^1\to S^1$ given by $$\rho_s(t):= \frac{h(\Gamma^{(s)}(t))-\Gamma^{(s)}(t)}{||h(\Gamma^{(s)}(t))-\Gamma^{(s)}(t)||}$$ is well defined for every $s\in [0,1]$. 

On the one hand, from the way the sides $\Gamma_1$, \ldots, $\Gamma_4$ are mapped by $h$ it is an immediate computation to check that $\rho_0:S^1\to S^1$ has index different from $0$. On the other hand, if $y_0$ denote the point that is the image of $\Gamma^{(1)}$, then $h(y_0)\neq y_0$ and one can consider a small ball $B$ containing $y_0$ so that $h(B)\cap B=\emptyset$. This immediately implies that for every $s$ close enough to $0$ (so that $\Gamma^{(s)}\subset B$) the map $\rho_s:S^1\to S^1$ must have index $0$. As the index of a continuous family of maps from $S^1$ to $S^1$ is an invariant of the family one gets to a contradiction. Hence $h$ must have a fixed point on $R$.
\end{proof}

\subsection{Proof of Theorem \ref{thmD}}

As was already mentioned in Remark \ref{rmkcenterfixing3ifdyncohunifcmpct}, item (\ref{item2propcenterfixingdim3}) of Theorem \ref{thmD} has been already shown once dynamical coherence was proven. It remains to show item (\ref{item1propcenterfixingdim3}).

Suppose $f\in \PH_{c=1}(M^3)$ satisfies the hypothesis of (\ref{item1propcenterfixingdim3}). Let $\W^c$ denote the center foliation such that $f(W)=W$ for every $W\in \W^c$. By Lemma \ref{lemmaperiodiccenters} there exists at least one compact leaf $\gamma$ of $\W^c$ (in fact, the union of such leaves is dense in $M$). Moreover, for every $x\in \W^s_{loc}(\gamma)$ the leaf $\W^c(x)$ is fixed by $f$ (in particular, periodic). It follows from \cite[Theorem 2]{BW} (see Remark \ref{rmkBW05}) that $f^n$ is a discretized Anosov flow for some $n>0$.

The next claim concludes then the proof of Theorem \ref{thmD}.

\begin{af}
In the setting of Theorem \ref{thmD} item (\ref{item1propcenterfixingdim3}), if $f^n$ is a discretized Anosov flow for some $n>1$ then $f$ is a discretized Anosov flow.
\end{af}
\begin{proof}
Note first that, by Proposition \ref{prop1DAFs} item \ref{item1prop1}, if $W$ is a leaf of $\W^c$ that is not compact, then $f^n$ has no fixed points in $W$. As a consequence, $f$ has no fixed points in $W$ either. In particular, $f$ has to preserve the orientation of $W$.

By Proposition \ref{propcenterfixingL} there exists $L>0$ such that $f^n(x)\in \W^c_L(x)$ for every $x\in M$. If $W\in \W^c$ is not compact it follows from the paragraph above that $f(x)\in W^c_L(x)$ for every $x$ in $W$.

Given $K>0$, by transverse hyperbolicity one can show that every leaf of $\W^c$ with length less than $K$ can not be accumulated by compact leaves of $\W^c$ with length less than $K$. As a consequence, there exist at most countably many compact leaves of $\W^c$.

Given $x$ in a compact leaf $W\in \W^c$ one can consider a sequence $x_n$ converging to $x$ so that $\W^c(x_n)$ is not compact for every $n$. As $f(x_n)$ belongs to $\W^c_L(x_n)$ for every $n$ and the sequence $f(x_n)$ tends to $f(x)$ one obtains that $f(x)$ must lie in $\W^c_L(x)$. 

We have shown that $f(x)\in \W^c_L(x)$ for every $x\in M$. By Proposition \ref{propcenterfixingL} we conclude that $f$ needs to be a discretized Anosov flow.
\end{proof}

{\small \emph{Acknowledgments:} I have deeply benefited from countless exchanges with my advisors S. Crovisier and R. Potrie. I would also like to thank A. Thazhibi, T. Barbot (thesis reviewers) and the anonymous referees of this paper.}

\end{document}